\PassOptionsToPackage{dvipsnames}{xcolor}
\documentclass{amsart}

\usepackage{zref-clever}
\usepackage{relsize}
\usepackage{amsmath}
\usepackage{amssymb}
\usepackage{amsthm}
\usepackage{adjustbox}
\usepackage{mathrsfs}
\usepackage{subcaption}

\usepackage{amsfonts}%
\usepackage{float}
\usepackage{paralist}
\usepackage{mathtools}
\usepackage{pgfplots}
\usepgfplotslibrary{polar}
\usepackage{xparse}
\usepackage{rotating}
\usepackage[makeroom]{cancel}
\usepackage{euscript}

\allowdisplaybreaks

\usepackage[normalem]{ulem}

\makeatletter
\newcommand*{\saved@uline}{}
\let\saved@uline\uline

\newcommand*{\mathuline}{%
  \mathpalette{\math@uline\saved@uline}%
}
\newcommand*{\math@uline}[3]{%
  \mbox{#1{$#2#3\m@th$}}%
}

\renewcommand*{\uline}{%
  \relax
  \ifmmode
  \expandafter\mathuline%
  \else
  \expandafter\saved@uline%
  \fi
}
\makeatother

\usepackage{enumitem}
\setlist[itemize,1]{label=--}
\setlist[itemize,2]{label=+}
\setlist[itemize,3]{label=$\bullet$}
\setlist[itemize,4]{label=$\circ$}

\setlist[enumerate,1]{label=(\roman*)}

\setlist{nosep}

\usepackage[
  style=alphabetic-verb,
  backref=true,
  backend=biber,
  hyperref=true,
  giveninits=true,
  opcittracker=true,
  maxbibnames=99,
  maxalphanames=4
]{biblatex}

\AtEveryBibitem{
  \clearfield{urlyear}
  \clearfield{urlmonth}
}

\AtBeginBibliography{\small}

\renewbibmacro*{pageref}{%
  \iflistundef{pageref}
  {}
  {%
\printtext{\addperiod\space $\uparrow${\printlist[pageref][-\value{listtotal}]{pageref}}}}}

\usepackage[dvipsnames]{xcolor}
\definecolor{cite}{HTML}{11871E}
\definecolor{url}{HTML}{698996}
\definecolor{link}{HTML}{912F1B}
\usepackage[pdfencoding=unicode, colorlinks=true, linkcolor=link, citecolor=cite, urlcolor=url, linktocpage]{hyperref}
\usepackage[hyphenbreaks]{breakurl}
\usepackage{xurl}
\hypersetup{breaklinks=true}
\hypersetup{final}

\usepackage[protrusion=true, expansion=true]{microtype}

\usepackage{iftex}
\ifLuaTeX
\usepackage{fontspec}
\usepackage[math-style=TeX]{unicode-math}

\directlua{
  fonts.handlers.otf.addfeature {
    name = "th_ligature",
    type = "ligature",
    data = {
      ["T_h"] = { "T", "h" }
    }
  }
}

\DeclareMathAlphabet{\eur}{U}{zeus}{m}{n}
\renewcommand{\mathcal}[1]{\eur{#1}}

\else
\usepackage[T1]{fontenc}
\usepackage[osf]{XCharter}
\usepackage[xcharter, scaled=1.07]{newtxmath}

\usepackage[
  cal=euler,%
  scr=euler,
  bb=esstix
]{mathalfa}
\fi

\usepackage{biolinum}

\renewcommand{\mathsf}[1]{\text{\normalfont\sffamily#1}}

\usepackage[
  letterpaper,
  twoside=true,
  textheight=22cm,
  textwidth=15cm,
  marginparsep=0.75cm,
  marginparwidth=2.5cm,
  heightrounded,
  centering
]{geometry}

\usepackage{pgf,tikz,nicematrix}
\usepackage{tikz-3dplot}
\usepackage{tkz-euclide}
\usepackage{wasysym}
\usepackage{tqft}

\newcommand{\drawGrid}[2]{%
  \pgfmathsetmacro\xbound{#1-1}
  \pgfmathsetmacro\ybound{#2-1}
  \foreach \yn in {0, ..., \ybound}{
    \foreach \xm in {0, ..., \xbound}{
      \path (\xm,-\yn) node (\yn\xm) {};
      \draw[fill=black] (\xm,-\yn) circle (0.04);
    }
  }
}

\NewDocumentCommand{\drawLine}{mmO{0.6}}{
  \draw[thickPath, double distance=#1,double=#2,draw=#2,opacity=#3]
}

\tikzstyle{thickPath} = [line cap=round,rounded corners=0.001pt]

\tikzset{toprule/.style={%
    execute at end cell={%
      \draw [line cap=rect,#1] (\tikzmatrixname-\the\pgfmatrixcurrentrow-\the\pgfmatrixcurrentcolumn.north west) -- (\tikzmatrixname-\the\pgfmatrixcurrentrow-\the\pgfmatrixcurrentcolumn.north east);%
    }
  },
  bottomrule/.style={%
    execute at end cell={%
      \draw [line cap=rect,#1] (\tikzmatrixname-\the\pgfmatrixcurrentrow-\the\pgfmatrixcurrentcolumn.south west) -- (\tikzmatrixname-\the\pgfmatrixcurrentrow-\the\pgfmatrixcurrentcolumn.south east);%
    }
  }
}

\usetikzlibrary[patterns]
\usetikzlibrary{3d, shapes.geometric}
\usetikzlibrary{arrows}
\usetikzlibrary{calc}
\usetikzlibrary{cd}
\usetikzlibrary{decorations.markings}
\usetikzlibrary{decorations.pathmorphing}
\usetikzlibrary{decorations.pathreplacing}
\usetikzlibrary{decorations.text}
\usetikzlibrary{knots}
\usetikzlibrary{matrix}
\usetikzlibrary{shapes.misc}
\usetikzlibrary{arrows, arrows.meta, positioning, calc}
\usetikzlibrary{backgrounds}
\tikzcdset{arrow style=tikz, diagrams={>={Straight Barb[scale=0.8]}}}
\tikzcdset{shift left/.default=0.38ex, shift right/.default=0.38ex}
\tikzset{>={Straight Barb[scale=0.8]}}

\tikzstyle{arrow} = [-{Straight Barb[scale=0.8]}, line width=0.4mm]

\makeatletter
\AddToHook{cmd/@makefntext/before}[zref-clever-footnotes]{%
  \zcsetup{currentcounter=\@mpfn}%
}
\makeatother

\NewDocumentCommand{\newzctheorem}{mO{subsubsection}mO{#3s}}{
  \newtheorem{#1}[#2]{#3}
  \AddToHook{env/#1/begin}{%
    \zcsetup{countertype={
      #2=#1}
    }
  }
  \zcRefTypeSetup{#1}{
    Name-sg= #3 ,
    Name-pl= #4
  }
}

\newcounter{introThmCounter}

\zcsetup{cap=true,
  nameinlink=false,
  lastsep = {,\space and\space}
}

\zcRefTypeSetup{section}{
  Name-sg={\S},
  name-sg={\S},
  Name-pl={\S\S},
  name-pl={\S\S},
}

\zcRefTypeSetup{equation}{
  Name-sg=,
  Name-pl=
}

\zcRefTypeSetup{item}{
  Name-sg=,
  Name-pl=
}

\newzctheorem{introThm}[introThmCounter]{Theorem}

\NewDocumentCommand{\setupTheoremEnv}{O{subsubsection}}{
  \newzctheorem{thm}[#1]{Theorem}
  \newzctheorem{lem}[#1]{Lemma}
  \newzctheorem{prop}[#1]{Proposition}
  \newzctheorem{cor}[#1]{Corollary}[Corollaries]

  \theoremstyle{definition}
  \newzctheorem{const}[#1]{Construction}
  \newzctheorem{defn}[#1]{Definition}
  \newzctheorem{ex}[#1]{Example}
  \newzctheorem{exs}[#1]{Examples}[Examples]
  \newzctheorem{nex}[#1]{Non-example}
  \newzctheorem{nexs}[#1]{Non-examples}[Non-examples]
  \newzctheorem{ntt}[#1]{Notation}
  \newzctheorem{rem}[#1]{Remark}
  \newzctheorem{rex}[#1]{Running Example}
  \newzctheorem{rmk}[#1]{Remark}
  \newzctheorem{conv}[#1]{Convention}
}

\setupTheoremEnv[subsubsection]

\def\op{{\mathsf{op}}}

\def\adq{\mathsf{adq}}

\def\bb{\mathbb}

\def\bP{\bb{P}}

\DeclareMathOperator{\constant}{\mathsf{const}}

\def\cart{\mathsf{cart}}

\def\Cat{\sf{Cat}}

\def\cC{\scr{C}}

\def\Corr{\sf{Corr}}

\def\dec{\mathsf{dec}}
\def\En{\mathsf{E}}

\def\eq{\mathsf{eq}}
\def\Ex{\mathsf{Ex}}

\DeclareMathOperator{\fr}{\sf{fr}}
\DeclareMathOperator{\Fun}{\mathsf{Fun}}
\DeclareMathOperator{\fv}{\mathsf{fv}}

\DeclareMathOperator{\Hom}{\mathsf{Hom}}

\DeclareMathOperator{\id}{\mathsf{id}}

\DeclareMathOperator{\Irr}{\mathsf{Irr}}

\newcommand{\lax}{\mathsf{lax}}
\def\lcart{\mathsf{lcart}}
\def\lex{\mathsf{lex}}

\DeclareMathOperator{\lr}{\mathsf{lr}}
\DeclareMathOperator{\lv}{\mathsf{lv}}
\def\main{\mathsf{main}}
\DeclareMathOperator{\Map}{\mathsf{Map}}
\let\max\undefined
\DeclareMathOperator{\max}{\mathsf{max}}
\def\sfmid{\mathsf{mid}}
\let\min\undefined
\DeclareMathOperator{\min}{\mathsf{min}}
\DeclareMathOperator{\mix}{\mathsf{mix}}
\def\msSet{\sf{msSet}}

\def\NN{\mathbb{N}}
\DeclareMathOperator{\Ob}{\mathsf{Ob}}

\def\oplax{\mathsf{oplax}}
\def\Oriental{\mathbb{O}}
\DeclareMathOperator{\Parity}{\mathsf{Par}}
\DeclareMathOperator{\Power}{\mathbb{P}}

\DeclareMathOperator{\Rel}{\mathsf{Rel}}
\DeclareMathOperator{\Segal}{\mathsf{Se}}
\def\sc{\mathsf{sc}}
\def\scr{\EuScript}
\DeclareMathOperator{\sd}{\mathsf{sd}}
\DeclareMathOperator{\SD}{\mathsf{SD}}
\def\Seg{\sf{Seg}}
\def\seg{\mathsf{seg}}
\def\Set{\sf{Set}}
\DeclareMathOperator{\SParity}{\mathcal{P}\mathsf{ar}}
\DeclareMathOperator{\Special}{\mathsf{Spe}}
\def\sf{\mathsf}

\def\Span{\sf{Span}}

\def\Spc{\sf{Spc}}

\DeclareMathOperator{\St}{\mathsf{St}}

\def\tDelta{{\sf{t}\Delta}}
\DeclareMathOperator{\tmod}{mod}
\DeclareMathOperator{\tr}{\mathsf{tr}}
\DeclareMathOperator{\Tw}{\mathsf{Tw}}

\def\uSpan{\uline{\Span}}

\newcommand{\teq}{\addtocounter{subsubsection}{1}\tag{\thesubsubsection}}

\DeclareMathOperator*\colim{colim}
\DeclareMathOperator*\hocolim{hocolim}

\makeatletter
\ExplSyntaxOn
\NewDocumentCommand { \LeftRightArrow } { m }
{
  \tl_gput_right:Nx \g_nicematrix_code_after_tl
  { \niru__LeftRightArrow:nnn { \arabic { iRow } } { \arabic { jCol } } { #1 } }
}

\cs_new_protected:Nn \niru__LeftRightArrow:nnn
{ \tikz \draw [-LaTeX,<->] (#1.5-|#2) -- (#1.5-|\int_eval:n{#2+#3}) ; }
\ExplSyntaxOff
\makeatother

\newcommand{\lacuna}[2]{%
  \small
  \begin{tabular}[t]{*{#1}{c}}#2%
  \end{tabular}%
}

\newcommand{\lacunaMove}[3]{%
  \begin{table}[h!]
    \small
    \setlength{\tabcolsep}{0.35em}
    \begin{NiceTabular}{lll}
      & #1 \\
      #2
      &
      \LeftRightArrow{1}
      &
      #3
    \end{NiceTabular}
\end{table}}

\title{The $(\infty,\infty)$-Category of Spans}
\author{Jonte G\"odicke}
\address{Max Planck Institute for Mathematics, Vivatgasse 7, 53111 Bonn, Germany}
\email{godicke@mpim-bonn.mpg.de}

\author{Quoc P. Ho}
\address{Department of Mathematics, The Hong Kong University of Science and Technology (HKUST), Clear Water Bay, Hong Kong}
\email{quoc.ho@ust.hk}

\author{Walker H. Stern}
\address{Department of Mathematics, Technical University of Munich, Munich, Germany}
\email{walker.stern@tum.de}

\date{\today}


\begin{document}
\maketitle
\begin{abstract}
  In this paper, we construct the $(\infty,\infty)$-category $\Span_\infty(\scr{C})$ of spans, also known as correspondences, in any given $(\infty,1)$-category $\scr{C}$ with finite limits. This yields new models for the span $(\infty,n)$-categories for $n \in \mathbb{N} \cup \{\infty\}$. We characterize the mapping $(\infty, n-1)$-categories in these $(\infty,n)$-categories, and thereby verify that our model agrees with other models for spans. Finally, and most importantly, we prove a new universal property, characterizing functors \emph{into} span $(\infty,n)$-categories, which specializes to the well-known relation with the twisted arrow categories in dimension $1$. These results will be used in the sequels to construct higher analogs of the classical Hall algebra construction, where ``higher'' refers to both higher categorical and ``higher monoidal'' structures, i.e., $\En_k$-algebras in $(\infty, n)$-categories for $n, k>1$.
\end{abstract}

\tableofcontents

\section{Introduction}

\subsection{Motivation}
This is the first in a series of papers whose overarching purpose is to explore lax phenomena in \emph{higher categorical settings} and their roles in constructing and studying interesting topological quantum field theories (TQFTs), especially those of \emph{geometric representation theoretic} origin and those relevant to \emph{low dimensional topology} and its \emph{categorification program}. Here, ``higher-categorical'' means $(\infty,n)$-categorical for $n\geq 2$, while ``lax'' refers to structures in which certain morphisms required to be equivalences in the ordinary, non-lax setting are no longer assumed to be so.

The interplay between geometric representation theory and TQFTs, both as a rich source of TQFTs and as a field that benefits from their philosophy is well established, as witnessed by the flourishing of ideas within and inspired by the various forms of the geometric Langlands program~\cite{ben-zvi_betti_2016,ben-zvi_relative_2023,GLC_proof_2024-1,GLC_proof_2024-2,GLC_proof_2024-3,GLC_proof_2024-4,GLC_proof_2024-5,ben-zvi_integral_2010,li_functions_2024,ho_eisenstein_2022,ho_graded_2025}. Recent developments have further demonstrated the power and necessity of deepening this perspective, benefitting from more recent advances in the theory of even higher (i.e., $(\infty,n)$-) and enriched categories~\cite{gaitsgory_toy_2022,ben-zvi_potent_2025,gaitsgory_local_2025,gaitsgory_applications_2025,scholze_geometry_2026}. Central to this perspective are $6$-functor formalisms and categories of correspondences, or spans, as we call them here, which were originally envisioned by Grothendieck's school and have recently received renewed attention in geometric representation theory~\cite{liu_enhanced_2017,gaitsgory_study_2017,stefanich_higher_2020,cnossen_universality_2025,scholze_six-functor_2026,fisher_exchange_2026}, not least because fundamental objects such as Hecke and Hall algebras and categories, together with categorical manipulations and computations involving them, arise elegantly and efficiently by linearizing the corresponding algebraic structures in the $(\infty,1)$-category of spans via a $6$-functor formalism~\cite{mirkovic_geometric_2007,ho_eisenstein_2022,porta_two-dimensional_2022,ho_graded_2025}.

Although our approach is geometrically inspired, it is distinguished from much existing work in categorical and geometric representation theory as it is motivated by the study of categorified link invariants, a subject that, with notable exceptions including~\cite{webster_geometric_2017,oblomkov_soergel_2020,ho_graded_2025,ho_relative_2026}, has largely developed through algebraic rather than geometric methods. Our point of departure is that, while lax structures are interesting in their own right, genuinely new non-lax algebraic structures can arise most naturally from lax geometric structures: the lax structure supplies all the necessary coherent \emph{data}, and the desired structure is recovered at the final stage by checking the \emph{property} that certain morphisms are invertible.

If ordinary, non-lax geometric structures live in $(\infty,1)$-categories of spans, their lax counterparts naturally inhabit higher categories of spans. It is well established (see, for example,~\cite{stefanich_higher_2020,haugseng_iterated_2018}) that for every $(\infty,1)$-category $\mathcal{C}$ with finite limits and every $n\in\mathbb{N}$, there is an $(\infty,n)$-category $\Span_n(\mathcal{C})$ whose $k$-morphisms, for $k\leq n$, are iterated spans in $\mathcal{C}$. We carry this construction to its limit by considering the $(\infty,\infty)$-category $\Span_\infty(\mathcal{C})$, whose objects are those of $\mathcal{C}$, whose morphisms are spans, whose $2$-morphisms are spans between spans, and so on. More formally, one may set
\[
  \Span_\infty(\mathcal{C})
  = \colim_{n\in \mathbb{N}} \Span_n(\mathcal{C}),
  \teq\label{eq:tentative_def_span_infty_colim}
\]
where the colimit is taken in the $(\infty,1)$-category $\Cat_\infty$ of $(\infty,\infty)$-categories.

Constructing lax structures in $\Span_\infty(\mathcal{C})$ is delicate, however, because of the immense amount of coherence data involved. It therefore requires a deeper understanding of how to \emph{map into} $\Span_\infty(\mathcal{C})$, as opposed to \emph{mapping out of} categories of spans. The latter problem, which in special cases corresponds to the construction of 6-functor formalisms, is by far much better understood.

Our central result, \zcref{intro:thm:functor_UP_inftycat} below, addresses precisely this problem by revealing a novel description of the $(\infty,\infty)$-category of functors into $\Span_\infty(\mathcal{C})$, equipped with oplax natural transformations (see \zcref{defn:Fun_oplax}). This structure is particularly well suited to constructing lax ``compound algebraic structures'' whose constituent parts interact only laxly, a feature we will exploit fully in our forthcoming work; see also \zcref{subsec:applications}.

\begin{introThm}[\zcref{thm:functor_UP_inftycat,rmk:span_infty_colim}] \label{intro:thm:functor_UP_inftycat}
  Let $\scr{C}$ be an $(\infty, 1)$-category with finite limits and $\scr{D}$ an $(\infty, \infty)$-category. Then, there is an equivalence
  \[
    \Fun^{\oplax}(\scr{D},\Span_\infty(\scr{C}))\simeq \Span_\infty(\Fun^{\cart}(\SD(\scr{D}),\scr{C}))
  \]
  natural in $\scr{D}$ and $\scr{C}$. Here, $\cart$ denotes the condition that certain diagrams are required to be limit and $\SD: \Cat_\infty \to \Cat_1$ denotes the functor of taking the barycentric subdivision defined via left Kan extending the usual barycentric subdivision functor on simplicial sets (see \zcref{defn:sd_SD_LKE} for a precise definition).
\end{introThm}

From the colimit description in \zcref{eq:tentative_def_span_infty_colim}, \zcref{intro:thm:functor_UP_inftycat} is highly nontrivial, as it concerns functors into a colimit. In this paper, we circumvent this problem completely by introducing a new model of $\Span_\infty(\mathcal{C})$, using complicial sets as a model for $(\infty,\infty)$-categories (note that a sketch of this complicial construction appeared in~\cite{dyckerhoff_cyclic_2026}). \zcref{intro:thm:functor_UP_inftycat} is then proved within this new model and \zcref{eq:tentative_def_span_infty_colim} is obtained by comparing with existing models.

In what follows, we begin in \zcref{subsec:applications} by outlining the applications we have in mind, then explain in \zcref{subsec:ideas} the ideas motivating our construction of $\Span_\infty(\mathcal{C})$, and finally summarize in \zcref{subsec:main_results} the paper's main results and what go into their proofs.

\subsection{Conventions}
Throughout, we follow the conventions of \zcref{subsec:convention}, including the ``implicit $\infty$'' convention whenever no confusion is likely.

\subsection{Applications} \label{subsec:applications}
A driving force for much of the research in low-dimensional topology and quantum algebra over the past few decades, the categorification program of Crane and Frenkel \cite{crane_four-dimensional_1994} seeks to construct $4$-dimensional TQFTs, conjectured to encode categorified link invariants and powerful invariants of $4$-manifolds. Despite the three decades of sustained work and a wealth of interesting examples, including the pioneering works on categorified quantum groups and link homology theories~\cite{rouquier_categorification_2006,khovanov_matrix_2008,khovanov_matrix_2008-1,khovanov_categorification_2010,khovanov_diagrammatic_2009,khovanov_diagrammatic_2010,webster_knot_2017,stroppel_categorification_2022}, we still know very little about how to put what we know together coherently to construct these highly desirable TQFTs. One of the main difficulties has to do with the immense coherence data included in building such a gadget.

The first concrete advance in this direction is the pioneering work~\cite{liu_braided_2024}, which, roughly speaking, assembles module categories over type-$A$ Hecke categories in their Soergel-bimodule realizations (known to govern many categorified link homology theories) into an $\En_2$-algebra object in $(\infty,2)$-categories. The requisite coherence data are constructed indirectly by obstruction theory: low-dimensional data are supplied by hand, the obstruction to lifting, resp. the space of lifts, is shown to vanish, resp. contractible.

By Dunn's additivity theorem, an $\En_n$-algebra may be viewed as an $n$-fold associative algebra: it carries $n$ associative algebra structures that are homotopically compatible in the appropriate sense, and hence is precisely a compound algebraic structure of the kind alluded to above. For $n=2$, an $\En_2$-algebra is, roughly speaking, a braided algebra, in which multiplications in opposite orders are related by a braiding equivalence together with higher coherences.\footnote{An $\En_2$-algebra in categories is precisely a braided monoidal category.} Its lax analog is obtained by retaining the braiding morphisms and coherence data while no longer requiring that these morphisms are equivalences.

In our forthcoming work~\cite{godicke_lax_nodate}, we combine the technology developed here, especially \zcref{intro:thm:functor_UP_inftycat}, with an $n$-fold extension of the $2$-Segal formalism of~\cite{galvez-carrillo_decomposition_2018,dyckerhoff_HigherSegalSpaces_2019,stern_2-segal_2021,godicke_infty-category_2024} to construct lax $\En_n$-algebras in $\Span_\infty(\mathcal{C})$, for instance when $\mathcal{C}$ is the category of stacks, taking as input an iterated version of Waldhausen's $S_\bullet$ construction. Linearizing these objects via a variant of the categorified sheaf theory first appeared in~\cite{stefanich_categorification_2025} then yields lax $\En_n$-algebras in higher categories, providing higher analogs of the classical Hall algebra construction. Here, ``higher'' refers to both higher categorical and ``higher monoidal'' structures, namely $\En_n$-algebras for $n>1$, with the case $n=1$ being the classical story.

Applied to the category of representations of the one-vertex quiver with no loop and combined with the theory of graded sheaves developed in~\cite{ho_revisiting_2025}, this construction recovers the $\En_2$-algebra object in $(\infty,2)$-categories of~\cite{liu_braided_2024}, with all the coherences geometrically encoded. But our procedure is a general Hall-algebraic construction which can handle a broad class of examples, among them general quivers, which interacts profoundly with the theory of categorified quantum groups. Exploring this circle of ideas is a central aim of our long-term project.

While the above are some of the main motivations for us, we believe that this is quite a ubiquitous phenomenon and expect that many other interesting structures can be obtained in a similar manner.

\subsection{The ideas} \label{subsec:ideas}
We will now briefly describe the ideas that go behind our construction. At a first approximation, an $\infty$-category is a kind of structure where we have objects, morphisms, morphisms between morphisms, and so on, ad infinitum, where essentially nothing is required to be invertible. At least intuitively, such a structure is determined by the collection of laxly and strongly commutative diagrams whose shapes are given by the simplices. Below, we will give an intuitive idea for what this means in low dimensions for the $\infty$-category of spans. In several cases, we illustrate both the shape of the lax simplex in $\Span_\infty(\mathcal{C})$ (on the left) and what it actually looks like in terms of spans (on the right).

\subsubsection{Dimensions $0$ and $1$} Let $\mathcal{C}$ be a category with finite limits. Then, the $\infty$-category $\Span_\infty(\mathcal{C})$ of spans in $\mathcal{C}$ has the same objects as $\mathcal{C}$. Given $x, y\in \mathcal{C}$, a morphism $f$ from $x$ to $y$ in $\Span_\infty(\mathcal{C})$ is an object $f\in \mathcal{C}$ equipped with maps to $x$ and $y$.
\[
  \begin{tikzcd}
    x \ar{r}[description]{f} & y
  \end{tikzcd}\hspace{2em}\leftrightsquigarrow\hspace{2em}
  \begin{tikzcd}[sep=small]
    & f \ar{dl} \ar{dr} \\
    x && y
  \end{tikzcd}
\]
Note that this is the same as a functor $\sd(\Delta^1) \to \mathcal{C}$, where $\sd$ is the functor of taking the barycentric subdivision. See \zcref{ntt:subdivision} for a quick review.

\subsubsection{Dimension $2$}
A $2$-morphism between morphisms $f$ and $g$ from $x$ to $y$ in $\Span_\infty(\mathcal{C})$ is then given by a span $h$ between $f$ and $g$, compatible with all other data, visualized below.
\[
  \begin{tikzcd}
    & f \ar{dl} \ar{dr} \\
    x & h \ar{u} \ar{d} & y \\
    & g \ar{ul} \ar{ur}
  \end{tikzcd}
\]
A lax commutative triangle in $\Span_\infty(\mathcal{C})$, i.e., a functor $\Delta^2_\flat \to \Span_\infty(\mathcal{C})$, takes the following form,
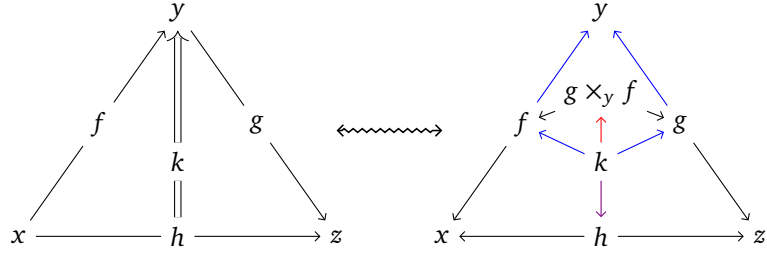
\[
  \begin{tikzpicture}[scale=0.7]
    \begin{scope}[xshift=-4cm]
      \coordinate (v0) at (-3,0);
      \coordinate (v2) at (3,0);
      \coordinate (v1) at (0,4.24);

      \coordinate (m01) at ($0.5*(v0) + 0.5*(v1)$);
      \coordinate (m02) at ($0.5*(v0) + 0.5*(v2)$);
      \coordinate (m12) at ($0.5*(v1) + 0.5*(v2)$);
      \coordinate (m012) at ($0.333*(v0) + 0.333*(v1) + 0.333*(v2)$);

      \foreach \x/\label in {0/x,1/y,2/z}{
        \path (v\x) node (v\x) {$\label$};
      };

      \foreach \x/\y/\label in {0/1/f,0/2/h,1/2/g}{
        \path (m\x\y) node (m\x\y) {$\label$};
      };

      \foreach \x/\y/\z/\label in {0/1/2/k}{
        \path (m\x\y\z) node (m\x\y\z) {$\label$};
      };

      \draw[->] (m01) to (v1);
      \draw[->] (m12) to (v2);
      \draw[->] (m02) to (v2);
      \draw[-] (m12) to (v1);
      \draw[-] (m01) to (v0);
      \draw[-] (m02) to (v0);
      \tikzset{myarrow/.tip={_[sep=-1.3pt].To[length=4pt]}}
      \draw[double distance=2pt, -] (m02) -- (m012);
      \draw[double distance=2pt, -myarrow] (m012) -- (v1);
    \end{scope}
    \begin{scope}[xshift=4cm]
      \coordinate (v0) at (-3,0);
      \coordinate (v2) at (3,0);
      \coordinate (v1) at (0,4.24);

      \node (e) at (0,2.7) {$g\times_{y}f$};

      \coordinate (m01) at ($0.5*(v0) + 0.5*(v1)$);
      \coordinate (m02) at ($0.5*(v0) + 0.5*(v2)$);
      \coordinate (m12) at ($0.5*(v1) + 0.5*(v2)$);
      \coordinate (m012) at ($0.333*(v0) + 0.333*(v1) + 0.333*(v2)$);

      \foreach \x/\label in {0/x,1/y,2/z}{
        \path (v\x) node (v\x) {$\label$};
      };

      \foreach \x/\y/\label in {0/1/f,0/2/h,1/2/g}{
        \path (m\x\y) node (m\x\y) {$\label$};
      };

      \foreach \x/\y/\z/\label in {0/1/2/k}{
        \path (m\x\y\z) node (m\x\y\z) {$\label$};
      };

      \draw[blue,->] (m01) to (v1);
      \draw[blue,->] (m12) to (v1);
      \draw[->] (m12) to (v2);
      \draw[->] (m01) to (v0);
      \draw[->] (m02) to (v0);
      \draw[->] (m02) to (v2);

      \foreach \x/\y/\z in {0/1/2}{
        \draw[blue,->] (m\x\y\z) to (m\x\y);
        \draw[->,violet] (m\x\y\z) to (m\x\z);
        \draw[blue,->] (m\x\y\z) to (m\y\z);
      }

      \draw[red,->] (m012) to (e);
      \draw[->] (e) to (m01);
      \draw[->] (e) to (m12);
    \end{scope}
    \draw[<->,semithick,decorate,
    decoration={zigzag,amplitude=0.7pt,segment length=1.2mm,pre=lineto,pre length=2pt,post=lineto,post length=2pt}] (-1,2) to (1,2);
  \end{tikzpicture}
\]
where the subscript $\flat$ indicates that we do not require the $2$-morphism in the source is an equivalence. One can read this diagram as a $2$-morphism between $h$ and $g\circ f$ (which is computed as $g\times_y f$), realized by the span below.
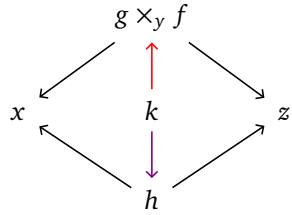
\[
  \begin{tikzcd}
    & g\times_y f \ar{dl} \ar{dr} \\
    x & k \ar[red]{u} \ar[violet]{d} & z \\
    & h \ar{ul} \ar{ur}
  \end{tikzcd}
\]

By the universal property of fiber products, the object $g\times_{y} f$ and the {\textcolor{red}{red}} morphism are uniquely determined by the {\textcolor{blue}{blue}} subdiagram. This means that we can, at least intuitively, remove $g\times_y f$ from the diagram, and still have the data of a $2$-simplex. We can therefore think of a lax triangle in $\Span_\infty(\scr{C})$ as being determined by a diagram $\sd(\Delta^2) \to \scr{C}$. Whether the $2$-morphisms represented by the unique non-degenerate $2$-simplex is invertible is then determined by whether the {\textcolor{blue}{blue}} \emph{and} {\textcolor{violet}{violet}} diagrams are limit diagrams. This simultaneously encodes the facts that compositions in spans are given by pullbacks and that a span is invertible precisely when both legs are.

\subsubsection{Higher dimensions}
One can correctly guess that the data of a map $\Delta^n_\flat \to \Span_\infty(\mathcal{C})$ is the same as a map $\sd(\Delta^n) \to \mathcal{C}$. Moreover, the top cell being invertible is equivalent to certain diagrams being limit. As we will see, the required limit conditions are precisely the higher Segal conditions (see \zcref{subsec:cubes_and_Segal} for a quick review), also known as the upper and lower Segal conditions in the literature \cite{poguntke_higher_2017}.

Working with the complicial set model of~\cite{verity_weak_2008,riehl_complicial_2018,ozornova_model_2020,loubaton_complicial_2024,loubaton_categorical_2024} (see also \zcref{rmk:weak_complicial_terminology}), this is precisely the approach we take in this paper.

\subsection{The main results} \label{subsec:main_results}
We will now outline the main results of the paper.

\subsubsection{The construction}
Let $\mathcal{C}$ be a $1$-category with finite limits, presented as a quasi-category. In \zcref{sec:fibrancy}, we define $\Span_\infty(\mathcal{C})$ as a decorated simplicial set\footnote{also known as a marked/stratified simplicial set. See also \zcref{ftn:marked_vs_stratified_etc}.}, whose $n$-simplices are given by
\[
  \Span_\infty(\mathcal{C})_n \coloneqq \Hom_{\Set_\Delta}(\sd(\Delta^n), \scr{C}).
\]
Moreover, an $n$-simplex is thin if the corresponding map $\sd(\Delta^n) \to \mathcal{C}$ satisfies certain limit conditions called the higher Segal conditions. We call such a functor \emph{Cartesian}. We note that the notion of a Cartesian functor can also be generalized to the case where $\Delta^n$ is replaced by a more general decorated simplicial set or an $\infty$-category; see \zcref{defn:cartesian_conditions,defn:more_general_cart_conditions}.

The first result of the paper then reads as follows.

\begin{introThm}[\zcref{thm:Span_infty_is_complicial}]
  $\Span_\infty(\mathcal{C})$ is a complicial set and hence, models an $(\infty, \infty)$-category.
\end{introThm}

\subsubsection{A model-independent characterization}

While the definition of $\Span_\infty(\mathcal{C})$ a priori depends on the presentation of $\mathcal{C}$ as a quasi-category, we upgrade our definition to a universal property in $\Cat_\infty$ in \zcref{sec:mapping_space_into_Span}. This universal property is, in fact, a model-independent manifestation of the observation that maps into $\Span_\infty(\scr{C})$ are fundamentally related to Cartesian maps out of barycentric subdivisions.

\begin{introThm}[\zcref{thm:SpanUP_I_internal}] \label{intro:thm:SpanUP_I_internal}
  Let $\scr{C}$ be a category with finite limits and $\scr{D}$ an $\infty$-category. Then, there is an equivalence
  \[
    \Map_{\Cat_\infty}(\scr{D},\Span_\infty(\scr{C}))\simeq \Map^{\cart}_{\Cat_1}(\SD(\scr{D}),\scr{C})
  \]
  natural in $\scr{D}\in \Cat_\infty$. Here, $\cart$ denotes the Cartesian conditions and $\SD: \Cat_\infty \to \Cat_1$ denotes the functor of taking the barycentric subdivision defined via left Kan extending the usual barycentric subdivision functor on simplicial sets (see \zcref{defn:sd_SD_LKE} for a precise definition).
\end{introThm}

As an immediate consequence of the theorem above, in \zcref{cor:upgrade_Span_infty_functor}, we upgrade our construction $\Span_\infty$ to a limit preserving functor
\[
  \Span_\infty: \Cat_1^{\lex} \to \Cat_\infty,
\]
where $\Cat_1^{\lex}$ denotes the category whose objects (resp. morphisms) are categories with finite limits (resp. finite limit preserving functors). From here, we deduce in \zcref{cor:monoidal_structure_Span_infty} that the Cartesian monoidal structure on $\mathcal{C}$ induces a symmetric monoidal structure on $\Span_\infty(\mathcal{C})$.

The main ingredient that goes into \zcref{intro:thm:SpanUP_I_internal} is the fact that a certain functor $\sd(\Delta^n \times \Delta^m) \to \sd(\Delta^n) \times \sd(\Delta^m)$ is a localization, whose proof occupies most of \zcref{sec:mapping_space_into_Span}. Despite being simple-looking, this statement goes through a fiber-wise criterion for localizations for locally Cartesian fibrations and an analysis of the Gray tensor product, both of independent interest.

\begin{rmk}
  The paper \cite{gepner_oriented_2026} appeared as we were completing this manuscript. Of particular relevance is its new model-independent description of $\infty$-categories via \emph{complicial spaces}, which suggests an alternative, more directly model-independent construction of $\Span_\infty(\mathcal{C})$ using the new results. Even along this route, however, most of the present work would remain unchanged: much of the proof of \zcref{intro:thm:SpanUP_I_internal}, the localization, in particular, is used crucially in the subsequent, model-independent sections.
\end{rmk}

\subsubsection{The $\infty$-category of functors}

We now come to the main thrust of the paper, \zcref{sec:mapping_categories_into_Span}, whose main goal is to show that the $\infty$-category of functors into $\Span_\infty(\mathcal{C})$ (with oplax natural transformations, see \zcref{defn:Fun_oplax}) reveals a novel nature of $\Span_\infty$.

\begin{introThm}[\zcref{thm:functor_UP_inftycat}] \label{intro:thm:functor_UP_inftycat_2}
  Let $\scr{C}\in \Cat_1^{\lex}$ and $\scr{D}\in \Cat_\infty$. There is an equivalence
  \[
    \Fun^{\oplax}(\scr{D},\Span_\infty(\scr{C}))\simeq \Span_\infty(\Fun^{\cart}(\SD(\scr{D}),\scr{C}))
  \]
  natural in $\scr{D}$ and $\scr{C}$.
\end{introThm}

The proof of this theorem amounts to transporting a large family of limit conditions across the localization constructed in the proof of \zcref{intro:thm:SpanUP_I_internal}. To organize this process, we develop a graphical calculus of lines on a grid that greatly simplifies reasoning about these conditions.

\subsubsection{Comparison with other models}
Finally, in \zcref{sec:comparison_of_models_and_misc}, we compare our construction with existing models of higher categories of spans. For each $n \geq 1$, let $\Span_{n\frac{1}{2}}(\mathcal{C})$ denote the underlying $(n+1)$-category of $\Span_\infty(\mathcal{C})$, in which morphisms above degree $n+1$ are invertible and $(n+1)$-morphisms are genuine morphisms in $\mathcal{C}$ rather than spans. The works \cite{stefanich_higher_2020,haugseng_iterated_2018} construct, by entirely different methods, higher categories of spans denoted respectively by $(n+1)\Corr(\mathcal{C})$ and $\Span^+_n(\mathcal{C})$, with the same informal description. These two models are known to be equivalent; we prove that ours is equivalent to them.

\begin{introThm}[\zcref{cor:functor_is_equiv_sym_mon}] \label{intro:cor:functor_is_equiv_sym_mon}
  For each $n\geq 2$, we have an equivalence of symmetric monoidal categories $n\Corr(\mathcal{C}) \simeq \Span_{(n-1)\frac{1}{2}} (\mathcal{C})$ compatible with the functors from $\mathcal{C}$.
\end{introThm}

We prove \zcref{intro:cor:functor_is_equiv_sym_mon} inductively, using the universal property of $n\Corr(\mathcal{C})$ established in \cite{stefanich_higher_2020} and expected properties about the $\Hom$-categories of $\Span_{n\frac{1}{2}}(\mathcal{C})$ which we prove along the way. Note also that the corresponding comparison statement about $\Span_n(\mathcal{C})$ follows immediately by taking the underlying $(\infty, n)$-categories of $\Span_{n\frac{1}{2}}(\mathcal{C})$ and $(n+1)\Corr(\mathcal{C})$.

In addition to \zcref{cor:functor_is_equiv_sym_mon}, we establish in~\zcref{sec:comparison_of_models_and_misc} analogs of \zcref{intro:thm:functor_UP_inftycat_2} for the various versions of the $(\infty,n)$-categories of spans, see~\zcref{thm:universal_property_Span_k,thm:universal_property_Span_k_1/2}, recovering and generalizing the results of \cite{haugseng_two-variable_2023} concerning spans.

\begin{introThm}[\zcref{thm:universal_property_Span_k,thm:universal_property_Span_k_1/2}]
  Let $\scr{C}$ be a $1$-category with finite limits and $\scr{D}$ an $\infty$-category. Then, there are natural equivalence
  \begin{align*}
    \Span_n^{\mix}(\Fun^{\cart}(\SD(\tau_n(\scr{D})),\scr{C})) &\simeq \Fun^{\oplax}(\scr{D},\Span_n(\scr{C})), \quad\text{and}\\
    \Span_{n\frac{1}{2}}^{\mix}(\Fun^{\cart}(\SD(\tau_{n+1}(\scr{D})),\scr{C}))&\simeq \Fun^{\oplax}(\scr{D},\Span_{n\frac{1}{2}}(\scr{C})),
  \end{align*}
  where $\tau$ is the truncation functor (see \zcref{defn:truncation_and_friends}) and $\Span^{\mix}$ essentially keeps track of invertible (higher) morphisms in the functor categories (see \zcref{defn:Span_k_mix}).
\end{introThm}

These results are necessarily less streamlined than \zcref{intro:thm:functor_UP_inftycat_2}, since they must explicitly track which morphisms are required to be invertible. Philosophically, this reflects the advantage of working directly in $(\infty,\infty)$-categories, where all morphisms are retained but invertibility may be imposed afterward as a property rather than built into the structure.

\subsection*{Acknowledgements}
During the long gestation of this work, we have greatly benefited from discussions with Fernando Abell\'an, Rune Haugseng, Talak Manum, Pelle Steffens, Claudia Scheimbauer, and Tobias Dyckerhoff. We are particularly indebted to Tim Campion for helping us construct the argument of \zcref{app:loc}; to Tashi Walde for helping us understand his notation and arguments in \cite{walde_higher_2020} and aiding us with some of the proofs in \zcref{sec:fibrancy}; and to David Gepner and Germ\'an Stefanich for helpful conversations surrounding the comparison with other models of spans.

We are grateful to the Hong Kong University of Science and Technology, the University of Hamburg, the Technical University of Munich, and the Max-Planck Institut for their hospitality and support. We also thank the Ludwig-Maximilian University of Munich for (unknowingly) providing us space to work on this project.

During our work on this paper,
J. G\"odicke was partially supported by the Collaborative Research Center - SFB 1624 ``Higher structures, moduli spaces and integrability'' - 506632645,
Q. Ho was partially supported by the Hong Kong RGC GRF grants 16304923 and 16301324,
and W. H. Stern was affiliated with the SFB 1624 ``Higher structures, moduli spaces and integrability'' - 506632645 and employed by the Technical University of Munich.

\section{Background and preliminaries}

In this section, we collect the various background results, definitions, and notations which will be necessary for our work. In particular, we recall \emph{decorated simplicial sets} (called \emph{stratified simplicial sets} in \cite{verity_weak_2008}), the model structure on decorated simplicial sets which models $(\infty,n)$-categories by \cite{loubaton_complicial_2024}, the Gray tensor product of complicial sets, and a density theorem for $(\infty,\infty)$-categories. We also recall the \emph{higher Segal conditions} of \cite{dyckerhoff_HigherSegalSpaces_2019} and \cite{poguntke_higher_2017}, as well as some useful properties and descriptions from \cite{walde_higher_2020}.

\subsection{Higher categories and notational conventions} \label{subsec:convention}

In this section, we briefly fix the notation and terminology we will use for $(\infty,n)$-categories.

We adopt the ``implicit $\infty$'' convention. Namely, unless confusion is likely to occur, the term $n$-category is used to refer to an $(\infty, n)$-category, and moreover, the term strict $n$-category is used to refer to a strict $n$-category in the usual sense. In particular, the term $\infty$-category is used to refer to an $(\infty, \infty)$-category. When $n=1$, we usually drop it from the notation together and use the term category instead of $1$-category.

\begin{ntt}
  For $n\in \NN\cup \{\infty\}$, we denote by $\Cat_n$ the $(\infty,1)$-category of $(\infty,n)$-categories. We will sometimes use the notation $\Spc$ for $\Cat_0$.
\end{ntt}

\begin{rmk}
  When we resort to model categorical arguments we will typically view $\Cat_0$ as presented by the Kan--Quillen model structure (\cite{quillen_homotopical_1967}), $\Cat_1$ as presented by either the Joyal model structure (\cite[Sect.2.2.5]{lurie_higher_2017-1}) or the marked model structure (\cite[Sect.3.1]{lurie_higher_2017-1}), and $\Cat_n$ for $n\geq 2$ as presented by the complicial model structures exposed in the next section.
\end{rmk}

\begin{ntt}
  Given an $(\infty,n)$-category $\scr{C}$ and objects $x,y\in \scr{C}$, we denote by $\scr{C}(x,y)$ the mapping $(\infty,n-1)$-category from $x$ to $y$ in $\scr{C}$.
\end{ntt}

\begin{ntt}
  Given two $(\infty,n)$-categories $\scr{C}$ and $\scr{D}$, we will denote by $\Fun(\scr{C},\scr{D})$ the $(\infty,n)$-category of functors from $\scr{C}$ to $\scr{D}$, and by $\Map(\scr{C},\scr{D})$ the underlying $\infty$-groupoid of $\Fun(\scr{C},\scr{D})$. Note that this means the notations  $\Map(\scr{C},\scr{D})$ and  $\Cat_n(\scr{C},\scr{D})$ coincide.
\end{ntt}

\subsection{Decorated simplicial sets and complicial sets}

The connection between the simplex category $\Delta$ and higher categories is well-known, manifesting in simplicial set-based model structures for $(\infty,0)$-categories and  $(\infty,1)$-categories, as well as the Segal space approach to $(\infty,n)$-categories. The idea of complicial sets is to view an $n$-simplex as an avatar of an $n$-morphism, in line with Joyal's notion of the orientals. However, to identify a fully faithful nerve from $\omega$-categories, Street and Roberts \cite{street_algebra_1987} found it necessary to consider additional data on a simplicial set --- a collection of simplicies in each dimension $n\geq 1$ representing \emph{invertible} $n$-morphisms. This leads to the definition of what we call a decorated simplicial set.

\begin{ntt}
  For categories $\scr{C}$ and $\scr{D}$, we write $\scr{C}^\scr{D} \coloneqq \Fun(\scr{D}, \scr{C})$ and $\scr{C}_\scr{D} \coloneqq \Fun(\scr{D}^\op, \scr{C})$.

  In particular, the category of simplicial sets is denoted by $\Set_\Delta$, where $\Delta$ is the simplex category, and the category of simplicial space is denoted by $\Spc_\Delta$.
\end{ntt}

\begin{defn}
  A \emph{decorated simplicial set}\footnote{\label{ftn:marked_vs_stratified_etc}The terms \emph{marked simplicial set} and \emph{stratified simplicial set} are both used in the literature for this concept. However, the former conflicts with the use of marked simplicial sets to denote simplicial sets with a chosen marking on $1$-simplices, and the latter conflicts with the notion of stratified spaces.} $(X,tX)$ is a pair consisting of a simplicial set $X$ together with, in every dimension $n\geq 1$, a subset $tX_n\subset X_n$ which contains all degenerate simplices. We write $tX=\bigsqcup_{n\geq 1} tX_n$ and call simplices in $tX$ \emph{thin simplices}.

  A \emph{morphism of decorated simplicial sets} $(X,tX)\to (Y,tY)$ is a morphism $f:X\to Y$ of simplicial sets such that $f(tX)\subset tY$. We denote the category of decorated simplicial sets $\Set_\Delta^{\dec}$.
\end{defn}

\begin{ntt}
  The forgetful functor $\Set_\Delta^{\dec}\to \Set_\Delta$ admits both a left and a right adjoint, which we will denote by $(-)_\flat$ and $(-)_\sharp$ respectively. That is, for a simplicial set $X$, $X_\flat$ is the decorated simplicial set in which only the degenerate simplices are thin, and $X_\sharp$ is the decorated simplicial set in which every $n$-simplex for $n\geq 1$ is thin.
\end{ntt}

\begin{ntt}
  For $n\geq -1$, we write $\Delta^n$ for the standard $n$-simplex as a simplicial set, letting $\Delta^{-1}=\varnothing$ be the empty simplicial set. We write $\partial\Delta^n$ for its boundary and, for $0\leq k\leq n$, we write  $\Lambda^n_k$ for the $k$\textsuperscript{th} horn of $\Delta^n$.  We also fix the following decorated simplices and horns.
  \begin{enumerate}
    \item For $n\geq 1$, we write $\Delta^n_+$ for the decorated simplicial set obtained from $\Delta_\flat$ by declaring the unique non-degenerate $n$-simplex to be thin.
    \item For $0\leq k\leq n$, we denote by $\Delta^n_k$ the standard $n$-simplex in which a non-degenerate simplex is thin
      if and only if it contains the vertices $\{k-1,k,k+1\}$.
    \item When no decoration on the horn is specified, we denote by $\Lambda^n_k$ the decorated simplicial set whose decoration is inherited from $\Delta^n_k$ under the inclusion.
    \item We denote by $\Delta^n_\partial$ the decorated simplicial set in which a non-degenerate simplex is declared to be thin if and only if it factors through $\partial \Delta^n$.
    \item by $\Delta^m_{k\partial}$ for $0\leq k \leq m$
      the $m$-simplex in which the marked simplices are the ones in
      $\Delta^m_k$ except for the top simplex.
    \item We denote by $\Delta^n_{k^\prime}$ the decorated simplicial set obtained from $\Delta^n_k$ by additionally marking the $(k-1)$- and $(k+1)$-faces.
    \item We denote by $\Delta^n_{k^{\prime\prime}}$ the decorated simplicial set obtained from $\Delta^n_{k^\prime}$ by additionally marking the $k$-face.
  \end{enumerate}
\end{ntt}

\begin{defn}
  Given decorated simplicial sets $(X,tX)$ and $(Y,tY)$, the \emph{join} is the decorated simplicial set $(X\star Y,tX\star tY)$ where $X\star Y$ denotes the usual join of simplicial sets, and a simplex $x\star y$ is thin if either $x$ or $y$ is.
\end{defn}

\begin{defn}
  We define the following four classes of maps
  \begin{enumerate}
    \item The \emph{complicial horn extensions} are the canonical inclusions
      \[
        \begin{tikzcd}
          \Lambda^n_k \arrow[r] & \Delta^n_k
        \end{tikzcd} \qquad n\geq 1, \qquad 0\leq k\leq n
      \]
    \item The \emph{complicial thinness extensions} are the maps
      \[
        \begin{tikzcd}
          \Delta^n_{k^\prime}\arrow[r] & \Delta^{n}_{k^{\prime\prime}}
        \end{tikzcd} \qquad n\geq 2, \qquad 0\leq k \leq n
      \]
    \item For a fixed $m\in \NN\cup{\infty}$, we define the \emph{$m$-triviality extensions} to be the inclusions
      \[
        \begin{tikzcd}
          \Delta^n_\flat \arrow[r] & \Delta^n_+
        \end{tikzcd}\qquad n>m
      \]
    \item For $n\geq 3$, we inductively define the \emph{saturation extensions} $\Delta^n_{\eq}\to\Delta^n_{\eq+}$ as follows:
      \begin{enumerate}
        \item In dimension $3$, $\Delta^3_{\eq}$ has all  simplices of dimension $2$ or higher thin, as well as the $1$-simplices $02$ and $13$. The target $\Delta^3_{\eq+}$ is simply the maximally marked $\Delta^3_\sharp$.
        \item In dimension $n$, the saturation extension is
          \[
            \begin{tikzcd}
              \Delta^n_{\eq}\arrow[r,phantom,"\coloneqq"{description}] &[-2em] \Delta^0\star \Delta^{n-1}_{\eq}\arrow[r] & \Delta^0\star\Delta^{n-1}_{\eq+} \arrow[r,phantom,"\eqqcolon"{description}] &[-2em] \Delta^n_{\eq+}
            \end{tikzcd}
          \]
      \end{enumerate}
  \end{enumerate}
  We collectively term the morphisms from these four sets the \emph{$m$-trivial elementary anodyne extensions}. When $m=\infty$, we simply refer to the \emph{elementary anodyne extensions}. We will call the subset of elementary anodyne extensions obtained by removing the horn extensions and thinness extensions for $k=0,n$ the \emph{inner elementary anodyne extensions}.
\end{defn}

\begin{defn}
  We define an \emph{$m$-complicial set} to be a decorated simplicial set with the right lifting property against all $m$-trivial elementary anodyne extensions. Moreover, we define an \emph{inner $m$-complicial set} to be a decorated simplicial set with the right lifting property against all inner $m$-trivial elementary anodyne extensions.
\end{defn}

\begin{rmk}
  Note that the associativity of the join of decorated simplicial sets implies that our saturation extensions agree with those of \cite{loubaton_complicial_2024}.
\end{rmk}

\begin{defn}
  We define a simplicial set $E$ to be the colimit of the diagram
  \[
    \begin{tikzcd}
      \Delta^1 & \Delta^2 \arrow[l,"s_1"'] \arrow[r,"d_2"] & \Delta^3 & \Delta^2 \arrow[l,"d_1"']\arrow[r,"s_0"] & \Delta^1
    \end{tikzcd}
  \]
  and consider the inclusion $\iota:\Delta^1\to E$ induced by the simplicial map $d_3\circ d_2:\Delta^1\to \Delta^3$.
\end{defn}

\begin{rmk}
  In \cite{verity_weak_2008}, Verity denotes $E$ by $E^-_3$ and identifies it with a simplicial subset of (the nerve of) the walking isomorphism. Loosely speaking, we can think of the inclusion $\iota:\Delta^1_\sharp\to E_\sharp$ as formally appending a $1$-morphism which is left and right inverse to the $1$-morphism of $\Delta^1$ in a coherent way. The import of this inclusion is the next proposition.
\end{rmk}

\begin{prop}\label{prop:innercomptocomp}
  An inner $m$-trivial complicial set is an $m$-trivial complicial set if and only if it has the right lifting property against the morphism
  \[
    \begin{tikzcd}
      \iota:&[-3em]\Delta^1_\sharp \arrow[r] & E_\sharp
    \end{tikzcd}
  \]
\end{prop}

\begin{proof}
  This is \cite[Lemma 47]{verity_weak_2008}.
\end{proof}

\begin{thm}
  For every $m\in \NN\cup \{\infty\}$, there is a combinatorial, simplicial, Cartesian model structure on $\Set_\Delta^{\dec}$ such that
  \begin{itemize}
    \item[(C)] The cofibrations are the monomorphisms.
    \item[(F)] The fibrant objects are the $m$-complicial sets.
  \end{itemize}
\end{thm}

\begin{proof}
  This theorem is an instance of \cite[Theorem 100]{verity_weak_2008} and the result has been formulated in \cite{riehl_complicial_2018}. The result has been proven in \cite[Theorem 1.25]{ozornova_model_2020}.
\end{proof}

\begin{thm}
  The model structure for $m$-complicial sets is a model for $(\infty,n)$-categories.
\end{thm}

\begin{proof}
  This is \cite[Theorem 3.3.1.11]{loubaton_complicial_2024}
\end{proof}

\begin{rem} \label{rmk:weak_complicial_terminology}
  In \cite{verity_weak_2008} Verity originally introduced the term complicial set for decorated simplicial sets that admit \emph{unique} lifts of elementary anodyne extensions and used the term \emph{weak complicial} set for what we would call a complicial set.
\end{rem}

\subsection{The category \texorpdfstring{$\tDelta$}{tΔ} and density}
The category $\tDelta$, along with a dense functor to $\Cat_\infty$, is one of the key technical tools we employ to connect model-categorical computations with model-independent results. Loosely speaking, $\tDelta$ consists of Street's orientals $\mathbb{O}^n$ together with the $n$-categories given by localizing $\mathbb{O}^n$ at its defining $n$-morphism. However, $\tDelta$ does not contain all morphisms between the orientals, but rather, only those which come from the simplex category. As such, the density of the orientals in $\Cat_\infty$ as proven in \cite{masuda_algebra_2024} does not suffice for our purposes.

\begin{defn}
  We define the category $\tDelta$ to be the full subcategory of $\Set_\Delta^{\dec}$ spanned by $\Delta^n_\flat$ for $n\geq 0$ and $\Delta^n_+$ for $n\geq 1$. We define
  \[
    \begin{tikzcd}
      \xi: &[-3em] \tDelta \arrow[r] & \Cat_\infty
    \end{tikzcd}
  \]
  to be the composition of the inclusion with the localization functor
  \[
    \begin{tikzcd}
      \Set_\Delta^{\dec} \arrow[r] & \Cat_\infty.
    \end{tikzcd}
  \]
\end{defn}

\begin{rmk}
  The functor $\xi$ sends $\Delta^n_\flat$ to the oriental $\mathbb{O}^n$ by \cite[Theorem 2.1]{maehara_orientals_2023}.
\end{rmk}

\begin{prop}\label{prop:tDelta_dense}
  The functor $\tDelta\to \Cat_\infty$ is dense, and hence, we have a fully faithful embedding $\Cat_\infty \subset \Spc_{\tDelta} \coloneqq \Fun(\tDelta^\op, \Spc)$.
\end{prop}

\begin{proof}
  By \cite{ozornova_model_2020}, there is a model structure on $\Fun(\tDelta^\op,\Set_\Delta)$, constructed as a left Bousfield localization of the injective model structure, which models $(\infty,\infty)$-categories. Thus, $\Cat_\infty$ is a reflective localization of $\Fun(\tDelta^\op, \Set_\Delta)$, and so the composite
  \[
    \begin{tikzcd}
      \tDelta \arrow[r] & \Fun(\tDelta^\op,\Set_\Delta)\arrow[r] & \Cat_\infty
    \end{tikzcd}
  \]
  of the localization with the Yoneda embedding is a dense functor. It thus suffices to identify our functor with the Yoneda embedding under the appropriate Quillen equivalences. In \cite{ozornova_model_2020}, the authors construct left Quillen equivalences
  \[
    \begin{tikzcd}
      \Set_\Delta^{\dec}  & \Set_{\tDelta} \arrow[r,"p^\ast"]\arrow[l,"R"'] & \Fun(\tDelta^\op,\Set_{\Delta})
    \end{tikzcd}
  \]
  where $R$ is the reflection onto a reflective subcategory, and $p^\ast$ takes a $\tDelta$-set to the corresponding discrete $\tDelta$-space. Since $R$ and $p^\ast$ commute with the inclusions of $\tDelta$ and all objects of $\Set_{\tDelta}$ are cofibrant, it follows that we can identify the dense functor defined above with our chosen functor $\xi$, completing the proof.
\end{proof}

\begin{rmk}
  The main use of \zcref{prop:tDelta_dense} in this paper will be that we can identify $(\infty,\infty)$-categories by identifying the restriction of their representing functors along $\xi$.
\end{rmk}

\subsection{The Gray product of decorated simplicial sets}

We now briefly recall the definition of the Gray product of decorated simplicial sets from \cite{verity_weak_2008}.

\begin{const}
  For any $n,p,q\geq 0$ with $n=p+q$, we define the complicial
  \emph{parity operators}
  \begin{align*}
    \sqcup_{p,q}^{1}: [p] &\rightarrow [n] && \text{and} & \sqcup_{p,q}^{2}: [q] &\rightarrow [n] \\
    k &\mapsto k &&& k &\mapsto k+p.
  \end{align*}
\end{const}

\begin{defn}
  Let $(X,tX)$ and $(Y,tY)$ be decorated simplicial sets. The \emph{Gray
  tensor product} of $(X,tX)$ and $(Y,tY)$ is the decorated simplicial set
  \[
    X\otimes Y = (X\times Y, tX\otimes tY)
  \]
  where $tX\otimes tY$ is the set of pairs $(x,y)$, such that for any
  partition $n=p+q$ either $\sqcup_{p,q}^{1}x$ or
  $\sqcup_{p,q}^{2}y$ is thin.
\end{defn}

\begin{defn} \label{defn:Fun_oplax}
  Let $S\in\Set_{\Delta}$ be a decorated simplicial set.
  We define the functor
  \[
    \begin{tikzcd}
      \Fun^{\oplax}(S,-): &[-3em]\Set_\Delta^{\dec}\arrow[r] & \Set_\Delta^{\dec}
    \end{tikzcd}
  \]
  as the right adjoint to $- \otimes S$.
\end{defn}

\subsection{Cubes and Segal conditions} \label{subsec:cubes_and_Segal}
In this subsection, we introduce the cubical version of the higher Segal conditions as introduced by Poguntke \cite{poguntke_higher_2017}.
These form the main ingredient for the construction of the $\infty$-category of spans in the next section.

Our arguments in the coming sections will often require the study of cubes in more general posets which correspond to Segal cubes under various localizations and adjunctions. We thus begin with some notation for cubes in posets.

\begin{ntt} \label{ntt:subdivision}
  For every finite set $M$, we regard the powerset $\bP(M)$ as a poset partially ordered by inclusion and denote $\Power_*(M)\coloneq \Power(M)\setminus \{\varnothing\}$.

  We denote by $\sd(M)$ the poset $\Power_*(M)^\op$. When $M$ is a totally ordered set, we call $\sd(M)$ the \emph{barycentric subdivision} of $M$.
\end{ntt}

\begin{defn}
  Let $\scr{C}$ be a category and $M$ a finite set. An \emph{$M$-shaped cube} in $\scr{C}$ is a functor $\bP(M)\rightarrow \scr{C}$. Similarly, we call a functor $\bP_{\ast}(M) \rightarrow \scr{C}$  an \emph{$M$-shaped punctured cube} in $\scr{C}$. A cube in $\scr{C}$ is called \emph{non-degenerate} if it is a conservative functor.

  We call an $M$-shaped cube a \emph{limit cube} if the canonical comparison morphism $F(\varnothing)\rightarrow \lim_{\Power_*(\mathcal{U})}F$ is an equivalence.
\end{defn}

We can construct cubes in a poset by considering meets and joins, in a way that generalizes the notion of a cover. Since we consider limit cubes, and our typical convention is that subsets are ordered by reverse inclusion, the roles of joins and meets are reversed from what one might expect.

\begin{defn} \label{defn:join_and_meet}
  Let $P$ be a poset and $\mathcal{U} \subset P$.
  \begin{itemize}
    \item The \emph{join} of $p \in \mathcal{U}$, denoted by $\vee_{p\in \mathcal{U}} p$, is the least upper bound (colimit) in $P$ of $p \in \mathcal{U}$.
    \item The \emph{meet} of $p \in \mathcal{U}$, denoted by $\wedge_{p\in \mathcal{U}} p$, is the greatest lower bound (limit) in $P$ of $p \in \mathcal{U}$.
  \end{itemize}
\end{defn}

\begin{defn}
  Let $P$ be a poset. A \emph{sieve $\mathcal{U}$ in $P$} is a finite subset $\mathcal{U} \subset P$. For $q\in P$ such that $q \leq p$ for all $p \in \mathcal{U}$, we can consider $\mathcal{U}$ to be \emph{a sieve on $q$} and denote it by $\mathcal{U}|q$. The sieve $\mathcal{U}$ is \emph{saturated} if for every set $\mathcal{V} \subset \mathcal{U}$, $\vee_{p\in \mathcal{V}} p$ exists.

  A cube in $P$ associated to a sieve $\mathcal{U}|q$ on $q$ is given by
  \[
    \begin{tikzcd}[row sep=0em]
      C_{\mathcal{U}|q}:&[-3em]\Power(\mathcal{U}) \arrow[r] & P\\
      & \mathcal{S} \arrow[r,mapsto] & \bigvee_{p\in \mathcal{S}} p \\
      & \varnothing \arrow[r,mapsto] & q.
    \end{tikzcd}
  \]
  When $q$ is clear from the context (for example, when $\mathcal{U}$ is already being considered to be a sieve on $q$), we simply write $C_{\mathcal{U}}$ for $C_{\mathcal{U}|q}$.

  We call the restriction of $C_{\mathcal{U}|q}$ to $\Power_*(\mathcal{U})$ the \emph{punctured cube associated to $\mathcal{U}|q$}.
\end{defn}

\begin{defn}
  Let $\scr{C}$ be a category with finite limits, $P$ a poset, and $\mathcal{U}\subset P$ a saturated sieve on $q\in P$. A functor $F:P \to \scr{C}$ is a \emph{$\mathcal{U}$-sheaf on $q$}, or equivalently, $\mathcal{U}$ is $F$-local on $q$, if the associated cube
  \[
    \begin{tikzcd}
      F\circ C_{\mathcal{U}|q}:\Power(\mathcal{U}) \arrow[r] & P \arrow[r] & \scr{C}
    \end{tikzcd}
  \]
  is a limit cube. When $q$ is clear from the context, we simply say that $F$ is a \emph{$\mathcal{U}$-sheaf}, or equivalently, $\mathcal{U}$ is $F$-local.
\end{defn}

Specializing to the poset $P=\sd(M)$ for a set $M$, we obtain the following definition.

\begin{defn} \label{defn:sieve_on_subset}
  Let $M$ be a finite set and $U \subset M$. A \emph{sieve on $U$} is a sieve $\mathcal{U}\subset \sd(M)$ on $U$ in $\sd(M)$.

  A sieve on $M$ is additionally called \emph{broad} if $\# (S)=\#(M)-1$ for any $S \in \mathcal{U}$.
\end{defn}

\begin{rmk}
  Any sieve $\mathcal{U}$ in $\sd(M)$ is a sieve on $M$, on $\bigcup_{S \in \mathcal{U}} S$, and on anything in between.
\end{rmk}

\begin{ntt}
  Let $M$ be a set and $I\subset M$. We write $M_{I}\coloneq M\setminus I$. When $I=\{i\}$, we also write $M_{i} \coloneqq M_{\{i\}}$, by abuse of notation.
\end{ntt}

\begin{defn}
  Let $M$ be a finite totally ordered set. The parity of $m$ in $M$, denoted by $\Parity_M(m)$, is defined to be the parity of the number of elements of $M$ greater than $m$. That is, $m$ is said to be \emph{even} (resp. \emph{odd}) if the number of elements in $M$ greater than $m$ is even (resp. odd). We use $-\Parity_M(m)$ to denote the parity opposite that of $m$. For example, when $m$ is odd, then $-\Parity_M(m)$ is even, and vice versa.

  We call $N\subset M$ \emph{even} (resp. \emph{odd}) if it is of the form $M_m \coloneqq M\setminus \{m\}$ for an \emph{even} (resp. \emph{odd}) element $m\in M$.
\end{defn}

\begin{defn}
  The \emph{even} (resp. \emph{odd}) \emph{saturated sieves} on $M$ consists of even (resp. odd) subsets of $M$ and is denoted by $\mathcal{E}_M$ (resp. $\mathcal{O}_M$). We call the associated cubes, denoted by $C_{\mathcal{E}_M}$ (resp. $C_{\mathcal{O}_M}$), the \emph{even} (resp. \emph{odd}) ($M$-)\emph{Segal cube}.

  By abuse of notation, given a sieve $\mathcal{U}$ on $M$ that is either even or odd, we also use $\Parity(\mathcal{U})$ to denote the parity of this sieve.
\end{defn}

\begin{defn} \label{defn:Segal_conditions}
  Let $M$ be a finite totally ordered set. A functor $F:\sd(M)\to \scr{C}$ is called
  \begin{enumerate}
    \item \emph{odd (M-)Segal} if it is an $\mathcal{O}_M$-sheaf; and
    \item \emph{even (M-)Segal} if it is a $\mathcal{E}_M$-sheaf.
  \end{enumerate}
  We further call such a functor ($M$-)\emph{Segal} if it is both even and odd $M$-Segal.

  When $M = [n]$, we also use the term (odd/even) $(n-1)$-\emph{Segal} to refer to (odd/even) $M$-Segal.
\end{defn}

\begin{ntt}
  Let $M$ be as above and $m\in M$. We use $\Parity_M(m)$\emph{-Segal sieve}  (or simply $\Parity(m)$\emph{-Segal sieve}, when $M$ is clear from the context), denoted by $\SParity_M(m)$ (or simply $\SParity(m)$), to refer to the Segal sieve containing $M_m$. We use $\Parity_M(m)$\emph{-Segal cube} (or simply $\Parity(m)$\emph{-Segal cube}) to refer to the cube associated to it.

  Similarly, for $m\in M$, we say that $F: \sd(M) \to \scr{C}$ is $\Parity_M(m)$\emph{-Segal} if it is a $\SParity_M(m)$-sheaf.
\end{ntt}

\begin{rmk}
  Our terminology for the Segal conditions differs from that introduced by Poguntke \cite{poguntke_higher_2017}. More precisely, we call even (resp. odd) $M$-Segal what is called lower (resp. upper) $M$-Segal in the literature.  We adopted this terminology to enhance the readability of our arguments in \zcref{sec:fibrancy}.
\end{rmk}

\begin{ex}
  For $M=[2]$ the even and odd Segal cubes are given by
  \[
    C_{\scr{E}_{2}}=
    \begin{tikzcd}
      012 \arrow[r] \arrow[d] & 01 \arrow[d] \\
      12 \arrow[r] & 1
    \end{tikzcd} \quad \text{and} \quad
    C_{\scr{O}_{2}}=
    \begin{tikzcd}
      012 \arrow[r] & 02
    \end{tikzcd}
  \]
  Note that these might be of different dimensions.
\end{ex}

\begin{ex}
  For $M=[3]$ the even and odd Segal cubes are given by
  \[
    C_{\scr{E}_{3}}=
    \begin{tikzcd}
      0123 \arrow[r] \arrow[d] & 023 \arrow[d] \\
      012 \arrow[r] & 02
    \end{tikzcd} \quad \quad \quad
    C_{\scr{O}_{3}}=
    \begin{tikzcd}
      0123 \arrow[r] \arrow[d] & 123 \arrow[d] \\
      013 \arrow[r] & 13
    \end{tikzcd}
  \]
\end{ex}

\subsection{Moves and lacuna diagrams}
Lacuna diagrams are an efficient bookkeeping device and visualization tool for working with Segal conditions.

\begin{const}
  Let $M$ be a finite set and $S$ a subset of $M$. We call the set $M\setminus S$ the \emph{(M-)lacunae} of $S$.

  The \emph{lacuna diagram of a sieve $\mathcal{U}$ on $M$} is a matrix where the rows (resp. columns) are indexed by elements of $\mathcal{U}$ (resp. $M$). Moreover, the entry in the row corresponding to $S\in \mathcal{U}$ and the column corresponding to $m\in M$ is a $0$ if $m$ is a lacuna of $S$ in $M$, and a $1$ otherwise.

  We typically list the elements of $M$ above the matrix, in order, if $M$ is totally ordered. Moreover, for each row, we leave the $1$-entries blank, writing $\varnothing$ for the $0$-entries.
\end{const}

\begin{rem}
  In the following sections, we will deal with lacuna diagrams of arbitrarily large dimension. We therefore only draw the part of the lacuna diagram that is relevant for the argument while leaving the rest implicit if it is clear from the context.
\end{rem}

\begin{ex}
  As an example, we consider the lacuna diagram associated to the odd and even sieve on $M=[4]$. The lacuna diagrams for the odd and even sieves $\scr{O}_{4}=\{0234,0124\}$ and $\scr{E}_{4}=\{1234,0134,0123\}$ are drawn below.

  \begin{table}[h]
    \lacuna{5}{
      $0$ & $1$ & $2$ & $3$ & $4$ \\ \hline
      & $\varnothing$ \\
      & & & $\varnothing$
    }
    $\quad$\text{and}$\quad$
    \lacuna{5}{
      $0$ & $1$ & $2$ & $3$ & $4$ \\ \hline
      $\varnothing$ \\
      && $\varnothing$ \\
      &&&& $\varnothing$
    }
  \end{table}
\end{ex}

Under the assumption that $F:\sd(M)\to \scr{C}$ is a sheaf for a collection of sieves on $M$, to be called the \emph{mediating} sieves below, the sheaf conditions for a sieve $\mathcal{U}$ can be equivalent to the sheaf condition for another sieve $\mathcal{V}$. These originate from properties of cubical limit diagrams. To establish these equivalences, we first discuss the general property of cubes and subsequently interpret them in the language of sieves.

\begin{lem}[Pasting Lemma]\label{lem:pasting}
  Let $M$ be a finite set and let $Q:[2]\times \Power(M)\rightarrow \scr{C}$ be a stacked pair of cubes. If $Q\vert_{\{1<2\} \times \Power(M)}$ is a limit cube, then $Q\vert_{\{0<1\}\times \Power(M)}$ is a limit cube if and only if the cube $Q\vert_{\{0<2\}\times \Power(M)}$ is a limit cube.
\end{lem}

\begin{const}[Pasting Move]\label{ex:pasting}
  Let $M$ be a finite set, $\mathcal{U}$ a saturated sieve on $M$, $S\in\mathcal{U}$, and $T\subset S$ such that the collection $\mathcal{V}=\{U\cap S\mid U\in \mathcal{U}\setminus \{S\}\}\cup \{T\}$ is a saturated sieve on $S$. We define a functor $[2] \times \Power(\mathcal{U}\setminus\{S\}) \to \Power(\mathcal{U}\cup \{T\})$ given by
  \[
    (0, V) \mapsto V \quad (1, V) \mapsto V\cup \{S\} \quad (2, V)\mapsto V\cup \{S,T\}.
  \]
  Then, the corresponding composition
  \[
    \begin{tikzcd}
      Q: [2] \times \Power(\mathcal{U}\setminus\{S\}) \arrow[r,hook] & \Power(\mathcal{U}\cup\{T\}) \arrow[r, "C_{(\mathcal{U}\cup \{T\})|M}"] &[2em] \sd(M)
    \end{tikzcd}
  \]
  defines a stacked pair of cubes in $\sd(M)$, whose individual cubes are given by
  \[
    Q\vert_{\{0<1\} \times \Power(\mathcal{U}\setminus\{S\})}= C_{\mathcal{U}|M},
    \quad
    Q\vert_{\{0<2\} \times \Power(\mathcal{U}\setminus\{S\})}= C_{(\mathcal{U}\circ \mathcal{V})|M},
    \quad\text{and}\quad
    Q\vert_{\{1<2\} \times \Power(\mathcal{U}\setminus\{S\})}= C_{\mathcal{V}|S},
  \]
  where $\mathcal{U} \circ \mathcal{V}$ is the sieve $(\mathcal{U}\setminus \{S\}) \cup \{T\}$ on $M$.

  By \zcref{lem:pasting}, for $F:\sd(M)\rightarrow \scr{C}$ a $\mathcal{V}$-sheaf on $S$, it is a $\mathcal{U}$-sheaf on $M$ if and only if it is a $\mathcal{U}\circ \mathcal{V}$-sheaf on $M$.
\end{const}

\begin{ex}
  An example is given by the pasting of squares. For the saturated sieve $\mathcal{U}=\{0234,0124\}$ on $[4]$ and the subsets $T=\{124\}\subset S=\{0124\}\in \mathcal{U}$, the above construction yields a stacked pair of squares
  \[
    \begin{tikzcd}
      F(01234)\arrow[r]\arrow[d] & F(0234)\arrow[d]\\
      F(0124)\arrow[r]\arrow[d] & F(024)\arrow[d]\\
      F(124)\arrow[r] & F(24)
    \end{tikzcd}
  \]
  If the bottom square is known to be pullback, then the top square is pullback if and only if the outer square is pullback.
  We illustrate the combinatorial move using lacunae diagrams below, where the left, top, and right diagrams correspond to the sieves $\mathcal{U}$, $\mathcal{V}$, and $\mathcal{U}\circ \mathcal{V}$, respectively. In effect, this move allows us to remove or add lacunae to lacuna diagrams.
  \lacunaMove{
    \lacuna{5}{
      $0$ & $1$ & $2$ & $\cancel{3}$ &$4$ \\ \hline
      & $\varnothing$ \\
      $\varnothing$
    }
  }{
    \lacuna{5}{
      $0$ & $1$ & $2$ & $3$ & $4$ \\ \hline
      & $\varnothing$ \\
      &  & & $\varnothing$
    }
  }{
    \lacuna{5}{
      $0$ & $1$ & $2$ & $3$ & $4$ \\ \hline
      & $\varnothing$ & & & \\
      $\varnothing$ &  & & $\varnothing$
    }
  }
\end{ex}

\begin{lem}[Cube Lemma, {\cite[Lemma 3.3.8]{walde_higher_2020}}]\label{lem:cube}
  Let $M$ be a finite set, $m\in M$, and $Q:\Power(M)\to \scr{C}$ an $M$-shaped cube. If the cube
  \[
    \begin{tikzcd}
      \Power(M\setminus \{m\})\arrow[r, "-\cup \{m\}"] &[2em] \Power(M) \arrow[r,"Q"] &[2em] \scr{C}
    \end{tikzcd}
  \]
  is a limit cube, then $Q$ is a limit cube if and only if the restriction $Q\vert_{\Power(M\setminus\{m\})}$ is a limit cube.
\end{lem}

\begin{const}[Cube moves] \label{const:cube_moves}
  Given a saturated sieve $\mathcal{U}$ on $M$ and an element $S\in \mathcal{U}$, the \emph{$S$-facet} $d_S(\mathcal{U})$ of  $\mathcal{U}$ is the saturated sieve on $S$ given by intersecting the other elements of $\mathcal{U}$ with $S$. The \emph{anti-$S$ facet} $d_{-S}(\mathcal{U})$ of  $\mathcal{U}$ is the saturated sieve on $M$ given by $\mathcal{U}\setminus\{S\}$.

  It follows from \zcref{lem:cube} that a $d_{S}(\mathcal{U})$-sheaf on $S$, $F:\sd(M)\to \scr{C}$, is a $\mathcal{U}$-sheaf on $M$ if and only if it is a $d_{-S}(\mathcal{U})$-sheaf on $M$.
\end{const}

\begin{ex}
  In the special case where $\mathcal{U}$ has $2$-elements, this is just the well-known statement that if one of the prongs of a pullback claw is an equivalence, then the associated square is pullback if and only if the opposite arrow is an equivalence. For example in case $\scr{U}=\scr{O}_{4}$ and $S=\{0234\}$, we are in the situation depicted in the diagram below.
  \[
    \begin{tikzcd}
      01234\arrow[r]\arrow[d,"g"'] & 0234\arrow[d,"f"]\\
      0124\arrow[r] & 024
    \end{tikzcd}
  \]
  If $f$ is an equivalence, then $g$ is an equivalence if and only if the square is pullback. We illustrate this move in terms of lacuna diagrams below.   In effect, this move allows us to remove or add rows in lacuna diagrams.
  \lacunaMove{
    \lacuna{5}{
      $0$ & $\cancel{1}$ & $2$ & $3$ & $4$ \\ \hline
      &&& $\varnothing$
    }
  }{
    \lacuna{5}{
      $0$ & $1$ & $2$ & $3$ & $4$ \\ \hline
      & $\varnothing$ \\
      &&& $\varnothing$
    }
  }{\lacuna{5}{
      $0$ & $1$ & $2$ & $3$ & $4$ \\ \hline
      &&& $\varnothing$
    }
  }
\end{ex}

As an easy consequence of \zcref{lem:cube}, we obtain the following result.

\begin{lem}\label{lem:cubefacelemma}
  Let $I$ be a finite set together with a decomposition into two disjoint subsets $I=I_{0}\sqcup I_{1}$. Let $Q:\bP(I) \rightarrow \cC$ be an $I$-shaped cube such that for every $J \in \bP_{\ast}(I_{1})$ the $I_{0}$-shaped cube
  \[
    Q(J\cup-):\bP(I_{0}) \rightarrow \cC
  \]
  is limit. Then $Q$ is a limit cube if and only if the $I_{0}$-shaped cube $Q(\varnothing\cup -)$ is.
\end{lem}
\begin{proof}
  We show by induction on $\#(J)=n$ that the cube
  \[
    Q:\bP(J \cup I_{0}) \rightarrow \cC
  \]
  is a limit cube if and only if the $I_{0}$-shaped cube $Q(\varnothing\cup-)$ is. If $J=\{j\}$ then this follows from a cube move, using the fact that the $I_{0}$-shaped cube $Q(\{j\}\cup-)$ is limit by assumption.

  For the inductive step, we consider $J \in \bP_{\ast}(I_{1})$ with $\#(J)=n$ and consider some $j\in J$. Observe that
  \[
    Q(\{j\} \cup -): \bP((J\setminus \{j\}) \cup I_{0}) \rightarrow \cC
  \]
  is a limit cube. Indeed, by inductive hypothesis, this is equivalent to
  \[
    Q(\{j\} \cup -): \bP(I_{0}) \rightarrow \cC
  \]
  being a limit cube, which is true by assumption. But now, by a cube move,
  \[
    Q:\bP(J \cup I_{0}) \rightarrow \cC
  \]
  is a limit cube if and only if
  \[
    Q:\bP((J\setminus \{j\}) \cup I_{0}) \rightarrow \cC
  \]
  is. But again, by inductive hypothesis, this is equivalent to $Q(\varnothing \cup -):\bP(I_{0}) \rightarrow \cC$ being a limit cube, and we are done.
\end{proof}

\begin{defn}
  Let $M$ be a finite set and $Q:\Power(M)\to \scr{C}$ a cube. For $m\in M$, we call $Q$ \emph{degenerate along the $m$-direction} if for all $J \in \Power_*(M)$, the map
  \[
    \begin{tikzcd}
      Q(J)\arrow[r] & Q(J\cup \{m\})
    \end{tikzcd}
  \]
  is an equivalence.
\end{defn}

\begin{lem}[{\cite[Lemma 3.3.7]{walde_higher_2020}}]\label{lem:degenerate}
  Let $M$ be a finite set, $m\in M$ and $Q:\Power(M)\to \scr{C}$ a cube that is degenerate along the $m$-direction. Then $Q$ is a limit cube if and only if the map $Q(\varnothing)\rightarrow Q(m)$ is an equivalence.
\end{lem}
\begin{proof}
  This follows from applying \zcref{lem:cubefacelemma} to the case where $I = M$, $I_{0}=\{m\}$, and $I_{1}=M\setminus \{m\}$.
\end{proof}

\begin{ex}(Degenerate cubes)\label{ex:deg-cube}
  Let $\mathcal{U}$ be a saturated sieve on $M$ and $F:\sd(M)\to \cC$ a functor. Assume that there exists an $I\in \mathcal{U}$ such that for all $\varnothing \neq \mathcal{V} \subset \mathcal{U}$ the map
  \[
    \begin{tikzcd}
      F(\cap_{J \in\mathcal{V}}J) \rightarrow F((\cap_{J\in \mathcal{V}}J) \cap I)
    \end{tikzcd}
  \]
  is an equivalence; i.e., $F$ is degenerate along the $I$-direction. Applying \zcref{lem:degenerate} to the cube $F\circ C_{\mathcal{U}}$, it follows that $F$ is a $\mathcal{U}$-sheaf if and only if the map $F(M)\to F(I)$ is an equivalence.
\end{ex}

\section{The \texorpdfstring{$(\infty,\infty)}{(∞,∞)}$-category of spans}\label{sec:fibrancy}

In this section, we construct an $(\infty, \infty)$-category $\Span_\infty(\cC)$ of spans as a decorated simplicial set for every $\infty$-category $\scr{C}$ with finite limits. The basic underlying idea of our construction of $\Span_\infty(\scr{C})$, which we will now explain in the case of a $2$-simplex, is surprisingly simple. Heuristically, the data of a $2$-simplex $\Delta^{2}_{\flat}\rightarrow \Span_\infty(\scr{C})$ is given by the following diagram in $\scr{C}$.

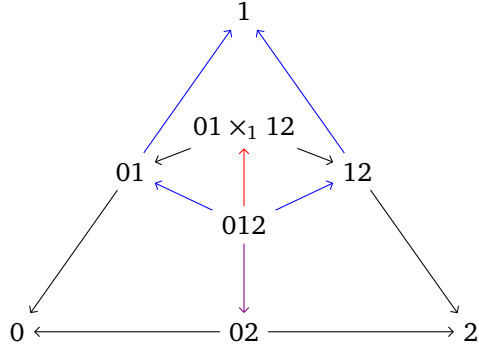
\[
  \begin{tikzpicture}
    \coordinate (v0) at (-3,0);
    \coordinate (v2) at (3,0);
    \coordinate (v1) at (0,4.24);

    \node (e) at (0,2.7) {$01\times_{1}12$};

    \coordinate (m01) at ($0.5*(v0) + 0.5*(v1)$);
    \coordinate (m02) at ($0.5*(v0) + 0.5*(v2)$);
    \coordinate (m12) at ($0.5*(v1) + 0.5*(v2)$);
    \coordinate (m012) at ($0.333*(v0) + 0.333*(v1) + 0.333*(v2)$);

    \foreach \x in {0,1,2}{
      \path (v\x) node (v\x) {$\x$};
    };

    \foreach \x/\y in {0/1,0/2,1/2}{
      \path (m\x\y) node (m\x\y) {$\x\y$};
    };

    \foreach \x/\y/\z in {0/1/2}{
      \path (m\x\y\z) node (m\x\y\z) {$\x\y\z$};
    };

    \draw[blue,->] (m01) to (v1);
    \draw[blue,->] (m12) to (v1);
    \draw[->] (m12) to (v2);
    \draw[->] (m01) to (v0);
    \draw[->] (m02) to (v0);
    \draw[->] (m02) to (v2);

    \foreach \x/\y/\z in {0/1/2}{
      \draw[blue,->] (m\x\y\z) to (m\x\y);
      \draw[->,violet] (m\x\y\z) to (m\x\z);
      \draw[blue,->] (m\x\y\z) to (m\y\z);
    }

    \draw[red,->] (m012) to (e);
    \draw[->] (e) to (m01);
    \draw[->] (e) to (m12);
  \end{tikzpicture}
\]

By the universal property of fiber products, the object $01\times_{1}12$ and the {\textcolor{red}{red}} morphism are uniquely determined by the {\textcolor{blue}{blue}} subdiagram. This means that we can, at least intuitively, remove $01\times_1 12$ from the diagram, and still have the data of a $2$-simplex. We can therefore think of $\Delta^2_\flat \to \Span_\infty(\scr{C})$ as being determined by a diagram $\sd(\Delta^2) \to \scr{C}$. Whether the $2$-morphisms represented by the unique non-degenerate $2$-simplex is invertible is then determined by whether the {\textcolor{blue}{blue}} \emph{and} {\textcolor{violet}{violet}} diagrams are limit diagrams. These are precisely the odd and even $1$-Segal conditions defined in \zcref{defn:Segal_conditions}.

\subsection{The construction and main result}
The discussion above motivates us to make the following definition.

\begin{defn}
  Consider the cosimplicial object $\sd: \Delta \to \Set_\Delta$ which sends $[n]$ to $\sd(\Delta^n)$. The corresponding nerve provides a functor $\uSpan_\infty: \Set_\Delta \to \Set_\Delta$.\footnote{This can be viewed as a variant of Kan's $\Ex$ functor. More precisely $\uSpan_\infty(X)=\Ex(X^\op)$.} More explicitly,
  \[
    \uSpan_\infty (\scr{C})_n \coloneqq \Hom_{\Set_\Delta}(\sd(\Delta^n), \scr{C}).
  \]

  For an $(\infty,1)$-category $\scr{C}$ with finite limits, we define the decorated simplicial set $\Span_\infty(\scr{C})$ with underlying simplicial set $\uSpan_\infty(\scr{C})$ by declaring $\sigma: \Delta^n \to \uSpan_\infty(\scr{C}$) to be thin if the corresponding functor $\sigma^\top: \sd(\Delta^n) \to \scr{C}$ satisfies both the odd and even $(n-1)$-Segal conditions.

  We call $\Span_\infty(\scr{C})$ the \emph{$(\infty,\infty)$-category of spans}.\footnote{This is justified by \zcref{thm:Span_infty_is_complicial} below.}
\end{defn}

\begin{defn} \label{defn:cartesian_conditions}
  For a decorated simplicial set $X_t$, we call a functor $f:\sd(X_t)\rightarrow \cC$ \emph{Cartesian} if for every marked simplex $\sigma:\Delta^{k}\rightarrow X_{t}$ (or equivalently, a morphism $\sigma: \Delta^k_+ \to X_t$ between decorated simplicial sets), the composite functor
  \[
    f\circ \sd(\sigma) : \sd(\Delta^{k})\rightarrow \sd(X_{t}) \rightarrow \cC
  \]
  satisfies the $(k-1)$-Segal condition, where $\sd: \Set_\Delta^{\dec} \to \Set_\Delta$ denotes the left Kan extension of the functor $\sd: \tDelta \to \Set_\Delta$ along the embedding $\tDelta \to \Set_\Delta^{\dec}$.\footnote{See also \zcref{subsec:model_cat_vs_infty_cat}, where basic properties of $\sd$ are established.}
\end{defn}

\begin{rmk}
  By construction, each morphism between decorated simplicial sets $\sigma: X_{t}\rightarrow \Span_{\infty}(\cC)$ corresponds, via a one-to-one correspondence, to a Cartesian functor $\sigma^\top: \sd(X_{t})\rightarrow \cC$.
\end{rmk}

Our first main result is the following theorem, whose proof will occupy the rest of this section.
\begin{thm} \label{thm:Span_infty_is_complicial}
  Let $\scr{C}$ be a category with finite limits. Then $\Span_\infty(\scr{C})$ is a complicial set.
\end{thm}
\begin{proof}[Proof (outline)]
  In the subsections below, we verify the necessary lifting properties individually, first showing that $\Span_\infty(\scr{C})$ is an inner complicial set, and then verifying the additional lifting property against $\Delta^1_\sharp\to E_\sharp$ required to apply \zcref{prop:innercomptocomp}. In particular, the lifting property against inner complicial horns is proven in \zcref{prop:inner_comp_horn}, the lifting property against inner complicial thinness extensions in \zcref{prop:inner_comp_thin}, the lifting property against complicial saturation extensions in \zcref{prop:comp_sat_ext}, and the condition of \zcref{prop:innercomptocomp} in \zcref{lem:specliftprop}.
\end{proof}

\subsection{Inner horns and inner thinness extensions}

We will show in this subsection that $\Span_{\infty}(\cC)$ satisfies the lifting properties against inner horn and thinness extensions.

\begin{ntt}
  Let $n\geq 0$ and $I\subset[n]$. We define
  \[
    I\ast \{k\} \coloneqq
    \begin{cases}
      I\cup\{k\}, & \text{if } \{k-1,k+1\} \subset I,\\
      I, & \text{else}.
    \end{cases}
  \]
  For every $k\in [n]$ and every $\mathcal{U}\subset \bP([n])$, we define $\mathcal{U}\ast\{k\} \subset \bP([n])$ to be the subset obtained from $\mathcal{U}$ by replacing each $I\in \mathcal{U}$ with $I\ast\{k\}$.
\end{ntt}

Observe that the operation $-\ast \{k\}$ defines a functor $\sd(\Delta^m) \to \sd(\Delta^m)$.

\begin{defn}
  For $m\geq 2$ and $0<k<m$, we define the category  $\overline{\sd(\Delta^{m}_k)}$ to be the full subcategory of $\sd(\Delta^m_k)$ spanned by the essential image of $-\ast \{k\}$. More explicitly, its objects consist of $S \subset [m]$ such that either
  \begin{itemize}
    \item $\{k-1,k+1\}\not\subset S$, or
    \item $\{k-1,k,k+1\}\subset S$.
  \end{itemize}
  We denote by $\overline{\sd(\Lambda_{k}^{m})} \coloneqq \sd(\Lambda^m_k) \cap \overline{\sd(\Delta^{m}_k)}$ the full subcategory spanned by subsets not containing $[m]$.\footnote{Note that $[m]_k \coloneqq [m] \setminus \{k\}$ is automatically excluded since it's not in $\overline{\sd(\Delta^m_k)}$.}

  The functor $-\ast \{k\}$ restricts to functors
  \[
    s_k: \sd(\Lambda^m_k) \to \overline{\sd(\Lambda^m_k)},
    \quad\text{and}\quad
    t_k: \sd(\Delta^m_k) \to \overline{\sd(\Delta^m_k)},
  \]
  which are both right adjoints to the inclusion functors.
\end{defn}

\begin{lem}\label{lem:horn_localization}
  The functors $s_{k}$ and $t_{k}$ are, respectively, colocalizations at the set of morphisms
  \begin{align*}
    \scr{M}_{k}&\coloneq
    \{U\cup\{k\} \rightarrow U\mid \text{ with }
    \{k-1,k+1\}\subset U \nsupset [m]_k\}, \quad\text{and} \\
    \scr{N}_{k}&\coloneq \{U\cup\{k\} \rightarrow U\mid \text{ with }
    \{k-1,k+1\}\subset U\}.
  \end{align*}
\end{lem}
\begin{proof}
  Being right adjoints to fully faithful functors, they are reflective colocalizations. It is easy to see that $s_{k}$, resp. $t_k$, sends precisely the morphisms in $\scr{M}_{k}$, resp. $\scr{N}_k$, to equivalences. Thus, they define a colocalization at $\scr{M}_k$ and $\scr{N}_k$, respectively.
\end{proof}

\begin{lem}\label{lem:k-vanishing}
  Let $m\geq 2$, $0<k<m$, and
  \[
    \sigma:\Delta^m_k \rightarrow \Span_{\infty}(\cC) \qquad (\text{resp. } \sigma:\Lambda^m_k \rightarrow \Span_{\infty}(\cC))
  \]
  a morphism of decorated simplicial sets. Then, the map $\sigma^\top(U * \{k\})\rightarrow \sigma^\top(U)$ is an equivalence for every $U\subset [m]$ (resp. for every $[m]_k \nsubset M\subset[m]$).
\end{lem}

\begin{proof}
  We will prove the statement for the $\Delta^m_k$ case, as the $\Lambda^m_k$ case is analogous.

  When $U \in \overline{\sd(\Delta^m_k)}$, there is nothing to prove since $U \ast \{k\} = U$. It remains to treat the case where $U \notin \overline{\sd(\Delta^m_k)}$, i.e., $\{k-1, k+1\} \subset U \subsetneq U \cup \{k\}$. We prove the statement by induction on $\#(U)\geq 2$.

  For $U=\{k-1,k+1\}$, this is the odd Segal condition for the thin $2$-simplex $\Delta^{\{k-1,k,k+1\}}_k \subset \Delta^m_k$. In general, since $\{k-1,k,k+1\}\subset U\cup\{k\}$, the functor $\sigma^\top$ is $\Parity_{U\cup\{k\}}(k)$-Segal. By the inductive hypothesis, the $\Parity_{U\cup\{k\}}(k)$-Segal cube is degenerate along the $U$-direction. Consequently, the conditions of \zcref{ex:deg-cube} are satisfied and the map $\sigma(U\cup\{k\})\rightarrow \sigma(U)$ is an equivalence.
\end{proof}

\begin{rmk} \label{rmk:k-vanishing_converse}
  Let $k\in [m]$ and $F: \sd(\Delta^m) \to \mathcal{C}$. The same proof as the above implies that the following are equivalent
  \begin{enumerate}
    \item for each $U \subset [m]$, $F(U \ast \{k\}) \to F(U)$ is an equivalence,
    \item for each $U \subset [m]$ containing $\{k-1,k+1\}$, $F|_{\sd(U \cup \{k\})}$ is $\Parity_{U \cup \{k\}}(k)$-Segal.
  \end{enumerate}
\end{rmk}

We now have all the necessary ingredients to show that $\Span_{\infty}(\cC)$ has the right-lifting property against inner horn and thinness extensions.

\begin{prop}[Inner horn extension] \label{prop:inner_comp_horn}
  Let $\cC$ be an $(\infty,1)$-category with finite limits. The decorated simplicial set $\Span_{\infty}(\cC)$ has the right lifting property against the complicial inner horn extensions.
\end{prop}

\begin{proof}
  Let $m\geq 2$ and $0<k<m$. We need to construct an extension $\overline{\sigma}$ given any $\sigma$ in the diagram below.
  \[
    \begin{tikzcd}
      \Lambda^m_k \arrow[r, "\sigma"] \arrow[d,hook] &
      \Span_{\infty}(\cC). \\
      \Delta^m_k \arrow[ur, dotted, swap, "\overline{\sigma}"] &
    \end{tikzcd}
  \]
  It follows from \zcref{lem:k-vanishing} that the functor $\sigma^\top:\sd(\Lambda^m_k)\rightarrow \cC$ (resp. $\overline{\sigma}^\top$) sends (resp. would send) morphisms in $\scr{M}_{k}$ (resp. $\scr{N}_k$) to equivalences. By \zcref{lem:horn_localization}, it therefore suffices to construct the following extension
  \[
    \begin{tikzcd}
      \sd(\Lambda^{m}_{k}) \arrow[r]\arrow[d,hook]   & \overline{\sd(\Lambda^m_k)} \arrow[r, "\tau"] \arrow[d,"\iota", hook] & \cC\\
      \sd(\Delta^{m}_{k}) \arrow[r,"\pi"] &  \overline{\sd(\Delta^m_k)} \arrow[ur,dotted] &
    \end{tikzcd}
  \]
  such that $\sd(\Delta^m_k) \to \scr{C}$ is Cartesian.

  Let $G$ be an extension of $\sigma^\top$ be given by the composition
  \[
    \begin{tikzcd}
      G:\sd(\Delta^{m}_{k}) \arrow[r,"\pi"] & \sd(\Delta^{m}_{k}) \arrow[r, "\iota_{\ast}\tau"] & \cC,
    \end{tikzcd}
  \]
  where $i_{\ast}\tau$ denotes the right Kan extension of $\tau$ along $\iota$. We will show that $G$ is Cartesian.

  Since any thin simplex of $\Delta^m_k$ of dimension less than $m$ is already in $\Lambda^m_k$, we only need to show that $G$ is $[m]$-Segal. Additionally, since $G$ factors through $\pi$ by construction, the restriction of $G$ to the $\Parity_{[m]}(k)$-Segal cube is degenerate along the $[m]_k$ direction and moreover, $G([m]) \to G([m]_k)$. Thus, $G$ is also $\Parity_{[m]}(k)$-Segal. It remains to show that $G$ is $(-\Parity_{[m]}(k))$-Segal.

  We will show the following stronger statement:
  \begin{itemize}
    \item[(*)] For any $I\subset [m]$, $G$ (restricted to $\sd(I)$) is a sheaf for any broad sieve $\mathcal{U}$ on $I$ such that
      \begin{enumerate}
        \item $\{k-1,k,k+1\} \subset I$,\footnote{Note that this implies that $\Delta^I$ is thin in $\Delta^m_k$.}
        \item $I_{k-1},I_{k+1}\in \mathcal{U}$, and
        \item $I_{k}\notin \mathcal{U}$
      \end{enumerate}
  \end{itemize}

  We show (*) by induction on $\#I$. The smallest simplex that satisfies the above conditions is $\Delta^{\{k-1,k,k+1\}}$ and the only choice of sieve is the even Segal cover $\mathcal{U}=\mathcal{E}_{I}=\{\{k-1,k\},\{k,k+1\}\}$. If $\Delta^{\{k-1, k, k+1\}}$ is not the top simplex, then this sieve is already known to be $G$-local. Otherwise, it is local by the construction of $G$ via right Kan extension.

  For the inductive step, we will show the following statement:
  \begin{itemize}
    \item[(**)] Assuming that (*) holds for all subsets of size smaller than $I \subset [m]$, then $G$ is a sheaf for a sieve on $I$ satisfying the condition in (*) if and only if it is also a sheaf for all such sieves on $I$.
  \end{itemize}
  This is sufficient to complete the induction step, and hence, the proof. Indeed, when $\Delta^I$ is not the top simplex, then $G$ is already known to be a $(-\Parity_{I}(k))$-sheaf. When $\Delta^I$ is the top simplex, i.e., $\Delta^I = \Delta^n$, first observe that the punctured cube associated to the broad sieve containing all subsets $[n]_l$ for $l\neq k$ is initial in $\sd(\Lambda^n_k)$. Thus, by the construction of $G$ via right Kan extension, $G$ is a sheaf with respect to the broad sieve containing all subsets $[n]_l$ for $l\neq k$.

  It remains to prove the \emph{only if} direction of (**) as the \emph{if} direction is automatic. To do this, it suffices to show that for sieves $\mathcal{U} \subsetneq \mathcal{V} = \mathcal{U} \cup \{I_l\}$ satisfying the conditions stated in (*), $G$ is a $\mathcal{U}$-sheaf if and only if it is a $\mathcal{V}$-sheaf. For the sieve $\mathcal{V}$, observe that the $I_l$-facet is a sieve of $I_l$ satisfying the assumption in (*), and thus, is $G$-local by the assumption in (**) since $I_l$ is smaller than $I$. Moreover, the anti-$I_l$ facet is precisely the sieve $\mathcal{U}$. Thus, $G$ is a $\mathcal{U}$-sheaf if and only if it is a $\mathcal{V}$-sheaf by a cube move, see \zcref{const:cube_moves}. The proof of (**) concludes.

  \lacunaMove{
    \lacuna{7}{
      $\dots$ & $k-1$ & $k$ & $k+1$ & $\dots$ & $\cancel{l}$ & $\dots$ \\ \hline
      & $\varnothing$ \\
      &&& $\varnothing$
    }
  }{\lacuna{7}{
      $\dots$ & $k-1$ & $k$ & $k+1$ & $\dots$ & $l$ & $\dots$ \\ \hline
      & $\varnothing$ \\
      &&& $\varnothing$
    }
  }{
    \lacuna{7}{
      $\dots$ & $k-1$ & $k$ & $k+1$ & $\dots$ & $l$ & $\dots$ \\ \hline
      & $\varnothing$ \\
      &&& $\varnothing$ \\
      &&&&& $\varnothing$
    }
  }
\end{proof}

\begin{rmk} \label{rmk:sheaf_broad_sieve}
  Consider $m\geq 2$, $0<k<m$, and a Cartesian map $\sigma^\top: \sd(\Delta^m_k) \to \scr{C}$. The proof above implies that for any $I \subseteq [m]$ containing $\{k-1, k, k+1\}$, the restriction of $\sigma^\top$ to $I$ is a sheaf with respect to any broad sieve $\mathcal{U}$ on $I$ such that $I_{k-1}, I_{k+1} \in \mathcal{U}$ and $I_k \notin \mathcal{U}$.
\end{rmk}

\begin{prop}\label{prop:inner_comp_thin}
  Let $\cC$ be an $(\infty,1)$-category with finite limits. Then the
  decorated simplicial set $\Span_{\infty}(\cC)$ has the right
  lifting property against inner complicial thinness
  extensions.
\end{prop}

\begin{proof}
  Let $m\geq 2$ and $0<k<m$. We need to construct an extension
  \[
    \begin{tikzcd}
      \Delta^m_{k^\prime} \arrow[r,"\sigma"] \arrow[d,hook] & \Span_\infty(\cC)\\
      \Delta^m_{k^{\prime\prime}} \arrow[ur,dotted] &
    \end{tikzcd}
  \]
  of $\sigma$. For this, it suffices to show that the simplex $\Delta^{[m]_{k}}$ is thin, or equivalently, it suffices to show that the restriction of $\sigma^\top$ to $\sd(\Delta^{[m]_k})$ satisfies the $[m]_k$-Segal conditions. We will show that it satisfies the $\Parity_{[m]_k}(k-1)$-Segal condition, the $(-\Parity_{[m]_k}(k-1))$-Segal condition being analogous.

  By \zcref{lem:k-vanishing}, we can equivalently show that $\SParity_{[m]_k}(k-1) \ast \{k\}$ is $\sigma^\top$-local on $[m]$. Consider the pasting move, where the lacuna diagram on the left represents $\SParity_{[m]_k}(k-1) \ast \{k\}$

  \lacunaMove{
    \lacuna{6}{
      $\dots$ & $\cancel{k-1}$ & $k$ & $k+1$ & $k+2$ & $\dots$ \\ \hline
      && $\varnothing$ \\
      &&&& $\varnothing$
    }
  }{
    \lacuna{6}{
      $\dots$ & $k-1$ & $k$ & $k+1$ & $k+2$ & $\dots$ \\ \hline
      & $\varnothing$ & $\varnothing$ \\
      &&&& $\varnothing$
    }
  }{
    \lacuna{6}{
      $\dots$ & $k-1$ & $k$ & $k+1$ & $k+2$ & $\dots$ \\ \hline
      & $\varnothing$ \\
      &&&& $\varnothing$
    }
  }
  \noindent and the following cube move.

  \lacunaMove{
    \lacuna{6}{
      $\dots$ & $k-1$ & $k$ & $\cancel{k+1}$ & $k+2$ & $\dots$ \\ \hline
      & $\varnothing$ \\
      &&&& $\varnothing$
    }
  }{
    \lacuna{6}{
      $\dots$ & $k-1$ & $k$ & $k+1$ & $k+2$ & $\dots$ \\ \hline
      & $\varnothing$ \\
      &&&& $\varnothing$
    }
  }{
    \lacuna{6}{
      $\dots$ & $k-1$ & $k$ & $k+1$ & $k+2$ & $\dots$ \\ \hline
      & $\varnothing$ \\
      &&& $\varnothing$ \\
      &&&& $\varnothing$
    }
  }

  The first, resp. second, mediating sieve is $\sigma^\top$-local on $[m]_{k-1}$, resp. $[m]_{k+1}$, because it is a Segal sieve on $\Delta^{[m]_{k-1}}$, resp. $\Delta^{[m]_{k+1}}$, both of which are thin simplices in $\Delta^m_{k'}$, by assumption. Moreover, the right sieve on the cube move is also $\sigma^\top$-local due to \zcref{rmk:sheaf_broad_sieve}. Thus, we are done.
\end{proof}

\subsection{Saturation extensions}
We will now show that $\Span_\infty(\scr{C})$ has the right extension property against saturation extensions.

\begin{prop}\label{prop:comp_sat_ext}
  Let $\scr{C}$ be an $(\infty, 1)$-category with finite limits. Then the decorated simplicial set $\Span_{\infty}(\scr{C})$ has the right lifting property with respect to complicial saturation extensions.
\end{prop}

\begin{proof}
  We will proceed by induction on $n$. The case $n=3$ involves a particularly tedious computation, which we will complete in \zcref{lem:3-simplex_saturation_extension}. Suppose that $X$ has the right lifting property against the complicial saturation extensions for $\ell<k$, we now aim to construct a lift for the saturation extension with $\ell=k$.

  We define an intermediate marking
  \[
    \Delta^n_{\sfmid}\coloneqq \Delta^n_{\eq} \coprod_{\Delta^{n-1}_{\eq}} \Delta^{n-1}_{\eq+}.
  \]
  Since, by assumption, $\Span_\infty(\scr{C})$ has the right lifting property against $\Delta^{n-1}_{\eq}\to \Delta^{n-1}_{\eq+}$, it suffices for us to solve the extension problem $\Delta^n_{\sfmid}\to \Delta^n_{\eq+}$. We thus consider the extension problem
  \[
    \begin{tikzcd}
      \sd(\Delta^n_{\sfmid}) \arrow[r,"f"]\arrow[d,hookrightarrow] & \scr{C} \\
      \sd(\Delta^n_{\eq+})\arrow[ur,dashed] &
    \end{tikzcd}
  \]
  in which $f$ is a sheaf for the Segal covers associated to thin simplices in $\Delta^n_{\sfmid}$.

  We define two maps of posets
  \[
    \begin{tikzcd}
      \rho, \mu:&[-3em] \sd(\Delta^{n-1}) \arrow[r] & \sd(\Delta^n)
    \end{tikzcd}
  \]
  which sends $(s_1,\ldots, s_k)$ to $(s_1+1,\ldots,s_k+1)$ and $(0,s_1+1,\ldots,s_k+1)$, respectively. Note that given a decoration $t$ of $\Delta^{n-1}$, a simplex in $\Delta^0\star \Delta^n$ is thin if and only if it is of the form $\mu(U)$ or $\rho(U)$ for a simplex $U\in t$.

  We first note that if $U$ is thin in $\Delta^{n-1}_{\eq}$, then $f$ is a sheaf associated to the Segal covers for $\rho(U)$ and $\mu(U)$. Moreover, since we can extend $f$ to $\Delta^n_{\sfmid}$ by inductive hypothesis, it follows that $f$ is a sheaf for the Segal covers associated to the thin simplices of the form $\rho(U)$ when $U$ is a thin simplex of $\Delta^{n-1}_{\eq+}$. It remains to show that $f$ is a sheaf for the Segal covers associated to the thin simplices of the form $\mu(U)$ when $U$ is a thin simplex of $\Delta^{n-1}_{\eq+}$.

  For a simplex $U$ of $\Delta^{n-1}$, we fix notation for three covers associated to $\mu(U)$:
  \begin{itemize}
    \item We denote the Segal cover not containing the $0$-lacuna by $C_\mu^1(U)$. Note that this is the image of one Segal cover of $U$ in $\Delta^{n-1}$ under the map $\mu$;
    \item We denote the Segal cover which does contain the $0$-lacuna by $C_\mu^2(U)$. Note that this cover involves elements of the images of both $\rho$ and $\mu$;
    \item We denote by $C_\mu^{2-}(U)$ the cover obtained from $C_\mu^2(U)$ by deleting the $0$-lacuna. Note that this is the image of the other Segal cover of $U$ in $\Delta^{n-1}$ under the map $\mu$.
  \end{itemize}
  We thus need to show that for a thin simplex $U$ of $\Delta^{n-1}_{\eq+}$, $f$ is a sheaf for $C_\mu^1(U)$ and $C_\mu^2(U)$.

  By the cube move, if $f$ is a sheaf for the covers associated to $\rho(U)$, then $f$ is a sheaf for $C_{\mu}^2(U)$ if and only if $f$ is a sheaf for $C_{\mu}^{2-}(U)$. We can thus conclude that $f\circ \mu$ is a sheaf for the Segal covers associated to the thin simplices of $\Delta^{n-1}_{\eq}$. Using the inductive hypothesis, it follows that $f\circ \mu$ is a sheaf for all of the Segal covers associated to the thin simplices of $\Delta^{n-1}_{\eq+}$.

  Thus, for $U$ a thin simplex of $\Delta^{n-1}_{\eq+}$, $f$ is a sheaf for $C_{\mu}^1(U)$ and $C_{\mu}^{2-}(U)$. However, since $f$ is also a sheaf for the Segal covers associated to $\rho(U)$, the cube move tells us that the latter of these conditions is equivalent to $f$ being a sheaf with respect to $C_\mu^2(U)$, completing the proof.
\end{proof}

\subsection{Subdivisions of the $3$-simplex} \label{subsec:3-simplex_saturation_extension}

We now have only two lifting problems to analyze the lowest dimensional saturation extension and the extension along $\Delta^1_\sharp\to E_\sharp$. Both of our arguments will make reference to \zcref{fig:bary3simp}.

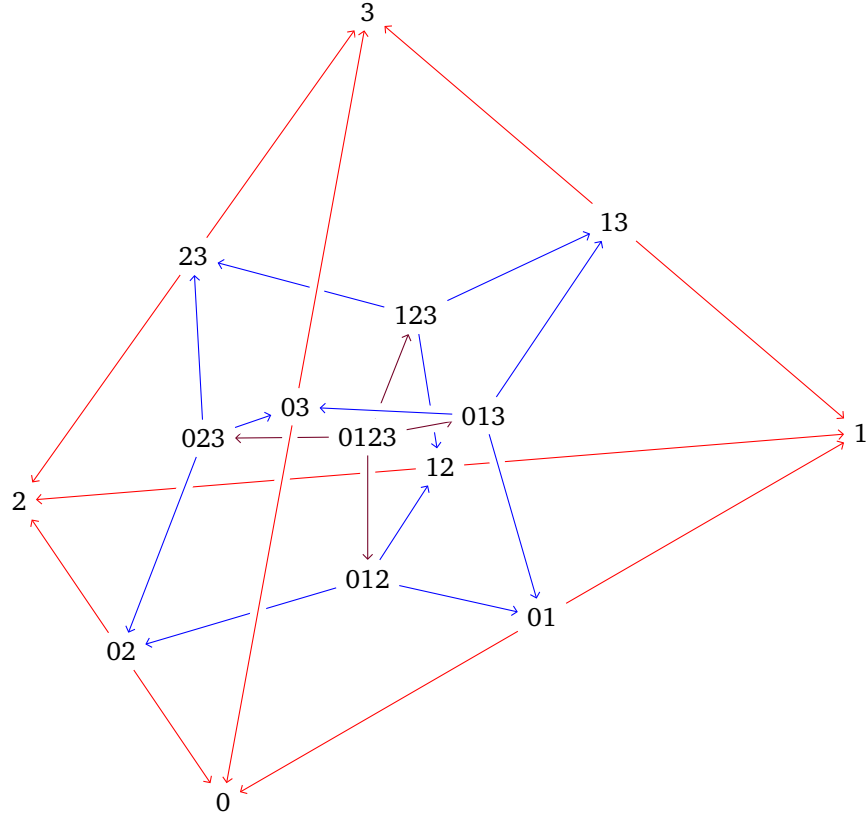
\begin{figure}[htb!]
  \begin{tikzpicture}[x=1cm,y=1.25cm,z=-0.5cm,rotate around y=10, scale=6]
    Coordinates of the original vertices
    \coordinate (v0) at (0,0,1);
    \coordinate (v1) at  (0.866,0,-0.5);
    \coordinate (v2) at  (-0.866,0,-0.5);
    \coordinate (v3) at (0,1,0);
    \coordinate (m01) at ($0.5*(v0) + 0.5*(v1)$);
    \coordinate (m02) at ($0.5*(v0) + 0.5*(v2)$);
    \coordinate (m03) at ($0.5*(v0) + 0.5*(v3)$);
    \coordinate (m12) at ($0.5*(v1) + 0.5*(v2)$);
    \coordinate (m13) at ($0.5*(v1) + 0.5*(v3)$);
    \coordinate (m23) at ($0.5*(v2) + 0.5*(v3)$);
    \coordinate (m012) at ($0.333*(v0) + 0.333*(v1) + 0.333*(v2)$);
    \coordinate (m013) at ($0.333*(v0) + 0.333*(v1) + 0.333*(v3)$);
    \coordinate (m023) at ($0.333*(v0) + 0.333*(v2) + 0.333*(v3)$);
    \coordinate (m123) at ($0.333*(v1) + 0.333*(v2) + 0.333*(v3)$);

    \coordinate (m0123) at ($0.25*(v0) + 0.25*(v1) + 0.25*(v2) + 0.25*(v3)$);

    \foreach \x in {0,1,2,3}{
      \path (v\x) node (v\x) {$\x$};
    };

    \foreach \x/\y in {0/1,0/2,0/3,1/2,1/3,2/3}{
      \path (m\x\y) node (m\x\y) {$\x\y$};
    };

    \foreach \x/\y/\z in {0/1/2, 0/1/3, 1/2/3,0/2/3}{
      \path (m\x\y\z) node (m\x\y\z) {$\x\y\z$};
    };

    \path (m0123) node (m0123) {$0123$};

    \foreach \x/\y in {0/1,0/2,1/2,1/3,2/3}{
      \draw[red,->] (m\x\y) to (v\x);
      \draw[red,->] (m\x\y) to (v\y);
    };

    \foreach \x/\y/\z in {0/1/2, 1/2/3}{
      \draw[blue,->] (m\x\y\z) to (m\x\y);
      \draw[blue,->] (m\x\y\z) to (m\x\z);
      \draw[blue,->] (m\x\y\z) to (m\y\z);
    }

    \draw[line width=2mm, white] (m0123) to (m012);
    \draw[line width=2mm, white] (m0123) to (m013);
    \draw[line width=2mm, white] (m0123) to (m023);
    \draw[line width=2mm, white] (m0123) to (m123);

    \draw[purple!60!black,->] (m0123) to (m012);
    \draw[purple!60!black,->] (m0123) to (m013);
    \draw[purple!60!black,->] (m0123) to (m023);
    \draw[purple!60!black,->] (m0123) to (m123);

    \foreach \x/\y/\z in {0/1/3, 0/2/3}{
      \draw[line width=2mm, white] (m\x\y\z) to (m\x\y);
      \draw[line width=2mm, white] (m\x\y\z) to (m\x\z);
      \draw[line width=2mm, white] (m\x\y\z) to (m\y\z);

      \draw[blue,->] (m\x\y\z) to (m\x\y);
      \draw[blue,->] (m\x\y\z) to (m\x\z);
      \draw[blue,->] (m\x\y\z) to (m\y\z);
    }

    \foreach \x/\y in {0/3}{
      \draw[line width=2mm, white]  (m\x\y) to (v\x);
      \draw[line width=2mm, white]  (m\x\y) to (v\y);

      \draw[red,->] (m\x\y) to (v\x);
      \draw[red,->] (m\x\y) to (v\y);
    };

  \end{tikzpicture}
  \caption{The barycentrically subdivided $3$-simplex $\sd(\Delta^3)$. }\label{fig:bary3simp}
\end{figure}

\begin{lem} \label{lem:3-simplex_saturation_extension}
  The decorated simplicial set $\Span_\infty(\scr{C})$ has the right lifting property against the morphism $\Delta^3_{\eq}\to \Delta^3_{\sharp}$.
\end{lem}

\begin{proof}
  Given a Cartesian functor $f: \sd(\Delta^3_{\eq})\to \mathcal{C}$, we will extend it to a functor $\sd(\Delta^3_\sharp) \to \mathcal{C}$. It is easy to see that it suffices to show that $f$ sends all morphisms to equivalences.\footnote{In fact, such an extension necessarily sends all morphisms to equivalences; see \zcref{lem:max-marked=grpd_subdiv}.} We strongly advise that the reader wishing to check this proof mark the morphisms verified to be sent to equivalences on the above diagram as they go through the proof.
  \begin{enumerate}
    \item Note that, by hypothesis, $02\to 0$, $02\to 2$, $13\to 1$, $13\to 3$, $023\to 03$, $013\to 03$, $012\to 02$, and $123\to 13$ are sent to equivalences.
    \item Since the squares
      \[
        \begin{tikzcd}
          023 \arrow[r]\arrow[d] & 23\arrow[d]\\
          02\arrow[r] & 2
        \end{tikzcd} \quad
        \begin{tikzcd}
          013 \arrow[r]\arrow[d] & 13\arrow[d]\\
          01\arrow[r] & 1
        \end{tikzcd}
      \]
      are sent to pullbacks, the morphisms $023\to 23$ and $013\to 01$ are sent to pullbacks of equivalences, hence equivalences.
    \item Since the squares
      \[
        \begin{tikzcd}
          0123 \arrow[r]\arrow[d] & 023\arrow[d]\\
          012\arrow[r] & 02
        \end{tikzcd} \quad
        \begin{tikzcd}
          0123 \arrow[r]\arrow[d] & 123\arrow[d]\\
          013\arrow[r] & 13
        \end{tikzcd}
      \]
      are sent to pullbacks, it follows that $0123\to 023$ and $0123\to 013$ are sent to equivalences.
    \item The image  of $012\to 01$ has right-inverse  the image of $0123\to 012$ and left-inverse the image of $01\to 0$, and so all three of these morphisms are sent to invertible morphisms.
    \item By 2-out-of-3, it then follows that $03\to 0$, $023\to 02$, and $23\to 2$ are sent to invertible morphisms.
    \item Since the square
      \[
        \begin{tikzcd}
          123 \arrow[r]\arrow[d] & 23\arrow[d]\\
          12\arrow[r] & 2
        \end{tikzcd}
      \]
      is sent to a pullback and $23\to 2$ is sent to an equivalence, the morphism $123\to 12$ is sent to an equivalence.
    \item By 2-out-of-3, it then follows that the morphism $12\to 1$ is sent to an equivalence.
    \item Since the morphisms $23\to 2$ and $123\to 12$ are equivalences, the square
      \[
        \begin{tikzcd}
          123 \arrow[r]\arrow[d] & 23\arrow[d]\\
          12 \arrow[r] & 2
        \end{tikzcd}
      \]
      identifies the morphisms $123 \to 23$ and $12 \to 2$. The former has a left inverse $23 \to 3$ and the latter has a right inverse $012\to 12$, and so all of these morphisms are sent to equivalences.
    \item The remainder of the morphisms are sent to equivalences by iterated 2-out-of-3.
  \end{enumerate}
\end{proof}

\begin{lem}\label{lem:specliftprop}
  The decorated simplicial set $\Span_\infty(\scr{C})$ has the right lifting property against the morphism $\iota:\Delta^1_\sharp\to E_\sharp$
\end{lem}

\begin{proof}
  A Cartesian functor $\sd(\Delta^1_\sharp) \to \mathcal{C}$ necessarily sends all morphisms to invertible morphisms. Thus, we are given a morphism $f: \sd(\Delta^1) \to \mathcal{C}^\simeq$, and we need to extend it to $\sd(E)$.

  To do this, we define a map
  \[
    \begin{tikzcd}
      p:&[-3em]\sd(\Delta^3) \arrow[r] & \sd(\Delta^1)
    \end{tikzcd}
  \]
  by specifying the preimages of objects as follows.
  \begin{align*}
    p^{-1}(1) & = \{1,13,3\}\\
    p^{-1}(0) & = \{0,02,2\}\\
    p^{-1}(01) &= \{01,03,23,012,013,023,123,0123\}
  \end{align*}
  It is straightforward to verify that $p$ is a map of posets and that $p$ descends to the quotient $\sd(E)$. As such, $p\circ \iota=\id_{\Delta^1}$, and so we can simply define the desired extension to be $f\circ p$.
\end{proof}

\subsection{Model categories and \texorpdfstring{$\infty$}{∞}-categories} \label{subsec:model_cat_vs_infty_cat}

Our definition thus far has been purely model categorical. We have defined $\Span_\infty(\scr{C})$ as an object in the model category $\Set_{\Delta}^{\dec}$, and then showed that it is fibrant. However, as we will see in the next section, our definition has a much more elegant formulation internal to the $(\infty, 1)$-category of $(\infty, \infty)$-categories. To facilitate this, we here develop some background on the subdivision and Cartesian functors we will need to easily formulate our theorems in the sequel.

\begin{defn} \label{defn:sd_SD_LKE}
  Recall that we denote by $\psi:\tDelta\to \Set_\Delta^{\dec}$ the inclusion, and by $\xi:\tDelta\to \Cat_\infty$ its composition with the localization $\Set_{\Delta}^{\dec}\to \Cat_\infty$. Viewing $\sd$ as a functor $\tDelta\to \Set_\Delta$, we define
  \[
    \begin{tikzcd}
      \sd: &[-3em] \Set_\Delta^{\dec} \arrow[r] & \Set_\Delta
    \end{tikzcd}
  \]
  to be the left Kan extensions of $\sd$ along $\psi$. Viewing $\sd$ as a functor $\tDelta\to \Cat_1$, we define
  \[
    \begin{tikzcd}
      \SD: &[-3em] \Cat_\infty \arrow[r] & \Cat_1
    \end{tikzcd}
  \]
  to be the left Kan extensions of $\sd$ along $\xi$.
\end{defn}

The main results of this paper are proven by working with $\tDelta$ and then extending to $\Set_{\Delta}^{\dec}$ and $\Cat_\infty$ using density. As such, each of our main results comes in two flavors, one model independent, and one model categorical. The key component of this  is to understand the interplay of the Segal conditions with the Kan extensions defined above.

\begin{defn} \label{defn:more_general_cart_conditions}
  For a decorated simplicial set $K_t$, we call a functor
  \[
    \begin{tikzcd}
      f: &[-3em] \sd(K) \arrow[r] & \scr{C}
    \end{tikzcd}
  \]
  \emph{Cartesian} if, for every $\iota:\Delta^n_+\to K_t$, the composite $f\circ \sd(\iota)$ satisfies the $\Delta^n$-Segal conditions. We denote by
  \[
    \Map^{\cart}(\sd(K_t),\scr{C})\subset \Map(\sd(K),\scr{C}),
    \qquad \text{resp.} \quad
    \Fun^{\cart}(\sd(K_t),\scr{C}) \subset \Fun(\sd(K),\scr{C}),
  \]
  the full subspace, resp. subcategory, spanned by the Cartesian functors.

  For an $(\infty,\infty)$-category $\scr{D}$, we call a functor
  \[
    \begin{tikzcd}
      f: &[-3em] \SD(\scr{D}) \arrow[r] & \scr{C}
    \end{tikzcd}
  \]
  \emph{Cartesian} if for every $\mathbb{O}^n\to \scr{D}$ which sends the unique non-invertible $n$-morphism of $\mathbb{O}^n$ to an equivalence, the corresponding composite
  \[
    \begin{tikzcd}
      \sd(\Delta^n)\arrow[r] & \SD(\scr{D}) \arrow[r,"f"] & \scr{C}
    \end{tikzcd}
  \]
  satisfies the $\Delta^n$-Segal conditions. We denote by
  \[
    \Map^{\cart}(\SD(\scr{D}),\scr{C})\subset \Map(\SD(\scr{D}),\scr{C}),
    \qquad \text{resp.} \quad
    \Fun^{\cart}(\SD(\scr{D}),\scr{C}) \subset \Fun(\SD(\scr{D}),\scr{C})
  \]
  the full subspace, resp. subcategory, spanned by the Cartesian functors.
\end{defn}

To use spaces of Cartesian maps effectively, we need some lemmas on their interaction with limits.

\begin{lem} \label{lem:cart_map_SD_lim_sd}
  The canonical map
  \[
    \begin{tikzcd}
      \Map^{\cart}(\SD(\scr{D}),\scr{C}) \arrow[r] & \lim_{\tDelta/\scr{D}} \Map^{\cart}(\sd(\Delta^n),\scr{C})
    \end{tikzcd}
  \]
  is an equivalence.
\end{lem}

\begin{proof}
  Since the $\Map^{\cart}(\SD(\scr{D}),\scr{C})\subset \Map(\SD(\scr{D}),\scr{C})$ are full subcategories, and we have defined $\Map^{\cart}(\SD(\scr{D}),\scr{C})$ such that a map $[0]\to \Map(\SD(\scr{D}),\scr{C})$ factors through $\Map^{\cart}(\SD(\scr{D}),\scr{C})$ if and only if the corresponding cone over $\Map(\sd(\Delta^{-}),\scr{C})$ factors through $\Map^{\cart}(\sd(\Delta^{-}),\scr{C})$, the result follows by universal property.
\end{proof}

A completely analogous argument shows

\begin{lem}
  The canonical map
  \[
    \begin{tikzcd}
      \Fun^{\cart}(\SD(\scr{D}),\scr{C}) \arrow[r] & \lim_{\tDelta/\scr{D}} \Fun^{\cart}(\sd(\Delta^n),\scr{C})
    \end{tikzcd}
  \]
  is an equivalence.
\end{lem}

We then turn to limit preservation in the second variable.

\begin{defn}
  Define $\Cat_1^{\lex}\subset \Cat_1$ to be the locally full subcategory on the categories with finite limits and the finite limit-preserving functors.
\end{defn}

\begin{lem}[{\cite[Proposition 2.1.12]{bunke_controlled_2024}}] \label{lem:limits_of_cats_with_limits}
  The subcategory $\Cat_1^{\lex}\subset \Cat_1$ is closed under limits.
\end{lem}

\begin{lem}\label{lem:cartmaps_pres_limits}
  Let $\scr{I}\in \Cat_1$ and let $f: \scr{I}\to \Cat_1^{\lex}$ be a functor. Then for any $(\infty,\infty)$-category $\scr{D}$, the canonical maps
  \[
    \begin{tikzcd}
      \Map^{\cart}(\SD(\scr{D}),\lim_{\scr{I}} f(-)) \arrow[r] & \lim_{\scr{I}} \Map^{\cart} (\SD(\scr{D}),f(-))
    \end{tikzcd}
  \]
  and
  \[
    \begin{tikzcd}
      \Fun^{\cart}(\SD(\scr{D}),\lim_{\scr{I}} f(-)) \arrow[r] & \lim_{\scr{I}} \Fun^{\cart} (\SD(\scr{D}),f(-))
    \end{tikzcd}
  \]
  are equivalences.
\end{lem}

\begin{proof}
  This follows directly from \zcref{lem:limits_of_cats_with_limits}.
\end{proof}

\section{The space of maps into \texorpdfstring{$\Span_\infty(\scr{C})$}{Span_∞(𝒞)}}
\label{sec:mapping_space_into_Span}

The main goal of this section is to take the model categorical definition of $\Span_\infty(\scr{C})$ and upgrade it to a universal property characterizing $\Span_\infty(\scr{C})$ in terms of the functor it represents.

\subsection{Main results}

\subsubsection{The space of maps into $\Span_\infty(\mathcal{C})$}

While the definition of $\Span_\infty(\mathcal{C})$ a priori depends on the presentation of $\mathcal{C}$ as a quasi-category, we show, in this section, that $\Span_\infty(\mathcal{C})$ in fact is model-independent by characterizing the space of maps into it.

\begin{thm}\label{thm:SpanUP_I_internal}
  Let $\scr{C}$ be an $(\infty,1)$-category with finite limits and $\scr{D}$ an $(\infty,\infty)$-category. There is an equivalence
  \[
    \Map_{\Cat_\infty}(\scr{D},\Span_\infty(\scr{C}))\simeq \Map^{\cart}_{\Cat_1}(\SD(\scr{D}),\scr{C})
  \]
  natural in $\scr{D}\in \Cat_\infty$.
\end{thm}

As discussed at the end of the preceding section, this also has a model-categorical analogue.

\begin{thm}\label{thm:SpanUP_I_model}
  Let $\scr{C}$ be an $(\infty,1)$-category with finite limits. There is a equivalence
  \[
    \Map_{\Set_\Delta^\dec}(K_t,\Span_\infty(\scr{C}))\simeq \Map^{\cart}_{\Cat_1}(\sd(K_t),\scr{C})
  \]
  natural in $K_t\in \Set_\Delta^{\dec}$.
\end{thm}

Both of these propositions follow immediately from the following proposition.

\begin{prop}\label{prop:mappingUP_prop}
  Let $\scr{C}\in \Cat_1^{\lex}$. There is an equivalence, natural in $\tDelta$,
  \[
    \Map_{\Set_\Delta^\dec}(\Delta^n_t,\Span_\infty(\scr{C}))\simeq \Map^{\cart}_{\Cat_1}(\sd(\Delta^n_t),\scr{C}).
  \]
\end{prop}

We will first briefly derive \zcref{thm:SpanUP_I_internal,thm:SpanUP_I_model} from \zcref{prop:mappingUP_prop}. The proof of \zcref{prop:mappingUP_prop} will occupy the rest of this section.

\begin{proof}[Proof (of \zcref{thm:SpanUP_I_internal})]
  The desired equivalence follows from the following string of equivalences
  \begin{align*}
    \Map(\scr{D},\Span_\infty(\scr{C})) & \simeq \Map(\colim_{\Delta^n_t \in \tDelta/\scr{D}}\Oriental^n_t,\Span_\infty(\scr{C})) \tag{\zcref{prop:tDelta_dense}} \\
    & \simeq \lim_{\Delta^n_t \in \tDelta/\scr{D}}\Map(\Delta^n_t,\Span_\infty(\scr{C}))\\
    & \simeq \lim_{\Delta^n_t \in \tDelta/\scr{D}}\Map^{\cart}(\sd(\Delta^n_t),\scr{C}) \tag{\zcref{prop:mappingUP_prop}}\\
    & \simeq \Map^{\cart}(\SD(\scr{D}),\scr{C}). \tag{\zcref{lem:cart_map_SD_lim_sd}}
  \end{align*}
\end{proof}

\begin{proof}[Proof (of \zcref{thm:SpanUP_I_model})]
  By \zcref{cor:K_t_hocolim}, an arbitrary decorated simplicial set $K_t$ is the homotopy colimit over the $1$-categorical slice $\tDelta_{/ K_t}$. Moreover, by \zcref{cor:sd_hocolim}, $\sd(K_t)$ is the homotopy colimit (in the Joyal model structure) over the $1$-categorical slice $\tDelta_{/K_t}$. Thus, the natural equivalence
  \[
    \Map(\Delta^n_t,\Span_\infty(\scr{C}))\simeq \Map^{\cart}(\sd(\Delta^n_t),\scr{C})
  \]
  induces an equivalence
  \[
    \Map(K_t,\Span_\infty(\scr{C}))\simeq \Map^{\cart}(\sd(K_t),\scr{C})
  \]
  on homotopy limits, as desired.
\end{proof}

\subsubsection{Immediate consequences}
We will now collect immediate consequences of the main results stated above.

\begin{defn}
  Let $(X,tX)$ be a decorated simplicial set. Its opposite is the decorated simplicial set $(X^{\op},tX)$.
\end{defn}

\begin{prop}\label{prop:span is self dual}
  Let $\cC$ be an $(\infty,1)$-category with finite limits. There exists an equivalence of $\infty$-categories
  \[
    \Span_{\infty}(\cC)^{\op}\simeq \Span_{\infty}(\cC).
  \]
\end{prop}
\begin{proof}
  It suffices to construct, for every $n\geq 0$ and every decorated $n$-simplex $\Delta^{n}_{t}$ a natural equivalence
  \[
    \Map(\Delta^{n}_{t},\Span_{\infty}(\cC)^{\op}) \simeq \Map(\Delta^{n,\op}_{t},\Span_{\infty}(\cC)) \simeq \Map(\Delta^{n}_{t},\Span(\cC)).
  \]
  By \zcref{prop:tDelta_dense,prop:mappingUP_prop}, it suffices to construct a natural equivalence
  \[
    \alpha:\sd(\Delta^{n}_{t}) \rightarrow \sd(\Delta^{n,\op}_{t}).
  \]
  We define $\alpha$ as the extension to $\sd(\Delta^{n})$ of the following map
  \[
    [n] \rightarrow [n]^{\op}, \quad i \mapsto n-i.
  \]
\end{proof}

\begin{cor} \label{cor:upgrade_Span_infty_functor}
  The assignment $\scr{C}\mapsto \Span_\infty(\scr{C})$ extends to a limit preserving functor
  \[
    \begin{tikzcd}
      \Cat_1^{\lex} \arrow[r] & \Cat_\infty.
    \end{tikzcd}
  \]
\end{cor}

\begin{proof}
  There is a functor
  \[
    \begin{tikzcd}
      \Map^{\cart}(\SD(-),-): &[-3em] \Cat_\infty^\op\times \Cat_1^{\lex} \arrow[r] & \Cat_0
    \end{tikzcd}
  \]
  which preserves limits in the second variable by \zcref{lem:cartmaps_pres_limits}. By \zcref{thm:SpanUP_I_internal}, for every $\scr{C}$ in $\Cat_1^{\lex}$, the resulting presheaf $\Cat_\infty^\op\to \Cat_0$ is represented by $\Span_\infty$. Thus, the adjoint
  \[
    \begin{tikzcd}
      \Cat_1^{\lex}\arrow[r] & \Fun(\Cat_\infty^\op,\Cat_0)
    \end{tikzcd}
  \]
  factors through $\Cat_\infty$, and since the Yoneda embedding preserve limits, it follows that the resulting functor also does.
\end{proof}

\begin{cor} \label{cor:monoidal_structure_Span_infty}
  Let $\scr{C}\in \Cat_1^{\lex}$. Then the Cartesian symmetric monoidal structure on $\scr{C}$ induces a symmetric monoidal structure on $\Span_\infty(\mathcal{C})$.
\end{cor}

\begin{proof}
  The Cartesian symmetric monoidal structure on $\scr{C}$ is a symmetric monoid in $\Cat_1^{\lex}$ and limit-preserving functors preserve symmetric monoids.
\end{proof}

\subsubsection{The proof strategy for \zcref{prop:mappingUP_prop}}
Let $\mathcal{C}$ be an $(\infty, 1)$-category with finite limits, presented as a fibrant object in the Joyal model structure. The groupoid $\Map^{\cart}_{\Cat_1}(\sd(\Delta^n_t), \mathcal{C})$ is presented as
\[
  \Hom_{\Set_\Delta}^{\cart}(\sd(\Delta^n_t) \times -, \mathcal{C}): \Delta^\op \to \Set,
\]
where, by abuse of notation, $f$ satisfies the $\cart$ condition if and only if $f$ is Cartesian on the first variable and sends all morphisms on the second variable to equivalences. This can be expressed conveniently in terms of homomorphisms between \emph{marked} simplicial sets
\[
  \Hom^{\cart}_{\Set_\Delta^+}(\sd(\Delta^n_t) \times (-)_\sharp, \mathcal{C}): \Delta^\op \to \Set, \teq \label{eq:unwinding_Map_sd}
\]
where we marked all the equivalences in $\mathcal{C}$.

On the other hand, the groupoid $\Map_{\Cat_\infty}(\Delta^n_t, \Span_\infty(\mathcal{C}))$ can be presented as
\[
  \Hom_{\Set_\Delta^\dec}(\Delta^n_t \times (-)_\sharp, \Span_\infty(\mathcal{C})): \Delta^\op \to \Set,
\]
which, by definition, is \emph{isomorphic} to
\[
  \Hom_{\Set_\Delta}^{\cart}(\sd(\Delta^n_t \times (-)_\sharp), \mathcal{C}): \Delta^\op \to \Set.
\]
Marking all the antidiagonal morphisms in $\sd(\Delta^n_t \times (-)_\sharp)$, i.e., those which are sent to equivalences by any Cartesian functor, the above is \emph{isomorphic} to
\[
  \Hom_{\Set_\Delta^+}^{\cart}(\sd^+(\Delta^n_t \times (-)_\sharp), \mathcal{C}): \Delta^\op \to \Set. \teq \label{eq:unwinding_Map_Span}
\]

The equivalence between \zcref{eq:unwinding_Map_sd,eq:unwinding_Map_Span} essentially follows from a natural equivalence of marked simplicial sets $\sd^+(\Delta^n_t \times (-)_\sharp) \simeq \sd(\Delta^n_t) \times (-)_\sharp$ (\zcref{cor:mapping_equiv_without_limits}) and carefully tracking the Cartesian conditions (\zcref{lem:mapping_equiv_with_limits}).

We will carry out the details of this argument in the remainder of this section. Most of the technical work centers around establishing the said equivalence, which goes through an important localization (\zcref{prop:product_localization})
\[
  \sd(\Delta^n \times \Delta^m) \to \sd(\Delta^n) \times \sd(\Delta^m). \teq\label{eq:key_localization}
\]

\subsection{Subdivision, Gray products, and antidiagonal morphisms}
The localization \zcref{eq:key_localization} factors through an intermediate localization
\[
  \begin{tikzcd}
    \sd(\Delta^n \times \Delta^m) \ar{r}{L_{n,m}} & R^{n,m} \ar{r} & \sd(\Delta^n) \times \sd(\Delta^m).
  \end{tikzcd}
\]
In this subsection, we will do the preparatory work necessary to define $L_{n,m}$ in \zcref{const:L-functor}. The proof that $L_{n,m}$, resp. the composite, is indeed a localization is carried out in \zcref{subsec:the_localization_L}, resp. \zcref{subsec:product_localization}.

\begin{const}
  We define maps
  \[
    \begin{tikzcd}[row sep=0em]
      \fv: &[-3em] \sd(\Delta^n)\arrow[r] & \Delta^n \\
      & U \arrow[r,mapsto] & \min(U)
    \end{tikzcd}
  \]
  and
  \[
    \begin{tikzcd}[row sep=0em]
      \max: &[-3em] \Delta^n \arrow[r] & \sd(\Delta^n)\\
      & k \arrow[r,mapsto] & {[k,n]}
    \end{tikzcd}
  \]
  Note that while the former are natural, the latter are not. By Kan extension, we obtain natural maps $\fv: \sd(K)\to K$ for every simplicial set $K$.
\end{const}

\begin{lem}\label{lem:sdEx}
  The natural maps $\fv$ are natural weak homotopy equivalences. Consequently, the adjunction
  \[
    \begin{tikzcd}
      \sd: &[-3em] \Set_\Delta \arrow[r,shift left] & \Set_\Delta\arrow[l,shift left] &[-3em] :\uSpan_\infty
    \end{tikzcd}
  \]
  is a Quillen self-equivalence of the Kan--Quillen model structure naturally equivalent to the identity.
\end{lem}

\begin{proof}
  The proof is effectively the same as~\cite[Proposition 10.1.4]{dyckerhoff_HigherSegalSpaces_2019}. Equivalently, this is the $\sd\dashv \Ex$ adjunction from classical homotopy theory.
\end{proof}

Our next goal is to identify $\sd(\Delta^n)\times \sd(\Delta^m)$ as a localization of $\sd(\Delta^n\times \Delta^m)$ at morphisms which every Cartesian functor must send to weak equivalences. This will allow us to relate functors in $\Span_\infty(\scr{C})$ to spans of functors into $\scr{C}$. Because we will reuse much of this work later in this paper, we rephrase our work in terms of Gray products $\Delta^n_t\otimes \Delta^m_s$, where $t, s \in \{\flat, +\}$. Our key result in this regard is \zcref{thm:the_localization}, which we will use repeatedly in later sections when we generalize the functoriality of our universal property.

\begin{defn}\label{defn:sd+_antidiag}
  We call a morphism $\nu$ in $\sd(X_t)$ \emph{antidiagonal} if for any finitely complete $(\infty,1)$-category $\scr{C}$ and any Cartesian functor
  \[
    \begin{tikzcd}
      f: &[-3em] \sd(X_t) \arrow[r] &\scr{C}
    \end{tikzcd}
  \]
  $f(\nu)$ is an equivalence. We denote by $\sd^+(X_t)$ the marked simplicial set given by marking the antidiagonal morphisms. It is immediate that the formation $X_t \mapsto \sd^+(X_t)$ is functorial, and we obtain a functor
  \[
    \sd^+: \Set_\Delta^\dec \to \Set_\Delta^+.
  \]
\end{defn}

\begin{rmk}
  By definition, we have an \emph{isomorphism} of sets
  \[
    \Hom^{\cart,+}(\sd^+(K_t),\scr{C})\cong \Hom^{\cart}(\sd(K_t),\scr{C}),
  \]
  where $\Hom^{\cart,+}$ denotes the set of Cartesian functors which respect markings, and $\mathcal{C}$ is viewed as a marked simplicial set by marking the equivalences in $\mathcal{C}$.
\end{rmk}

\begin{lem}\label{lem:max-marked=grpd_subdiv}
  The marked simplicial set $\sd^+(\Delta^n_\sharp)$ is maximally marked. Consequently, every morphism of marked simplicial sets $\sd^+(\Delta^n_\sharp)\to \scr{C}$ is Cartesian.
\end{lem}

\begin{proof}
  The second statement follows from the first because a cube consisting entirely of equivalences is always a limit cube.

  We prove the first statement by induction on $n$. It is immediate when $n=0$ and $n=1$. For general $n$, it then follows from the inductive hypothesis that every morphism in $\sd(\partial\Delta^n_\sharp)$ is antidiagonal. Since the antidiagonal marking is closed under composition, it then suffices to check that the morphisms $[n]\to [n]\setminus\{k\}$ are antidiagonal. Each such morphism is part of a Segal cube with cone point $[n]$. By induction, the Segal cube minus the cone point is a diagram consisting of only equivalences, so the image of $[n]$ under a Cartesian functor is the limit over a contractible indexing category of a diagram consisting only of equivalences. Thus, all of the morphisms in the Segal cone must be sent to equivalences, completing the proof.
\end{proof}

\begin{lem} \label{lem:sd^+(Delta^n_flat)=sd(Delta^n)_flat}
  The marked simplicial set $\sd^+(\Delta^n_\flat)$ is minimally marked. Consequently, every morphism of marked simplicial sets $\sd^+(\Delta^n_\flat)\to \scr{C}$ is Cartesian.
\end{lem}

\begin{proof}
  Immediate from the definitions.
\end{proof}

\begin{rmk} \label{rmk:paths_in_grid}
  Each object of $\sd(\Delta^n\times \Delta^m)$ is a chain of pairs $(\vec{u},\vec{v}) \coloneqq \{(u_0,v_0)<\cdots<(u_k,v_k)\}$, and thus can be viewed uniquely as a path in an $n\times m$ oriented grid. These paths are ordered by simple refinement of chains. We will, in general draw these grids with $n$ along the $y$ axis (increasing in the downward direction) and $m$ along the $x$ axis (increasing in the rightwards direction).

  Below are some examples in $\sd(\Delta^2\times \Delta^3)$.
  \[
    \begin{tikzpicture}
      \drawGrid{4}{3}
      \begin{scope}[on background layer]
        \drawLine{3pt}{red} (00.center) -- (01.center) -- (13.center);
        \drawLine{3pt}{blue} (11.center) -- (12.center) -- (22.center);
      \end{scope}
  \end{tikzpicture}\]
\end{rmk}

\begin{rmk} \label{rmk:antidiagonal_in_product_vs_gray}
  By the functoriality of $\sd^+$, if a morphism is antidiagonal in $\sd(\Delta^n_t\otimes \Delta^m_s)$ or $\sd(\Delta^m_s\otimes \Delta^n_t)$, then it is antidiagonal in $\sd(\Delta^n_t\times \Delta^m_s)$.
\end{rmk}

We now identify an important class of antidiagonal morphisms in $\sd(\Delta^n_\flat \otimes \Delta^m_\flat)$. For the remainder of this subsection, unless otherwise specified, the default marking on any simplex $\Delta^n$ is that of $\Delta^n_\flat$.

\begin{lem}\label{lem:marked_in_flat_gray}
  A non-degenerate $k$-simplex $(\vec{u},\vec{v})$ in $\Delta^n\otimes\Delta^m$ is thin if and only if there are indices $0\leq \ell_1<\ell_2<k$ such that $u_{\ell_1}=u_{\ell_1+1}$ and $v_{\ell_2}=v_{\ell_2+1}$.
\end{lem}

\begin{proof}
  First assume that such indices $\ell_{1},\ell_{2}$ exist. Then it follows that for every $p\leq \ell_{1}$ the simplex $\sqcup^{2}_{p,q}\vec{v} = (v_{p}, \dots , v_{\ell_{2}}, v_{\ell_{2}+1}=v_{\ell_2},\dots,v_{k})$ is degenerate and for every $p > \ell_{1}$ the simplex $\sqcup^{1}_{p,q}\vec{u} = (u_{1}, \dots, u_{\ell_{1}}, u_{\ell_{1}+1}=u_{\ell_{1}}, \dots, u_{p})$ is degenerate.

  On the other hand, assume that the non-degenerate simplex $(\vec{u},\vec{v})\in\Delta^{m}_{\flat} \otimes\Delta^{n}_{\flat}$ is thin. Since both $\vec{u}$ and $\vec{v}$ are degenerate, there exists a smallest index $0\leq \ell_{1}<k$ such that $u_{\ell_{1}}=u_{\ell_{1}+1}$ and a largest index $0\leq \ell_{2}<k$ such that $v_{\ell_{2}}=v_{\ell_{2}+1}$. It follows that $\sqcup_{\ell_{1},n-\ell_{1}}^{1}\vec{u}$ and $\sqcup_{\ell_{2}+1,n-\ell_{2}-1}^{2}\vec{v}$ are non-degenerate. The first implies that $\sqcup^2_{\ell_1, n-\ell_1}$ is degenerate, and hence, $\ell_1 < \ell_2 + 1$, or equivalently, $\ell_1 \leq \ell_2$. Since $(\vec{u}, \vec{v})$ is non-degenerate, $\ell_1 \neq \ell_2$. Thus, $\ell_1 < \ell_2$, completing the proof.
\end{proof}

\begin{defn} \label{defn:corner_collapse}
  There is a particular, special case of the lemma: the case where we can choose $\ell_2=\ell_1+1$. The corresponding path in the grid then contains an upper right angle. In this case, we call each triple $\{\ell_1,\ell_1+1,\ell_1+2\}$ such that in \zcref{lem:marked_in_flat_gray} we can take $\ell_2=\ell_1+1$ an \emph{upper right corner} of $(\vec{u},\vec{v})$.

  If $(\vec{s},\vec{t})$ is the face of $\sigma$ obtained by deleting the middle points of upper right corners, we refer to the resulting morphism  $(\vec{u},\vec{v})\to (\vec{s},\vec{t})$ in $\sd(\Delta^n \otimes \Delta^m)$ or any composite of such morphisms as a \emph{corner collapse}. We denote the set of all corner collapse morphisms in $\sd(\Delta^n\otimes \Delta^m)$ by $C_{n,m}$.

  Let $(\vec{u},\vec{v})$ be a simplex containing an upper right corner with indices $\{\ell_1,\ell_1+1,\ell_1+2\}$. We call the Segal condition on $(\vec{u},\vec{v})$ which contains the face removing the middle vertex of the corner a \emph{corner Segal condition}.

  Reversing the roles of $\vec{u}$ and $\vec{v}$,\footnote{or, in view of \zcref{lem:marked_in_flat_gray}, using the opposite convention when forming the Gray tensor product} we can similarly define the notions \emph{lower left corner}, the corresponding \emph{dual corner collapse}, and \emph{dual corner Segal condition}.
\end{defn}

In this new terminology, \zcref{lem:marked_in_flat_gray,lem:k-vanishing} together imply the following
\begin{cor} \label{cor:corner_collapse_antidiagonal}
  The corner collapse morphisms are antidiagonal.
\end{cor}

\begin{lem} \label{lem:corner_Segal_vs_corner_collapse}
  Let $\scr{C}$ be an $(\infty,1)$-category, considered as a marked $\infty$-category by marking the equivalences. Let $\sd(\Delta^n \otimes \Delta^m)^C$ be the marked simplicial set obtained by marking the corner collapses $C_{n,m}$ in $\sd(\Delta^n\otimes \Delta^m)$.
  \begin{enumerate}
    \item \label{item:lem_corner_Segal_vs_corner_collapse:functoriality} The formation $\Delta^n \times \Delta^m \mapsto \sd(\Delta^n\otimes \Delta^m)^C$ is functorial and we obtain a functor $\Delta \times \Delta \to \Set_\Delta^+$.
    \item \label{item:lem_corner_Segal_vs_corner_collapse:equivalence} A map $\sd(\Delta^n\otimes \Delta^m)\to \scr{C}$ satisfies the corner Segal conditions if and only if it respects markings $C_{n,m}$.
  \end{enumerate}
\end{lem}

\begin{proof}
  Item \ref{item:lem_corner_Segal_vs_corner_collapse:functoriality} is checked on face and degeneracy maps. Item \ref{item:lem_corner_Segal_vs_corner_collapse:equivalence} is a direct consequence of \zcref{rmk:k-vanishing_converse}.
\end{proof}

\begin{defn}
  We define an order on pairs $(u,v)\in \Delta^n\times \Delta^m$, which we call the \emph{twisted arrow order}, by declaring $(u_1,v_1)\leq_{\Tw} (u_2,v_2)$ if and only if $u_1\geq u_2$ and $v_1\leq v_2$. Note that, as this is the twisted arrow category of a poset, it is itself a poset. Visually, $(u,v)<_{\Tw}(s,t)$ implies that $(u,v)$ lies in the right-angled cone in the grid to the lower-left of $(s,t)$, including the edges which share one coordinate with $(s,t)$, but excluding the point $(s,t)$ itself.

  We will additionally write $(u,v)\ll_{\Tw} (s,t)$ if $(u,v)<_{\Tw}(s,t)$ \emph{and} neither coordinate of $(u,v)$ agrees with the corresponding coordinate of $(s,t)$.

  We say that a path $(\vec{u},\vec{v})$ lies \emph{below-left} of a path $(\vec{s},\vec{t})$ if for each $(s_i,t_i)$ in the latter path and each $(u_j,v_j)$ in the former, $(s_i,t_i)\not\ll_{\Tw} (u_j,v_j)$. We then also say that $(\vec{s},\vec{t})$ lies \emph{above-right} of $(\vec{u},\vec{v})$. Note that this is not a partial order on the set of paths.
\end{defn}

\begin{defn}
  We call a non-degenerate $k$-simplex (or path) $(\vec{u},\vec{v})$ in $\Delta^n\times \Delta^m$ \emph{rectilinear} if, for every index $0\leq j<k$, either $u_k=u_{k+1}$ or $v_k=v_{k+1}$.

  We define a partial order $\preceq$ on the set of paths in $\Delta^n\times \Delta^m$ by requiring $(\vec{u},\vec{v})\preceq (\vec{s},\vec{t})$ if and only if
  \begin{enumerate}
    \item $(\vec{u},\vec{v})$ is below-left of $(\vec{s},\vec{t})$,
    \item \emph{and} the minimal subgrid $\Delta^{U}\times \Delta^V\subset \Delta^n\times \Delta^m$ containing $(\vec{u},\vec{v})$ also contains $(\vec{s},\vec{t})$.
  \end{enumerate}

  We denote the subposet of \emph{rectilinear} paths equipped with the order $\preceq$ by $R^{n,m}$.
\end{defn}

\begin{lem}
  The relation $\preceq$ is a partial order on $R^{n,m}$.
\end{lem}

\begin{proof}
  Reflexivity and transitivity are immediate from the definitions. We will now prove anti-symmetry.

  Suppose that $(\vec{u},\vec{v})$ and $(\vec{s},\vec{t})$ are rectilinear paths with $(\vec{s},\vec{t})\preceq (\vec{u},\vec{v})$ and $(\vec{u},\vec{v})\preceq (\vec{s},\vec{t})$. Then the second condition in the definition of $\preceq$ implies that the minimal subgrids defined by $(\vec{s},\vec{t})$ and $(\vec{u},\vec{v})$ are the same. We thus may assume without loss of generality that $(\vec{s},\vec{t})$ and $(\vec{u},\vec{v})$ are maximal paths in $\Delta^n\times \Delta^m$, that is, the minimal subgrid containing (either of) them is $\Delta^n\times \Delta^m$ itself.

  Let $k$ be the last index such that $(u_k,v_k)=(s_k,t_k)$. Suppose by way of contradiction that $k\neq n+m$. Then there are two possibilities.
  \begin{itemize}
    \item If $(u_{k+1},v_{k+1})=(u_k+1,v_k)$ and $(s_{k+1},t_{k+1})=(s_k,t_{k}+1)$, then $(u_{k+1},v_{k+1})\ll_{\Tw}(s_{k+1},t_{k+1})$, which contradicts the assumption that $(\vec{s},\vec{t})\preceq (\vec{u},\vec{v})$.
    \item If $(u_{k+1},v_{k+1})=(u_k,v_k+1)$ and $(s_{k+1},t_{k+1})=(s_k+1,t_{k})$, then $(s_{k+1},t_{k+1})\ll_{\Tw}(u_{k+1},v_{k+1})$, another contradiction.
  \end{itemize}
  Thus, the two paths agree at every point, completing the proof.
\end{proof}

\begin{rmk}
  There is another, more visually intuitive description of the partial order $\preceq$ on $R^{n,m}$. Interpret $(\vec{u},\vec{v})$ and $(\vec{s},\vec{t})$ as piecewise linear paths $\mu$ and $\alpha$ in the rectangle $[0,n]\times[0,m]$ in the obvious way, and call a positive slope ray originating from the bottom or left edges of the grid an \emph{upward ray}. We say that $(\vec{u},\vec{v})\preceq (\vec{s},\vec{t})$ if and only if every upward ray which intersects $\alpha$ does not intersect $\alpha$ first. For example, in the diagram below, {\color{red} Red Path} $\preceq$ {\color{blue}Blue Path}.
  \[
    \begin{tikzpicture}
      \drawGrid{4}{3}

      \begin{scope}[on background layer]
        \drawLine{3pt}{red}  (10.center) -- (11.center) -- (21.center) -- (22.center) -- (23.center);
        \drawLine{3pt}{blue} (11.center) -- (12.center) -- (13.center);
      \end{scope}
    \end{tikzpicture}
  \]
\end{rmk}

\begin{rmk}
  Note that if a rectilinear path $(\vec{s},\vec{t})$ is a subsimplex of $(\vec{u},\vec{v})$, then $(\vec{u},\vec{v})\preceq (\vec{s},\vec{t})$.
\end{rmk}

\begin{const}\label{const:L-functor}
  We define natural functors
  \[
    \begin{tikzcd}
      L_{n,m}:&[-3em] \sd(\Delta^n\times \Delta^m) \arrow[r] & R^{n,m}
    \end{tikzcd}
  \]
  as follows. Given a simplex $(\vec{u},\vec{v})$, $L_{n,m}(\vec{u},\vec{v})$ is the rectilinear path  constructed by adding, for each index $i$ such that $u_i<u_{i+1}$ and $v_i<v_{i+1}$, the intermediate vertex $(u_i,v_{i+1})$. Note that, by construction, the minimal grids of $(\vec{u},\vec{v})$ and $L_{n,m}(\vec{u},\vec{v})$ coincide.
\end{const}

\subsection{The localization \texorpdfstring{$\sd(\Delta^n\times \Delta^m) \to R^{n,m}$}{sd(Δⁿ×Δᵐ) → Rⁿ,ᵐ}} \label{subsec:the_localization_L}

The following key technical result will be used throughout the paper.

\begin{thm}\label{thm:the_localization}
  The functors
  \[
    \begin{tikzcd}
      L_{n,m}:&[-3em] \sd(\Delta^n\times \Delta^m) \arrow[r] & R^{n,m}
    \end{tikzcd}
  \]
  are localizations of $(\infty,1)$-categories at the corner collapses.
\end{thm}

We first note that (1) every corner collapse is sent to an identity by $L_{n,m}$, and (2) the canonical morphism $L_{n,m}(\vec{u},\vec{v})\to(\vec{u},\vec{v})$ always lies in $C_{n,m}$. Thus, if an inclusion $(\vec{s},\vec{t})\subset(\vec{u},\vec{v}) $ becomes an isomorphism (i.e., equality) after applying $L_{n,m}$ there is a zigzag
\[
  \begin{tikzcd}
    (\vec{u},\vec{v}) & \arrow[l] L_{n,m}(\vec{u},\vec{v}) = L_{n,m}(\vec{s},\vec{t}) \ar{r} & (\vec{s},\vec{t})
  \end{tikzcd}
\]
of corner collapses which necessarily commutes with the original inclusion. Thus, if $L_{n,m}$ is a localization, it is a localization at $C_{n,m}$, and so it suffices to show that $L_{n,m}$ is a localization.

\begin{defn}
  We define a functor $\fr:R^{n,m}\to \Delta^n\times \Delta^m$ by sending a rectilinear path to its first vertex. It is immediate that the diagram
  \[
    \begin{tikzcd}
      \sd(\Delta^n\times \Delta^m) \arrow[dr,"\fv"']\arrow[rr,"L_{n,m}"] & & R^{n,m}\arrow[dl,"\fr"]\\
      & \Delta^n\times \Delta^m
    \end{tikzcd}
  \]
  commutes.

  We define functors $\lr:R^{n,m}\to (\Delta^n\times \Delta^m)^\op$ and $\lv:\sd(\Delta^n\times \Delta^m)\to (\Delta^n\times \Delta^m)^\op$ analogously by taking the last vertex.
\end{defn}

\begin{ntt}
  We denote elements of $\Delta^n\times \Delta^m$ by boldface lowercase Latin letters $\mathbf{a}$, $\mathbf{b}$, etc. We will denote the components of the pair $\mathbf{a}$ by $a_1$ and $a_2$.
\end{ntt}

\begin{lem}
  The functor $(\fv,\lv): \sd(\Delta^n \times \Delta^m) \to (\Delta^n\times \Delta^m)\times (\Delta^n\times \Delta^m)^\op$ is a locally Cartesian fibration.
\end{lem}

\begin{proof}
  We must show that for every arrow $\Delta^1\to (\Delta^n\times \Delta^m)\times (\Delta^n\times \Delta^m)^\op$, corresponding to a map $(\mathbf{f}_0, \mathbf{l}_0) \to (\mathbf{f}_1, \mathbf{l}_1)$, the corresponding pullback is a Cartesian fibration. If $\mathbf{f}_1 \nleq \mathbf{l}_1$, then the fiber over $(\mathbf{f}_1,\mathbf{l}_1)$ is empty, so the statement is trivial. If $\mathbf{f}_1 \leq \mathbf{l}_1$, then for any $S\in \sd(\Delta^n\times \Delta^m)$ with $\mathbf{f}_1=\min(S)$ and $\mathbf{l}_1=\max(S)$, the Cartesian lift of $(\mathbf{t}_0,\mathbf{t}_0)\to (\mathbf{s}_0,\mathbf{s}_0)$ ending in $S$ is $\{\mathbf{f}_0\}\cup S\cup\{\mathbf{l}_0\}\to S$. We leave the verification that this morphism is Cartesian to the reader.
\end{proof}

\begin{lem}
  The functor $(\fr,\lr): R^{n,m} \to (\Delta^n\times \Delta^m)\times (\Delta^n\times \Delta^m)^\op$ is a locally Cartesian fibration.
\end{lem}

\begin{proof}
  The proof is virtually identical to the above, replacing $\{\mathbf{f}_0\}\cup S\cup\{\mathbf{l}_0\}$ with $L_{n,m}(\{\mathbf{f}_0\}\cup S\cup\{\mathbf{l}_0\})$.
\end{proof}

\begin{prop}\label{prop:fiberwise_localization_locally_Cartesian}
  Let $p:\scr{C}\to \scr{B}$ and $q: \scr{D}\to \scr{B}$ be locally Cartesian fibrations. Let $f:\scr{C}\to \scr{D}$ preserve locally Cartesian morphisms. If $f$ is a localization on every fiber, then $f$ is a localization.
\end{prop}

\begin{proof}
  Let $\mathbb{B}$ denote the \emph{lax morphism classifier} of $\scr{B}$ from~\cite[Construction 2.1.12]{abellan_straightening_2024}. Then pullback along the natural functor  $\scr{B}\to \mathbb{B}$ defines an equivalence of $\infty$-categories between $(0,1)$-Cartesian fibrations over $\mathbb{B}$ and locally Cartesian fibrations over $\scr{B}$. Denote by $\mathbb{C}\to \mathbb{B}$ and $\mathbb{D}\to \mathbb{B}$ the $(0,1)$-fibrations corresponding to $p$ and $q$.

  Given an $(\infty,1)$-category $\scr{E}$, we then have natural equivalences
  \begin{align*}
    \Map(\scr{D},\scr{E}) & \simeq \Map_{\scr{B}}(\scr{D},\scr{E}\times \scr{B}) \\
    & \simeq  \Map_{\mathbb{B}}(\mathbb{D},\scr{E}\times \mathbb{B}) \\
    & \simeq \Map^{\lax}(\St_{\mathbb{B}}(\mathbb{D}),\constant_\scr{E})
  \end{align*}
  Here, $\St_{\mathbb{B}}(\mathbb{D})$ is the straightening of the $(0, 1)$-fibration $\mathbb{D}\to \mathbb{B}$, $\constant_\scr{E}$ is the constant functor with value $\scr{E}$, and $\Map^{\lax}(-,-)$ is the space of lax natural transformations between two lax functors. Moreover, the second equivalence follows using~\cite[Theorem 3.3.6]{abellan_straightening_2024} and the fact that Gray hom-objects into a fibrant object in $\Set_\Delta^{\sc}$ are stable under composition along a trivial cofibration, and the third equivalence follows from~\cite[Theorem 3.3.6]{abellan_straightening_2024}.

  Since, by hypothesis, the natural transformation $\St_{\mathbb{B}}(\mathbb{C})\to \St_{\mathbb{B}}(\mathbb{D})$ is an objectwise localization, it follows that composition with this transformation induces an equivalence
  \[
    \Map^{\lax}(\St_{\mathbb{B}}(\mathbb{D}),\constant_\scr{E})\simeq \Map^{\lax,\dagger}(\St_{\mathbb{B}}(\mathbb{C}),\constant_\scr{E})
  \]
  where the superscript $\dagger$ indicates the subspace spanned by those natural transformations whose components send the localized morphisms to equivalences. Unstraightening, the same natural equivalences as before show us that here  is a natural equivalence
  \[
    \Map^{\lax,\dagger}(\St_{\mathbb{B}}(\mathbb{C}),\constant_\scr{E})\simeq \Map^\dagger(\scr{C},\scr{E})
  \]
  where $\Map^\dagger(\scr{C},\scr{E})$ denotes the space of maps which send the localized morphisms of the fibers to equivalences. Thus, we have a natural equivalence
  \[
    \Map(\scr{D},\scr{E})\simeq \Map^\dagger(\scr{C},\scr{E})
  \]
  induced by composition with $f$, and so $f$ is a localization.
\end{proof}

Consequently, it will suffice for us to show that the maps
\[
  \begin{tikzcd}
    \sd(\Delta^n\times \Delta^m)_{\mathbf{f},\mathbf{l}} \arrow[r] & R^{n,m}_{\mathbf{f},\mathbf{l}}
  \end{tikzcd}
\]
induced by $L_{n,m}$ on fibers are localizations. Without loss of generality, we can consider the case $\mathbf{t}=(0,0)$ and $\mathbf{s}=(n,m)$.

\begin{rmk}\label{rmk:univ_prop_of_gray_loc}
  This fiberwise reformulation of our desired localization also allows us to understand the statement more systematically and intuitively. We note that $\sd(\Delta^n\times \Delta^m)_{a,b}$ is isomorphic to the hom-category $\mathfrak{C}[\Delta^n\times \Delta^m](a,b)$ in the rigidification of the product, and $R^{n,m}_{a,b}$ is isomorphic to the hom-category from $a$ to $b$ in the strict $2$-categorical gray product $\tr_2^g(\mathbb{O}^n)\otimes_2^s \tr_2^g(\mathbb{O}^m_2)$ of $2$-truncated orientals.

  Folklorically, $\mathfrak{C}[\scr{P}]$ for $\scr{P}$ a poset is the lax morphism classifier of $\scr{P}$. Thus, our claimed localization is equivalently the statement that the Gray product $\tr_2^g(\mathbb{O}^n)\otimes \tr_2^g(\mathbb{O}^m_2)$ is a homwise localization of the lax morphism classifier of $\mathbb{O}^n_2\otimes_2 \mathbb{O}^m_2$ --- precisely the universal property of the Gray product in terms of cubical lax functors. Intuitively speaking, this is the reason we should expect $L_{n,m}$ to be a localization.
\end{rmk}

\begin{ntt}
  For ease of notation, write $\mathbb{O}^n_2\coloneqq\tr_2^g(\mathbb{O}^n)$ for the strict $2$-truncation of the $n$\textsuperscript{th} oriental.
\end{ntt}

\begin{lem}\label{lem:equiv_gray_loc_det_by_nat}
  Let
  \begin{tikzcd}
    F,G:&[-3em] \mathfrak{C}^{\sc}[\Delta^n_\flat \otimes \Delta^m_\flat] \arrow[r] & \mathbb{O}^n_2\otimes^2\mathbb{O}^m_2
  \end{tikzcd}
  be maps of marked-simplicial-set-enriched categories, natural in $\Delta\times\Delta$. If $F$ and $G$ agree on  $\mathfrak{C}^{\sc}[\Delta^1_\flat \otimes \Delta^1_\flat]$, then $F=G$.
\end{lem}

\begin{proof}
  It is immediate that $F$ and $G$ must agree on objects.

  For fixed $n$ and $m$, every path $\mathbf{a}<\mathbf{c}_1<\cdots<\mathbf{c}_{k-1}<\mathbf{b}$ from $\mathbf{a}$ to $\mathbf{b}$ in $\mathfrak{C}^{\sc}[\Delta^n_\flat \otimes \Delta^m_\flat]$ can be written as a composite of the paths $\mathbf{c}_i<\mathbf{c}_{i+1}$.  By the commutativity of the diagrams
  \[
    \begin{tikzcd}
      \mathfrak{C}^{\sc}[\Delta^n_\flat \otimes \Delta^m_\flat] \arrow[r,"F_{n,m}"] & \mathbb{O}^n_2\otimes^2\mathbb{O}^m_2 \\
      \mathfrak{C}^{\sc}[\Delta^1_\flat \otimes \Delta^1_\flat] \arrow[r,"F_{1,1}"']\arrow[u,"{\{\mathbf{c}_i,\mathbf{c}_{i+1}\}}"] & \mathbb{O}^1_2\otimes^2\mathbb{O}^1_2\arrow[u,"{\{\mathbf{c}_i,\mathbf{c}_{i+1}\}}"']
    \end{tikzcd}
    \qquad
    \begin{tikzcd}
      \mathfrak{C}^{\sc}[\Delta^n_\flat \otimes \Delta^m_\flat] \arrow[r,"G_{n,m}"] & \mathbb{O}^n_2\otimes^2\mathbb{O}^m_2 \\
      \mathfrak{C}^{\sc}[\Delta^1_\flat \otimes \Delta^1_\flat] \arrow[r,"G_{1,1}"']\arrow[u,"{\{\mathbf{c}_i,\mathbf{c}_{i+1}\}}"] & \mathbb{O}^1_2\otimes^2\mathbb{O}^1_2\arrow[u,"{\{\mathbf{c}_i,\mathbf{c}_{i+1}\}}"']
    \end{tikzcd}
  \]
  It follows that $F$ and $G$ agree on the morphisms $\mathbf{c}_i<\mathbf{c}_j$, and thus on all morphisms, completing the proof.
\end{proof}

\begin{lem}\label{lem:unique_equiv_gray_loc}
  There is a unique equivalence of marked-simplicial-set-enriched categories
  \[
    \begin{tikzcd}
      \mathfrak{C}^{\sc}[\Delta^1_\flat \otimes \Delta^1_\flat] \arrow[r] & \mathbb{O}^1_2\otimes^2\mathbb{O}^1_2.
    \end{tikzcd}
  \]
\end{lem}

\begin{proof}
  By considering homotopy (strict) $1$-categories, it is immediate that any such equivalence must act as the identity of sets of objects. The only mapping space which is not terminal is the mapping space from $(0,0)$ to $(1,1)$. In this case, it is easy to verify that the only equivalence of marked simplicial sets compatible with the composition is the first-vertex map
  \[
    \begin{tikzcd}
      \sd(\Delta^1) \arrow[r] & \Delta^1
    \end{tikzcd}
  \]
  completing the proof.
\end{proof}

\begin{prop}\label{prop:L_nat_equiv}
  The functors
  \[
    \begin{tikzcd}
      \sd(\Delta^n\times \Delta^m)_{\mathbf{f},\mathbf{l}} \arrow[r] & R^{n,m}_{\mathbf{f},\mathbf{l}}
    \end{tikzcd}
  \]
  induced by $L_{n,m}$ are natural weak equivalences.
\end{prop}

\begin{proof}
  By \zcref{cor:gray_prod_nat_equiv_scaled}, there is a natural weak equivalence
  \[
    \begin{tikzcd}
      W:&[-3em] \mathfrak{C}^{\sc}[\Delta^n_\flat \otimes \Delta^m_\flat] \arrow[r] & \mathbb{O}^n_2\otimes^2\mathbb{O}^m_2.
    \end{tikzcd}
  \]
  However, the maps
  \[
    \begin{tikzcd}
      \sd(\Delta^n\times \Delta^m)_{\mathbf{f},\mathbf{l}} \arrow[r] & R^{n,m}_{\mathbf{f},\mathbf{l}}
    \end{tikzcd}
  \]
  induced by $L_{n,m}$ also piece together into such a natural transformation
  \[
    \begin{tikzcd}
      L:&[-3em] \mathfrak{C}^{\sc}[\Delta^n_\flat \otimes \Delta^m_\flat] \arrow[r] & \mathbb{O}^n_2\otimes^2\mathbb{O}^m_2.
    \end{tikzcd}
  \]

  It is easy to check that $L$ induces an equivalence on $\mathfrak{C}^{\sc}[\Delta^1_\flat \otimes \Delta^1_\flat]$, and so by \zcref{lem:unique_equiv_gray_loc}, $L$ and $W$ agree on $\mathfrak{C}^{\sc}[\Delta^1_\flat \otimes \Delta^1_\flat]$. Thus, by \zcref{lem:equiv_gray_loc_det_by_nat}, $L\cong W$, and so $L$ is a natural weak equivalence, completing the proof.
\end{proof}

\begin{proof}[Proof (of \zcref{thm:the_localization})]
  We have already argued that if $L_{n,m}$ is a localization, then it is a localization at the corner collapses, so it suffices to show that it is a localization. However, \zcref{prop:fiberwise_localization_locally_Cartesian} shows that it suffices to show that $L_{n,m}$ is a fiberwise localization, which is true due to \zcref{prop:L_nat_equiv}.
\end{proof}

\subsection{The localization \texorpdfstring{$\sd(\Delta^n\times \Delta^m) \to \sd(\Delta^n)\times \sd(\Delta^m)$}{sd(Δⁿ×Δᵐ) → sd(Δⁿ)×sd(Δᵐ)}} \label{subsec:product_localization}

We are now ready to prove the localization \zcref{eq:key_localization} and extract important equivalences of marked simplicial sets. Note that, by \zcref{cor:corner_collapse_antidiagonal}, every corner collapse in $\sd(\Delta^m\otimes \Delta^n)$ is antidiagonal. Thus, by \zcref{rmk:antidiagonal_in_product_vs_gray}, in $\sd^+(\Delta^n\times\Delta^m)$ both the corner collapses and their duals are marked.

\begin{prop}\label{prop:product_localization}
  The universal map $Z_{n,m}:\sd(\Delta^n\times \Delta^m)\to \sd(\Delta^n)\times \sd(\Delta^m)$ is a localization at the collection of corner collapses and dual corner collapses.
\end{prop}

\begin{proof}
  As above, it suffices to show that $Z_{n,m}$ is a localization.

  On elements, $Z_{n,m}(\vec{u}, \vec{v})$ consists of the underlying nondegenerate simplex of $\vec{u}$ and the underlying nondegenerate simplex of $\vec{v}$. Thus, $Z_{n,m}$ sends all corner collapses to identities, and so, it factors through
  \[
    \begin{tikzcd}
      \gamma_{n,m}: R^{n,m} \arrow[r] & \sd(\Delta^n) \times \sd(\Delta^m)
    \end{tikzcd}
  \]
  by \zcref{thm:the_localization}. It remains to show that $\gamma_{n,m}$ is a localization. However, this functor admits a fully faithful left adjoint $\ell_{n,m}$, which sends $(\vec{u},\vec{v})$ to the path
  \[
    (u_0,v_0)<\cdots<(u_k,v_0)< (u_k,v_1)<\cdots<(u_k,v_\ell).
  \]
  It is in fact a coreflective colocalization, and hence, a localization.
\end{proof}

\begin{rmk} \label{rmk:L-shaped}
  In terms of paths (see \zcref{rmk:paths_in_grid}), the essential image of $\ell_{n,m}$ consists of rectilinear paths which start with a vertical line, followed by a horizontal line. We refer to these as \emph{L-shaped} paths.
\end{rmk}

\begin{cor}
  For any decorated simplices $\Delta^n_t$ and $\Delta^m_s$, the universal map
  \[
    \begin{tikzcd}
      Z_{n,m}: &[-3em] \sd^+(\Delta^n_t\times\Delta^m_s)\arrow[r] & (\sd(\Delta^n)\times \sd(\Delta^m))^{(t,s)}
    \end{tikzcd}
  \]
  is an equivalence of marked simplicial sets, where the superscript $(t,s)$ indicates the minimal marking such that $Z_{n,m}$ preserves markings.
\end{cor}

We now specialize to the case where the marking on $\Delta^m$ is maximal.

\begin{lem}\label{lem:antidiagonal_on_flat+sharp}
  A morphism of $\sd^+(\Delta^n_\flat\times \Delta^m_\sharp)$ is marked if and only if its image in $\sd(\Delta^n)$ is degenerate. In particular, the marking $(\sd(\Delta^n)\times \sd(\Delta^m))^{\flat,\sharp}$
  coincides with the marking $\sd^+(\Delta^n_\flat)\times \sd^+(\Delta^m_\sharp)$.
\end{lem}

\begin{proof}
  By the functoriality of $\sd^+$, we have a morphism of marked simplicial sets $\sd^+(\Delta^n_\flat \times \Delta^m_\sharp) \to \sd^+(\Delta^n_\flat)$. But since $\sd^+(\Delta^n_\flat) = \sd(\Delta^n)_\flat$, by \zcref{lem:sd^+(Delta^n_flat)=sd(Delta^n)_flat}, the \emph{only if} direction follows.

  The \emph{if} direction follows from the claim that for $(\vec{u},\vec{v})\in \sd(\Delta^n\times \Delta^m)$ such that $\vec{u}$ is degenerate at $i$, the morphism $(\vec{u},\vec{v})\to d_i(\vec{u},\vec{v})$ is antidiagonal. But this claim follows from \zcref{rmk:k-vanishing_converse}, completing the proof.

\end{proof}

\begin{cor}\label{cor:antidiagonal_in_t_sharp}
  A morphism of $\sd^+(\Delta^n_t\times \Delta^m_\sharp)$ is marked if and only if its image in $\sd^+(\Delta^n_t)$ is marked.
\end{cor}

\begin{proof}
  As above, the \emph{only if} direction follows from the functoriality of $\sd^+$. It remains to show that if the image in $\sd^+(\Delta^n_t)$ is marked, then so is the original morphism. By \zcref{lem:antidiagonal_on_flat+sharp}, the projection
  \[
    \begin{tikzcd}
      (\sd(\Delta^n)\times \sd(\Delta^m))^{ (t,\sharp)} \arrow[r] & \sd(\Delta^n)^{(t,\sharp)}
    \end{tikzcd}
  \]
  is an equivalence of marked simplicial sets, where the superscripts indicate the markings given by marking the images of marked morphisms from $\sd^+(\Delta^n_t\times \Delta^m_\sharp)$. By functoriality, we see that
  \[
    \begin{tikzcd}
      \sd(\Delta^n)^{(t,\sharp)}\arrow[r] & \sd^+(\Delta^n_t)
    \end{tikzcd}
  \]
  preserves markings. Since there is an inclusion of decorated simplicial sets $\Delta^n_t\times\{v\}\to \Delta^n_t\times \Delta^m_\sharp$, it follows that
  \[
    \sd^+(\Delta^n)^{(t,\sharp)}=\sd^+(\Delta^n_t).
  \]
  Thus, every Cartesian functor $f:\sd(\Delta^n_t\times \Delta^m_\sharp)\to \scr{C}$ descends to a marking preserving functor $\sd^+(\Delta^n_t)\to \scr{C}$. But this means that if the image of a morphism in $\sd^+(\Delta^n_t)$ is marked, then it is sent to an equivalence. Thus the original morphism is antidiagonal, as desired.
\end{proof}

\subsection{Mapping spaces}
We will now use the above results to prove \zcref{prop:mappingUP_prop}, starting with some elementary model category theoretic results.

\begin{defn}
  For a decorated simplex $\Delta^n_t$, we define a functor
  \[
    \begin{tikzcd}
      \sd^K(\Delta^n_t\times(-)_\sharp): &[-3em] \Set_\Delta \arrow[r] & \Set_\Delta^+
    \end{tikzcd}
  \]
  by left Kan extending $\sd^+(\Delta^n_t\times \Delta^m_\sharp)$. Its right adjoint,
  \[
    M_{\Delta^n_t}:\Set_\Delta^+ \to \Set_\Delta
  \]
  is given by $M_{\Delta^n_t}(\mathcal{C}) = \Hom_{\Set_\Delta^+}(\sd^+(\Delta^n_t\times (-)_\sharp), \mathcal{C})$.
\end{defn}

\begin{lem}
  For each decorated simplex $\Delta^n_t$, the functor $\sd^K(\Delta^n_t\times (-)_\sharp)$ is left Quillen from the Kan--Quillen model structure to the marked model structure. Consequently, $M_{\Delta^n_t}$ is right Quillen.
\end{lem}

\begin{proof}
  It is immediate that the functor preserves cofibrations. A similar argument to that of \cite[Proposition 10.1.3]{dyckerhoff_HigherSegalSpaces_2019} applied to the natural maps $\sd^+(\Delta^n_t\times X_\sharp)\to \sd^+(\Delta^n_t)\times X_\sharp$ shows that these are weak equivalences, so that this functor preserves weak equivalences.
\end{proof}

\begin{cor} \label{cor:antidiagonal_in_t_sharp_Kan}
  A $1$-simplex in $\sd^K(\Delta^n_t\times X_\sharp)$ is marked if and only if its image in $\sd^+(\Delta^n_t)$ is marked.
\end{cor}

\begin{proof}
  This follows from explicitly computing the marking in the left Kan extension using \zcref{cor:antidiagonal_in_t_sharp}.
\end{proof}

\begin{cor}
  The functors $\sd^+(\Delta^n_t\times (-)_\sharp)$ and $\sd^K(\Delta^n_t\times (-)_\sharp)$ coincide.
\end{cor}

\begin{proof}
  This is immediate from the characterization of the marking on $\sd^K(\Delta^n_t\times (-)_\sharp)$ stated in \zcref{cor:antidiagonal_in_t_sharp_Kan} and the fact that we have morphisms of marked simplicial sets
  \[
    \begin{tikzcd}
      \sd^K(\Delta^n_t\times(-)_\sharp)\arrow[r] & \sd^+(\Delta^n_t\times (-)_\sharp) \arrow[r] & \sd^+(\Delta^n_t).
    \end{tikzcd}
  \]
\end{proof}

\begin{cor}\label{cor:mapping_equiv_without_limits}
  For any decorated simplex $\Delta^n_t$, the natural transformations
  \[
    \begin{tikzcd}
      \sd^+(\Delta^n_t\times X_\sharp) \arrow[r] & \sd^+(\Delta^n_t)\times \sd(X)_\sharp \arrow{r}{(\id, \fv)} & \sd^+(\Delta^n_t)\times X_\sharp
    \end{tikzcd}
  \]
  are natural weak equivalences of left Quillen functors. In particular, for any $(\infty,1)$-category $\scr{C}$ presented as a fibrant marked simplicial set, we have an induced weak equivalence of right adjoint functors induced by composition with these maps.
\end{cor}

\begin{lem} \label{lem:mapping_equiv_with_limits}
  Let $\mathcal{C}$ be an $(\infty,1)$-category with finite limits. Then, $\Map(\Delta^n_t,\Span_\infty(\scr{C}))$ is the full sub-groupoid of $M_{\Delta^n_t}(\scr{C})$ spanned by the Cartesian functors $\sd^+(\Delta^n_t) \to \scr{C}$.
\end{lem}

\begin{proof}
  This is equivalently the statement that a functor $f:\sd^+(\Delta^n_t\times \Delta^m_\sharp)\to \scr{C}$ is Cartesian if and only if its restriction to each $\sd^+(\Delta^n_t\times\{v\})$ is Cartesian. The \emph{only if} implication is immediate.

  To see the other implication, note that by localization, $f$ descends to a functor $\tilde{f}:\sd^+(\Delta^n_t)\to \scr{C}$, and that the image under the canonical map
  \[
    \begin{tikzcd}
      \sd^+(\Delta^n_t\times \Delta^m_\sharp) \arrow[r] & \sd^+(\Delta^n_t)
    \end{tikzcd}
  \]
  of a Segal cube in $\sd^+(\Delta^n_t\times \Delta^m_\sharp)$ is either a degenerate cube or a Segal cube in $\sd(\Delta^n_t)$ (this is because a simplex of $\Delta^n_t\times \Delta^m_\sharp$ is thin precisely if its image in $\Delta^n_t$ is). As such $f$ is Cartesian if and only if $\widetilde{f}$ is Cartesian. But if the composites
  \[
    \begin{tikzcd}
      \sd^+(\Delta^n_t\times \{v\})\arrow[r] & \sd^+(\Delta^n_t\times \Delta^m_\sharp) \arrow[r] & \sd^+(\Delta^n_t)\arrow[r,"\tilde{f}"] & \scr{C}
    \end{tikzcd}
  \]
  are Cartesian, then $\widetilde{f}$ is Cartesian completing the proof.
\end{proof}

We now have all the pieces in place for the proof of \zcref{prop:mappingUP_prop}.

\begin{proof}[Proof (of \zcref{prop:mappingUP_prop})] \label{proof:prop:mappingUP_prop}
  By \zcref{cor:mapping_equiv_without_limits}, we have an equivalence
  \[
    M_{\Delta^n_t}(\scr{C})\simeq \Map(\sd^+(\Delta^n_t),\scr{C})
  \]
  natural in $\Delta^n_t$. By \zcref{lem:mapping_equiv_with_limits}, restricting to the full subgroupoids on Cartesian vertices yields a natural equivalence
  \[
    \Map(\Delta^n_t,\Span_\infty(\scr{C}))\simeq \Map^{\cart}(\sd^+(\Delta^n_t),\scr{C}).
  \]
  Note that by the definition of the marking, the latter space is simply $\Map^{\cart}(\sd(\Delta^n_t),\scr{C})$.
\end{proof}

\section{The \texorpdfstring{$(\infty,\infty)}{(∞,∞)}$-category of functors into \texorpdfstring{$\Span_\infty(\mathcal{C})$}{Span_∞(𝒞)}}
\label{sec:mapping_categories_into_Span}

The main result of this paper is, in effect, an upgrade of the mapping space property of $\Span_\infty(\scr{C})$ proven in the previous section to accommodate transfors of all dimensions.

\subsection{Main results}
We briefly recall some notation and definitions on functor categories before stating the main results.

\begin{defn}
  Let $\Fun(\scr{C},\scr{D})$ denote the exponential objects in both $\Cat_\infty$ and $\Set_\Delta^{\dec}$ associated to the Cartesian product. This notational collision is justified by the fact that the two coincide when mapping from a cofibrant object to a fibrant object.
\end{defn}

\begin{defn}
  For $K_t\in \Set_\Delta^{\dec}$ and $\scr{C}\in \Cat_\infty$, we denote by
  \[
    \begin{tikzcd}
      \Fun^{\oplax}(\scr{C},-):&[-3em] \Cat_\infty \arrow[r] & \Cat_\infty
    \end{tikzcd}
  \]
  and
  \[
    \begin{tikzcd}
      \Fun^{\oplax}(K_t,-):&[-3em] \Set_\Delta^{\dec} \arrow[r] & \Set_\Delta^{\dec}
    \end{tikzcd}
  \]
  the right adjoints to $\scr{C}\otimes-$ and $K_t\otimes-$, respectively. Again, the notational collision is justified by the fact that the two agree on fibrant-cofibrant objects.
\end{defn}

With these definitions fixed, we can now state our main theorem in both of its forms.

\begin{thm}\label{thm:functor_UP_inftycat}
  Let $\scr{C}\in \Cat_1^{\lex}$ and $\scr{D}\in \Cat_\infty$. There is an equivalence
  \[
    \Fun^{\oplax}(\scr{D},\Span_\infty(\scr{C}))\simeq \Span_\infty(\Fun^{\cart}(\SD(\scr{D}),\scr{C}))
  \]
  natural in $\scr{D}$ and $\scr{C}$.
\end{thm}

Model-categorically, this becomes

\begin{thm}\label{thm:functor_UP_modelcat}
  Let $K_t\in \Set_\Delta^{\dec}$ and let $\scr{D}$ be a $1$-category with finite limits. There is an equivalence
  \[
    \Fun^{\oplax}(K_t,\Span_\infty(\scr{C}))\simeq \Span_\infty(\Fun^{\cart}(\sd(K_t),\scr{C}))
  \]
  natural in $K_t$.
\end{thm}

As in the previous section, both of these theorems follow from the same proposition.

\begin{prop}\label{prop:functor_UP_prop}
  There is an equivalence
  \[
    \Map^{\cart}(\sd(\Delta^n_t\otimes \Delta^m_s),\scr{C})\simeq \Map^{\cart}(\sd(\Delta^n_t), \Fun^{\cart}(\sd(\Delta^m_s),\scr{C})),
  \]
  where $t, s \in \{+,\flat\}$, natural in $\tDelta\times \tDelta$.
\end{prop}

As before, we first briefly derive the theorems from the proposition, and then devote the remainder of the section to the (more technical) proof of the proposition. As the proof of \zcref{thm:functor_UP_modelcat} is conceptually identical, we only present the proof of \zcref{thm:functor_UP_inftycat}.

\begin{proof}[Proof (of \zcref{thm:functor_UP_inftycat})]
  We write a string of natural equivalences
  \begin{align*}
    \Map\left(\Delta^n_t,\Fun^{\oplax}(\scr{D},\Span_\infty(\scr{C}))\right)
    & \simeq \Map\left(\Delta^n_t,\Fun^{\oplax}\left(\colim_{\tDelta_{/\scr{D}}}\Oriental^m_s,\Span_\infty(\scr{C})\right)\right) \\
    & \simeq \Map\left(\Delta^n_t,\lim_{\tDelta_{/\scr{D}}^\op}\Fun^{\oplax}\left(\Oriental^m_s,\Span_\infty(\scr{C})\right)\right) \\
    & \simeq \lim_{\tDelta_{/\scr{D}}^\op}\Map\left(\Delta^n_t,\Fun^{\oplax}\left(\Oriental^m_s,\Span_\infty(\scr{C})\right)\right) \\
    & \simeq \lim_{\tDelta_{/\scr{D}}^\op}\Map\left(\Delta^n_t,\Fun^{\oplax}\left(\Delta^m_s,\Span_\infty(\scr{C})\right)\right) \\
    & \simeq \lim_{\tDelta_{/\scr{D}}^\op} \Map(\Delta^n_t \otimes \Delta^m_s, \Span_\infty(\scr{C})) \\
    & \simeq \lim_{\tDelta_{/\scr{D}}^\op}\Map^{\cart}\left(\sd(\Delta^n_t\otimes \Delta^m_s),\scr{C}\right) \tag{\zcref{prop:mappingUP_prop}} \\
    & \simeq \lim_{\tDelta_{/\scr{D}}^\op}\Map^{\cart}\left(\sd(\Delta^n_t), \Fun^{\cart}\left(\sd(\Delta^m_s),\scr{C}\right)\right) \tag{\zcref{prop:functor_UP_prop}}\\
    & \simeq \Map^{\cart}\left(\sd(\Delta^n_t),\lim_{\tDelta_{/\scr{D}}^\op} \Fun^{\cart}\left(\sd(\Delta^m_s),\scr{C}\right)\right)\\
    & \simeq \Map^{\cart}\left(\sd(\Delta^n_t), \Fun^{\cart}\left(\colim_{\tDelta_{/\scr{D}}}\sd(\Delta^m_s),\scr{C}\right)\right)\\
    & \simeq \Map^{\cart}\left(\sd(\Delta^n_t), \Fun^{\cart}\left(\SD(\scr{D}),\scr{C}\right)\right)
  \end{align*}
  The theorem then follows by the density of $\tDelta$ in $\Cat_\infty$, \zcref{prop:tDelta_dense}.
\end{proof}

\subsection{The strategy}

The strategy for proving \zcref{prop:functor_UP_prop} can be summarized in the following commutative diagram, where our aim is to construct the vertical arrow on the left and show that it is an equivalence. The idea is to embed both the source and target to a larger functor category, and then identify the essential images. Below, we explain the maps in the diagram.

\noindent\adjustbox{scale=0.965,center}{%
  \begin{tikzcd}[column sep=small]
    \Map^\cart(\sd(\Delta^n_t), \Fun^\cart(\sd(\Delta^m_s), \scr{C})) \ar[dashrightarrow]{d} & \ar{l}[swap]{\simeq} \Map^{\cart,\cart}(\sd(\Delta^n_t) \times \sd(\Delta^m_s), \scr{C}) \ar[hookrightarrow]{r} \ar{d}{\text{\zcref{prop:auxialiary_functor_UP_prop}}}[swap]{\simeq} & \Map(\sd(\Delta^n_t) \times \sd(\Delta^m_s), \scr{C}) \ar[hookrightarrow]{d}{(\ell_{n,m})_*} \\
    \Map^\cart(\sd(\Delta^n_t \otimes \Delta^m_s), \scr{C}) & \ar{l}[swap]{L_{n,m}^*}{\simeq}  \Map^\lcart(R^{n,m}_{t,s}, \scr{C}) \ar[hookrightarrow]{r} & \Map(R^{n,m}, \scr{C})
  \end{tikzcd}
}

\subsubsection{Removing the Cartesian conditions} We start with the fully faithful functor $(\ell_{n,m})_*$ on the right.

Recall the proof of \zcref{prop:product_localization}, which displays $\sd(\Delta^n)\times \sd(\Delta^m)$ as a coreflective colocalization of $R^{n,m}$, where, in particular, we have a fully faithful left adjoint $\ell: \sd(\Delta^n) \times \sd(\Delta^m) \to R^{n,m}$. We wish to consider the \emph{right} adjoint $(\ell_{n,m})_*$ of $\ell_{n,m}^*$ given by right Kan extension
\[
  \begin{tikzcd}
    \ell_{n,m}^\ast:&[-3em] \Fun(R^{n,m},\scr{C}) \arrow[r,shift left] & \Fun(\sd(\Delta^n)\times \sd(\Delta^m),\scr{C})\arrow[l,shift left, hookrightarrow]  &[-3em]: (\ell_{n,m})_\ast.
  \end{tikzcd}
\]
Note that $\ell_{n,m}^*$ is fully faithful since $\ell_{n,m}$ is. Thus, $(\ell_{n,m})_*$ naturally identifies $\Fun(\sd(\Delta^n)\times \sd(\Delta^m),\scr{C})$ with the full subcategory of $\Fun(R^{n,m},\scr{C})$ spanned by the functors given by right Kan extension.

\begin{rmk}
  The adjunction of posets induces an adjunction $\gamma_{n,m}^* \dashv \ell_{n,m}^*$ on functor categories
  \[
    \begin{tikzcd}
      \gamma_{n,m}^\ast:&[-3em] \Fun(\sd(\Delta^n)\times \sd(\Delta^m),\scr{C})\arrow[shift left, hookrightarrow]{r} & \Fun(R^{n,m},\scr{C})\arrow[l,shift left] &[-3em]: \ell_{n,m}^\ast,
    \end{tikzcd}
  \]
  where $\gamma_{n,m}^*$ can also be defined as the \emph{left} Kan extension $(\ell_{n,m})_!$, which is also fully faithful. However, this is \emph{not} the one we consider.
\end{rmk}

\subsubsection{The localized Cartesian condition} We move on to the middle column.

\begin{defn}
  Let $\vec{\bf{u}}\in \sd(\Delta^n_t \otimes \Delta^m_s)$ be a marked simplex, and $C\to \sd(\Delta^n\otimes \Delta^m)$ a non-corner Segal cube diagram associated to this simplex. We denote the composite of this cube with $L_{n,m}$ by $f_{\vec{\bf{u}}}$ and call it a \emph{localized Segal cube}. We will denote by $\Map^{\lcart}(R^{n,m}_{t,s},\scr{C})\subset \Map(R^{n,m},\scr{C})$ the full subcategory spanned by the functors which send the localized Segal cubes to limit cubes and call such a functor \emph{localized Cartesian}.
\end{defn}

By definition, we have the two horizontal embeddings on the right and the top left horizontal equivalence, where the superscript in the top middle term indicates the Cartesian condition in each variable separately. The bottom left horizontal equivalence follows from \zcref{lem:corner_Segal_vs_corner_collapse}.\zcref{item:lem_corner_Segal_vs_corner_collapse:equivalence} and the observation that the corner collapses are marked in $\sd^+(\Delta^n_t\otimes \Delta^m_s)$ regardless of the decorations $t$ and $s$.

Altogether, it remains to establish the equivalence represented by the middle vertical arrow, which is the content of the following proposition.

\begin{prop}\label{prop:auxialiary_functor_UP_prop}
  A functor $f:R^{n,m}_{t,s}\to \scr{C}$ sends localized Segal cubes to limit cubes if and only if $f$ is right Kan extended from a functor $f: \sd(\Delta^n_t)\times \sd(\Delta^m_s) \to \scr{C}$ along $\ell_{n,m}$ which is Cartesian in each variable separately.
\end{prop}

We will first prove \zcref{prop:functor_UP_prop} (or equivalently, \zcref{prop:auxialiary_functor_UP_prop}) in the case where the decorations on the two simplices are minimal in \zcref{prop:unmarkedcase}, and then establish the remaining cases in \zcref{prop:one_sided_decorated_case,prop:fully_decorated_case} by tracing the Segal conditions through the equivalence established in the minimally decorated case. For the reader's convenience, we have visualized in \zcref{fig:zigzag_11flat} a low-dimensional example of the procedure.

\begin{figure}[htb!]
  \begin{subfigure}{1\textwidth}
    \centering
    \begin{tikzpicture}[scale=1.7]
      \path (0,1.5) node {$\sd(\Delta^1_\flat\otimes\Delta^1_\flat)$};
      \coordinate (v1) at  (-1,1);
      \coordinate (v2) at  (1,1);
      \coordinate (v3) at (-1,-1);
      \coordinate (v4) at (1,-1);

      \coordinate (m12) at ($0.5*(v1) + 0.5*(v2)$);
      \coordinate (m13) at ($0.5*(v1) + 0.5*(v3)$);
      \coordinate (m14) at ($0.5*(v1) + 0.5*(v4)$);
      \coordinate (m24) at ($0.5*(v2) + 0.5*(v4)$);
      \coordinate (m34) at ($0.5*(v3) + 0.5*(v4)$);

      \coordinate (m124) at ($0.333*(v1) + 0.333*(v2) + 0.333*(v4)$);
      \coordinate (m134) at ($0.333*(v1) + 0.333*(v3) + 0.333*(v4)$);

      \foreach \x in {1,2,3,4}{
        \path (v\x) node (v\x) {$\bullet$};
      };

      \foreach \x/\y in {1/2,1/3,1/4,2/4,3/4}{
        \path (m\x\y) node (m\x\y) {$\bullet$};
      };

      \foreach \x/\y/\z in {1/2/4,1/3/4}{
        \path (m\x\y\z) node (m\x\y\z) {$\bullet$};
      };

      \foreach \x/\y in {1/3,1/4,3/4}{
        \draw[->] (m\x\y) to (v\x);
        \draw[->] (m\x\y) to (v\y);
      };

      \draw[->] (m12) to (v1);
      \draw[blue,->] (m12) to (v2);

      \draw[->] (m24) to (v4);
      \draw[blue,->] (m24) to (v2);

      \draw[blue,->] (m124) to (m12);
      \draw[red,->] (m124) to (m14);
      \draw[blue,->] (m124) to (m24);

      \foreach \x/\y/\z in {1/3/4}{
        \draw[->] (m\x\y\z) to (m\x\y);
        \draw[->] (m\x\y\z) to (m\x\z);
        \draw[->] (m\x\y\z) to (m\y\z);
      }
      \draw[->] (1.3,0) to (1.7,0);

      \begin{scope}[xshift=3cm]
        \path (0,1.5) node {$R^{1,1}$};
        \coordinate (v1) at  (-1,1);
        \coordinate (v2) at  (1,1);
        \coordinate (v3) at (-1,-1);
        \coordinate (v4) at (1,-1);

        \coordinate (m12) at ($0.5*(v1) + 0.5*(v2)$);
        \coordinate (m13) at ($0.5*(v1) + 0.5*(v3)$);
        \coordinate (m14) at ($0.5*(v1) + 0.5*(v4)$);
        \coordinate (m24) at ($0.5*(v2) + 0.5*(v4)$);
        \coordinate (m34) at ($0.5*(v3) + 0.5*(v4)$);

        \coordinate (m124) at ($0.333*(v1) + 0.333*(v2) + 0.333*(v4)$);
        \coordinate (m134) at ($0.333*(v1) + 0.333*(v3) + 0.333*(v4)$);

        \foreach \x in {1,2,3,4}{
          \path (v\x) node (v\x) {$\bullet$};
        };

        \foreach \x/\y in {1/2,1/3,1/4,2/4,3/4}{
          \path (m\x\y) node (m\x\y) {$\bullet$};
        };

        \foreach \x/\y/\z in {1/3/4}{
          \path (m\x\y\z) node (m\x\y\z) {$\bullet$};
        };

        \foreach \x/\y in {1/3,1/4,3/4}{
          \draw[->] (m\x\y) to (v\x);
          \draw[->] (m\x\y) to (v\y);
        };

        \draw[->] (m12) to (v1);
        \draw[blue,->] (m12) to (v2);

        \draw[->] (m24) to (v4);
        \draw[blue,->] (m24) to (v2);

        \foreach \x/\y/\z in {1/3/4}{
          \draw[->] (m\x\y\z) to (m\x\y);
          \draw[->] (m\x\y\z) to (m\x\z);
          \draw[->] (m\x\y\z) to (m\y\z);
        }
        \draw[blue,->] (m14) to (m12);
        \draw[blue,->] (m14) to (m24);
      \end{scope}

    \end{tikzpicture}
    \caption{The localization map $L_{1,1}$. The edge marked in red is the antidiagonal morphism sent to an equivalence by $L_{1,1}$ and the blue square to the left is the remaining Segal cube associated to the lone decorated $2$-simplex. The blue square on the right is the localized Segal cube.} \label{subfig:zigzag_11flat_L}
  \end{subfigure}

  \begin{subfigure}{1\textwidth}
    \centering
    \begin{tikzpicture}[scale=1.7]
      \begin{scope}[xshift=3cm]
        \path (0,1.5) node {$R^{1,1}$};
        \coordinate (v1) at  (-1,1);
        \coordinate (v2) at  (1,1);
        \coordinate (v3) at (-1,-1);
        \coordinate (v4) at (1,-1);

        \coordinate (m12) at ($0.5*(v1) + 0.5*(v2)$);
        \coordinate (m13) at ($0.5*(v1) + 0.5*(v3)$);
        \coordinate (m14) at ($0.5*(v1) + 0.5*(v4)$);
        \coordinate (m24) at ($0.5*(v2) + 0.5*(v4)$);
        \coordinate (m34) at ($0.5*(v3) + 0.5*(v4)$);

        \coordinate (m124) at ($0.333*(v1) + 0.333*(v2) + 0.333*(v4)$);
        \coordinate (m134) at ($0.333*(v1) + 0.333*(v3) + 0.333*(v4)$);

        \foreach \x in {1,2,3,4}{
          \path (v\x) node (v\x) {$\bullet$};
        };

        \foreach \x/\y in {1/2,1/3,1/4,2/4,3/4}{
          \path (m\x\y) node (m\x\y) {$\bullet$};
        };

        \foreach \x/\y/\z in {1/3/4}{
          \path (m\x\y\z) node (m\x\y\z) {$\bullet$};
        };

        \foreach \x/\y in {1/2,1/3,2/4,3/4}{
          \draw[red,->] (m\x\y) to (v\x);
          \draw[red,->] (m\x\y) to (v\y);
        };
        \draw[->] (m14) to (v1);
        \draw[->] (m14) to (v4);

        \foreach \x/\y/\z in {1/3/4}{
          \draw[red,->] (m\x\y\z) to (m\x\y);
          \draw[->] (m\x\y\z) to (m\x\z);
          \draw[red,->] (m\x\y\z) to (m\y\z);
        }
        \draw[->] (m14) to (m12);
        \draw[->] (m14) to (m24);
        \draw[dashed,red,->] (m134) to (m12);
        \draw[dashed,red,->] (m134) to (m24); \draw[<-] (1.3,0) to (1.7,0);
      \end{scope}

      \begin{scope}[xshift=6cm]
        \path (0,1.5) node {$\sd(\Delta^1_\flat)\times\sd(\Delta^1_\flat)$};
        \coordinate (v1) at  (-1,1);
        \coordinate (v2) at  (1,1);
        \coordinate (v3) at (-1,-1);
        \coordinate (v4) at (1,-1);

        \coordinate (m12) at ($0.5*(v1) + 0.5*(v2)$);
        \coordinate (m13) at ($0.5*(v1) + 0.5*(v3)$);
        \coordinate (m14) at ($0.5*(v1) + 0.5*(v4)$);
        \coordinate (m24) at ($0.5*(v2) + 0.5*(v4)$);
        \coordinate (m34) at ($0.5*(v3) + 0.5*(v4)$);

        \coordinate (m124) at ($0.333*(v1) + 0.333*(v2) + 0.333*(v4)$);
        \coordinate (m134) at ($0.333*(v1) + 0.333*(v3) + 0.333*(v4)$);

        \foreach \x in {1,2,3,4}{
          \path (v\x) node (v\x) {$\bullet$};
        };

        \foreach \x/\y in {1/2,1/3,1/4,2/4,3/4}{
          \path (m\x\y) node (m\x\y) {$\bullet$};
        };

        \foreach \x/\y in {1/2,1/3,2/4,3/4}{
          \draw[->] (m\x\y) to (v\x);
          \draw[->] (m\x\y) to (v\y);
        };

        \draw[->] (m14) to (m12);
        \draw[->] (m14) to (m24);

        \draw[->] (m14) to (m34);
        \draw[->] (m14) to (m13);
      \end{scope}
    \end{tikzpicture}
    \caption{The inclusion $\ell_{1,1}$. The image of $\sd(\Delta^1_\flat)\times \sd(\Delta^1_\flat)$ in $R^{1,1}$ is marked in red. The remaining vertex is added by right Kan extension, which in this case corresponds to the localized Segal cube above being a limit cube.} \label{subfig:zigzag_11flat_ell}
  \end{subfigure}
  \caption{The zigzag defining the equivalence of \zcref{prop:functor_UP_prop} in the case of $\Delta^1_\flat\otimes \Delta^1_\flat$.} \label{fig:zigzag_11flat}
\end{figure}

\subsection{The minimally decorated case}
We will now prove the minimally decorated case, encapsulated in \zcref{prop:unmarkedcase} below. Note that in this case, there are no Segal cubes to consider in $\sd(\Delta^n_\flat)\times \sd(\Delta^m_\flat)$.

\begin{prop}\label{prop:unmarkedcase}
  There is an equivalence of spaces
  \[
    \Map(\sd(\Delta^{n})\times \sd(\Delta^{m}),\cC) \simeq \Map^{\cart}(\sd(\Delta^{n}_{\flat}\otimes\Delta^{m}_{\flat}),\cC)
  \]
  given by right Kan extension along $\ell_{n,m}$ followed by precomposition with $L_{n,m}$.
\end{prop}

To avoid clutter, we remove subscripts $\flat$ from our notation in what follows. By the discussion above, \zcref{prop:unmarkedcase} follows from the following special case of \zcref{prop:auxialiary_functor_UP_prop}.

\begin{prop}\label{prop:SegaliffKE}
  A functor $g:R^{n,m}\to \scr{C}$ sends localized Segal cubes to limit diagrams if and only if $g$ is right Kan extended along $\ell_{n,m}$.
\end{prop}

\begin{proof}
  The basic strategy of proof is induction, together with careful manipulation of diagram categories. First, as base cases, if $n$ or $m$ is $0$, then $\ell_{n,m}$ is an isomorphism, and there are no localized Segal conditions.

  Inductively, we assume that the statement is true for any $(n',m')$ such that $n'\leq n$ and $m'\leq m$, with one of these inequalities being strict. Consider a rectilinear path $M\in R^{n,m}$ of maximal length and a localized Segal cube defined by arms $\{M\to M_i\}_{i=1}^\ell$. Fix a functor $g: R^{n,m}\to \scr{C}$ and assume, by induction, that on every subgrid, it is both Kan extended and sends localized Segal cubes to limit cubes. For each $I\in \Power_*(\ell)$, we write $M_I$ for the join (see \zcref{defn:join_and_meet}).

  Since we are dealing with posets, the slice $\sd(\Delta^n)\times \sd(\Delta^m)_{M_I/}$ is simply a subposet of $\sd(\Delta^n)\times \sd(\Delta^m)$. Let $N_1,\ldots, N_k$ be the minimal elements in the poset $\sd(\Delta^n)\times \sd(\Delta^m)_{M/}$. As with the $M_I$, we write $N_S$ for the join for any $S \in \Power_*(k)$. For conciseness, in what follows, we will write $\SD$ for  $\sd(\Delta^n)\times \sd(\Delta^m)$ and denote
  \[
    U_S\coloneqq \SD_{N_S/} = \bigcap_{s\in S} (\SD)_{N_s/}.
  \]

  We will then consider the zigzag of morphisms
  \[
    \begin{tikzcd}[row sep=tiny]
      & \lim_{I\in \Power_*(\ell)} g(M_I) \arrow[r,"\text{(1)}"]
      & \lim_{I\in \Power_*(\ell)}\lim_{V\in\SD_{M_I/}}g(V) \\
      \phantom{a} \arrow[r,"\text{(2)}"]
      & \lim_{I\in \Power_*(\ell)}\lim_{S\in \Power_*(k)}\lim_{V\in (U_S)_{M_I/}}g(V) \arrow[r,"\cong"]
      & \lim_{S\in \Power_*(k)}\lim_{I\in \Power_*(\ell)}\lim_{V\in (U_S)_{M_I/}}g(V) \\
      \phantom{a}
      & \lim_{S\in \Power_*(k)}\lim_{V\in U_S}g(V) \arrow[l,"\text{(3)}"']
      & \lim_{V \in \SD_{M/}} g(V) \arrow[l,"\text{(4)}"']
    \end{tikzcd}
  \]
  and will show that, under our hypotheses, all four are equivalences. Briefly,
  \begin{itemize}
    \item The morphism (1) is an equivalence because $g$ is Kan extended from $\SD$ by induction, using the fact that $M_I$ lives in a proper subgrid, guaranteed by the fact that the Segal condition under consideration is non-corner.
    \item The morphism (2) is an equivalence because
      \[
        \colim_{S\in \Power_*(k)^\op} (U_S)_{M_I/}\simeq \SD_{M_I/}
      \]
      which we will prove as \zcref{lem:colim_of_USs} below.
    \item The unlabeled equivalence is simply the fact that limits commute with limits.
    \item The morphism (3) is an equivalence because, for any $S$, there is a choice of $i_0$ such that, for all $I\in \Power(\ell)$ with $i_0\notin I$, the induced morphism
      \[
        \begin{tikzcd}
          (U_{S})_{M_{I\cup\{i_0\}}/} \arrow[r] & (U_S)_{M_I/}
        \end{tikzcd}
      \]
      is an isomorphism of posets (\zcref{lem:constantdirectionSegalcube}). Thus, the cube consisting of the $(U_S)_{M_I/}$ indexed by the poset $\Power(\ell)^\op$ is a colimit cube, with colimit $(U_S)_{M/}=U_S$.
    \item The morphism (4) is an equivalence due to the same reason as (2).
  \end{itemize}
  Since the morphisms $g(M)\to \lim_{I\in \Power_*(\ell)} g(M_I)$ and $g(M)\to \lim_{V\in \SD_{M/}} g(V)$ commute with the zigzag, it follows that $g$ is satisfied the chosen Segal condition with vertex $M$ if and only if the value of $g$ at $M$ is Kan extended from $\SD$. We conclude the proof by observing that $M \in R^{n,m}$ is the cone of a localized Segal cube if and only if $M$ does not live in the image of $\ell_{n,m}$.
\end{proof}

\begin{lem}\label{lem:colim_of_USs}
  Using the notation of \zcref{prop:SegaliffKE}, for any $I\in \Power(\ell)$
  \[
    \colim_{S\in \Power_*(k)^\op} (U_S)_{M_I/}\simeq \SD_{M_I/}.
  \]
\end{lem}

\begin{proof}
  This is a colimit diagram over the poset of intersections of subposets. As such, it is a Reedy cofibrant diagram of simplicial sets, and so is an $(\infty,1)$-categorical colimit since it is a $1$-categorical colimit.
\end{proof}

\begin{lem}\label{lem:constantdirectionSegalcube}
  Using the notation of \zcref{prop:SegaliffKE}, given an element $S$ in the poset of intersections of the $N_s$, there is an element $i_0$ such that for all $I\in\Power(\ell)$, the induced morphism
  \[
    \begin{tikzcd}
      (U_{S})_{M_{I\cup\{i_0\}}/} \arrow[r] & (U_S)_{M_I/}
    \end{tikzcd}
  \]
  is an isomorphism of posets.
\end{lem}

\begin{proof}
  When $i_0 \in I$, there is nothing to show. Thus, we will assume that $i_0 \notin I$.

  Observe that $U_S = \SD_{N_S/}$ is a subposet of $\SD_{N_{s}/}$ for any $s\in S$, where, by definition, $N_s$ is an L-shaped path (see \zcref{rmk:L-shaped}) under $M$. We will show the following stronger statement:
  \begin{itemize}
    \item[(*)] For any L-shaped path $L$ under $M$, there exists $i_0$ such that for all subposets $V$ of $\SD_{L/}$, the natural morphism
      \[
        V_{M_{I\cup \{i_0\}}/} \to V_{M_I/}
      \]
      is an isomorphism of posets.
  \end{itemize}

  First, by assumption, $M \in R^{n,m}$ comes from $\overline{M} \in \sd(\Delta^n\times \Delta^m)$ such that
  \begin{enumerate}[label=(\arabic*)]
    \item $\overline{M}$ is mapped to $M$ under the localization $L_{n,m}$;
    \item $\overline{M}$ contains a horizontal line segment preceding a vertical line segment; in particular, it contains a horizontal segment joined to a vertical segment by a (possibly trivial) diagonal segment as illustrated below.
  \end{enumerate}
  
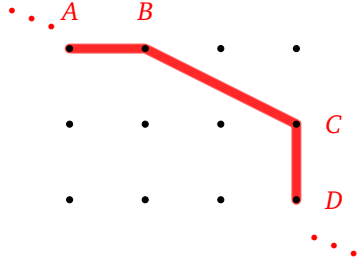
\[
    \begin{tikzpicture}
      \drawGrid{4}{3}
      \begin{scope}[on background layer]
        \begin{scope}[red]
          \drawLine{3pt}{red} (00.center) -- (01.center) -- (13.center) -- (23.center);
          \path (-0.5,0.5) node {\huge$\ddots$};
          \path (3.5,-2.5) node {\huge$\ddots$};
          \path (0,0.5) node {$A$};
          \path (1,0.5) node {$B$};
          \path (3.5,-1) node {$C$};
          \path (3.5,-2) node {$D$};
        \end{scope}

      \end{scope}
    \end{tikzpicture}
  \]

  Now, consider the four vertices on our chosen horizontal and vertical segments in $\overline{M}$, labeled $A$--$D$ in the schematic above. Note that the Segal cube we are considering must involve either removing $A$ or $B$, and must involve either removing $C$ or $D$. We then have the following cases.
  \begin{enumerate}[label=(\arabic*)]
    \item If $L$ lies \emph{above} the schematic, removing one of $C$ or $D$ does not impose any extra condition for things under $L$.
    \item If $L$ lies \emph{to the right} of the schematic, removing one of $A$ or $B$ does not impose any extra condition for things under $L$.
    \item If $L$ lies \emph{in the middle} of the schematic (see \zcref{fig:third-case}), removing $A$, $B$, $C$, or $D$ does not impose any extra condition for things under $L$.
  \end{enumerate}
  \begin{figure}[htb]
    \[
      \begin{tikzpicture}
        \drawGrid{4}{4}

        \begin{scope}[on background layer]
          \begin{scope}[red]
            \drawLine{3pt}{red}  (00.center) -- (01.center) -- (12.center) -- (23.center) -- (33.center);
            \draw ;
            \path (-0.5,0.5) node {\huge$\ddots$};
            \path (3.5,-2.5) node {\huge$\ddots$};
            \path (0,0.5) node {$A$};
            \path (1,0.5) node {$B$};
            \path (3.5,-2) node {$C$};
            \path (3.5,-3) node {$D$};
          \end{scope}

          \begin{scope}[color=blue]
            \drawLine{3pt}{blue} (02.center) -- (12.center) -- (13.center);
            \path (3.8,-1) node {\huge$\dots$};
            \path (2,0.5) node {\huge$\vdots$};
          \end{scope}
        \end{scope}
      \end{tikzpicture}
    \]
    \caption{Schematic description of case $(3)$}\label{fig:third-case}
  \end{figure}
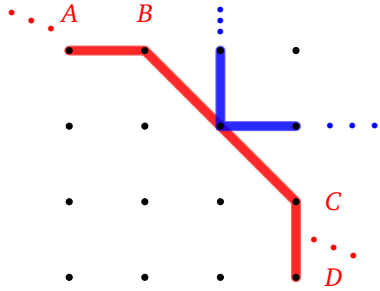

  The claim (*) follows from the observation that for any $I$, $V_{M_I/}$ is a subposet of $\SD_{L/}$, concluding the proof.
\end{proof}

\subsection{Transferring Segal conditions}

To deduce the general case from \zcref{prop:unmarkedcase}, we will repeatedly need to transfer limit conditions through the zigzag of equivalences. In this subsection, we introduce a special class of decorations called L-saturated (see \zcref{defn:decoration_L-saturated}) and show in \zcref{prop:transfer_L-saturated_decorations} that this class of decorations is particularly well-adapted for this purpose. We will adopt the notation of the preceding subsection.

We start with \zcref{prop:mixed_segal_suffices_for_L_shaped}, which shows that Segal conditions on a single L-shaped path correspond to the mixed Segal conditions we will now define.

\begin{defn}
  Let $\sigma:\Delta^k\to \Delta^n$ and $\tau:\Delta^\ell\to \Delta^m$ be non-degenerate simplices. When $k,l \geq 1$, we define the \emph{mixed Segal cubes} in $\sd(\Delta^n)\times \sd(\Delta^m)$ associated to $(\sigma,\tau)$ as follows:
  \begin{itemize}
    \item The \emph{inner mixed Segal cube} is the product of the Segal cube of $\sigma$ containing the $k$-lacuna and the Segal cube of $\tau$ containing the $0$-lacuna.
    \item The \emph{outer mixed Segal cube} is the product of the Segal cube of $\sigma$ not containing the $k$-lacuna and the Segal cube of $\tau$  not containing the $0$-lacuna.
  \end{itemize}
  When $k=0$ (resp. $\ell=0$), the \emph{mixed Segal cubes} associated to $(\sigma,\tau)$ are simply the products of the Segal cubes for $\sigma$ (resp. $\tau$) and $\tau(0)$ (resp. $\sigma(0)$). In this case, the \emph{inner} (resp. \emph{outer}) \emph{mixed Segal cubes} are defined to contain the appropriate lacuna (as defined above) in the non-degenerate direction.

  We say that a functor $f:\sd(\Delta^n)\times \sd(\Delta^m)\to \scr{C}$ satisfies the \emph{mixed Segal conditions} if it sends the mixed Segal cubes to limit cubes. In the inner mixed Segal cube, we call the directions in the cube corresponding to the $k$ and $0$-lacunae the \emph{corner directions}.
\end{defn}

\begin{prop}\label{prop:mixed_segal_suffices_for_L_shaped}
  Let $f:R^{n,m} \to \scr{C}$ be a functor which is right Kan extended from a functor $\overline{f}: \SD \coloneqq \sd(\Delta^n)\times \sd(\Delta^m)\to \scr{C}$. Let $L$ be an L-shaped path in $\sd(\Delta^n\otimes \Delta^m)$, $L_{R^{n,m}}$ (resp. $L_{\SD}$) its image in $R^{n,m}$ (resp. $\SD$) under $L_{n,m}$ (resp. $Z_{n,m}$). Let $(\sigma, \tau)$ denote the pair of simplices in $\Delta^n\times \Delta^m$ corresponding $L_{\SD}$. Then, $f$ sends localized Segal cubes associated to the simplex corresponding to $L$ to limit cubes if and only if $\bar{f}$ satisfies the mixed Segal conditions for $(\sigma, \tau)$.
\end{prop}

\begin{proof}
  The case where the horizontal or vertical component of $L$ consists of a single point is immediate. In what follows, we will assume that both components are non-degenerate.

  There are two localized Segal conditions we need to consider on $L_{R^{n,m}}$. The one which does not contain the $(\sigma(\dim \sigma),\tau(0))$-lacuna is simply the image of the outer mixed Segal condition. It remains to show that the other is equivalent to the inner mixed Segal condition.

  Let $C_{1}: \Power(I)\to R^{n,m}$ be the remaining localized Segal cube. Observe that of the defining prongs of the cube, one of them, denoted by $U$, does not live in the image of $\SD$. The slice under $U$ contains precisely two minimals, denoted by $N$ and $M$ (see \zcref{fig:multicube} for a schematic depiction).
  \begin{figure}[htb]
    \begin{tikzpicture}
      \begin{scope}
        \drawGrid{3}{3}
        \begin{scope}[on background layer]
          \drawLine{6pt}{red}  (00.center) -- (10.center) -- (20.center) -- (21.center) -- (22.center);
        \end{scope}
      \end{scope}

      \begin{scope}[xshift=4cm]
        \drawGrid{3}{3}
        \begin{scope}[on background layer]
          \drawLine{6pt}{red}  (00.center) -- (10.center) -- (11.center) -- (21.center) -- (22.center);
        \end{scope}
      \end{scope}

      \begin{scope}[xshift=8cm,yshift= 2cm]
        \drawGrid{3}{3}
        \begin{scope}[on background layer]
          \drawLine{6pt}{red}  (00.center) -- (10.center) -- (11.center) -- (21.center) -- (22.center);
          \drawLine{3pt}{blue}  (01.center) -- (11.center) -- (21.center) -- (22.center);
        \end{scope}
      \end{scope}

      \begin{scope}[xshift=8cm,yshift= -2cm]
        \drawGrid{3}{3}
        \begin{scope}[on background layer]
          \drawLine{6pt}{red} (00.center) -- (10.center) -- (11.center) -- (21.center) -- (22.center);
          \drawLine{3pt}{blue}  (00.center) -- (10.center) -- (11.center) -- (12.center) ;
        \end{scope}
      \end{scope}
      \draw[->] (2.5,-1) to (3.5,-1);
      \draw[->] (6.5,-0.5) to (7.5,1);
      \draw[->] (6.5,-1.5) to (7.5,-3);
      \path (1,-3) node {$L_{R^{n,m}}$};
      \path (5,-3) node {$U$};
    \end{tikzpicture}
    \caption{Schematic depiction of the minimals in the slice under $U$.}\label{fig:multicube}
  \end{figure}
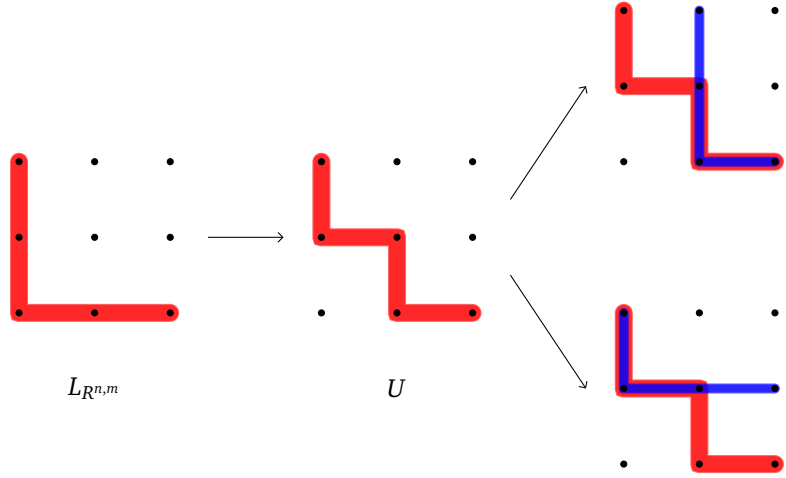
  The (image in $R^{n,m}$ under $\ell_{n,m}$ of the) inner mixed Segal cube is precisely the cube obtained by removing $U$ from the cover, and including $N$ and $M$ instead. Let $C_2:\Power(I\cup\{M,N\})\to R^{n,m}$ be defined by adding $M$ and $N$ to the cover defining $C_1$. Finally, let $C_3:\Power(J)\to R^{n,m}$ be the cube obtained from $C_{2}$ by removing $U$ from the cover. Note that this is precisely the (image of the) inner mixed Segal cube.

  In this notation, we need to show that $f\circ C_1$ is limit if and only if $f\circ C_3$ is, which follows from the following observations:
  \begin{enumerate}
    \item As $M,N<U$ in $R^{n,m}$, for every $K \in \Power_*(\{M,N\})$ the $I$-indexed cube $C_{2}(K\cup-)$ is constant in the direction of $U$. Thus, $f\circ C_{2}$ is limit if and only if $f\circ C_{1}$ is, by \zcref{lem:cubefacelemma}.
    \item For every $K \in \Power(J \setminus \{M,N\})$ the $\{M,N\}$-indexed cube $f\circ C_{2}(\{U\}\cup K\cup -)$ is limit. Indeed, this is just the Kan extension condition applied to the path $U \cap \bigcap_{P \in K} P$ (see \zcref{fig:MN-cube} for a schematic depiction). \zcref{lem:cubefacelemma} then implies that the $J$-indexed cube $f\circ C_2(\{U\} \cup -)$ is limit.
    \item By \zcref{const:cube_moves}, given that the $J$-indexed cube $f\circ C_2(\{U\}\cup -)$ is limit, $f\circ C_{3}$ is limit if and only if $f\circ C_{2}$ is.
      \begin{figure}[htb]
        \begin{tikzpicture}
          \begin{scope}
            \drawGrid{4}{4}
            \begin{scope}[on background layer]
              \drawLine{3pt}{red}  (00.center) --  (20.center) -- (21.center) -- (31.center)--  (33.center);
            \end{scope}
          \end{scope}

          \begin{scope}[xshift=5cm]
            \drawGrid{4}{4}
            \begin{scope}[on background layer]
              \drawLine{3pt}{red}  (00.center) --  (20.center) -- (21.center) --  (23.center);
            \end{scope}
          \end{scope}

          \begin{scope}[yshift=-5cm, xshift=5cm]
            \drawGrid{4}{4}
            \begin{scope}[on background layer]
              \drawLine{3pt}{red}  (01.center) --  (21.center) -- (23.center);
            \end{scope}
          \end{scope}

          \begin{scope}[yshift=-5cm]
            \drawGrid{4}{4}
            \begin{scope}[on background layer]
              \drawLine{3pt}{red}  (01.center) --  (21.center) -- (31.center) -- (31.center)--  (33.center);
            \end{scope}
          \end{scope}
          \draw[->] (3.5,-1.5) to (4.5,-1.5);
          \draw[->] (3.5,-6.5) to (4.5,-6.5);
          \draw[->] (1.5,-3.5) to (1.5, -4.5);
          \draw[->] (6.5,-3.5) to (6.5,-4.5);
        \end{tikzpicture}
        \caption{Schematic depiction of the $\{M,N\}$-indexed cube $C_{2}(U\cup K\cup-)$.}\label{fig:MN-cube}
      \end{figure}
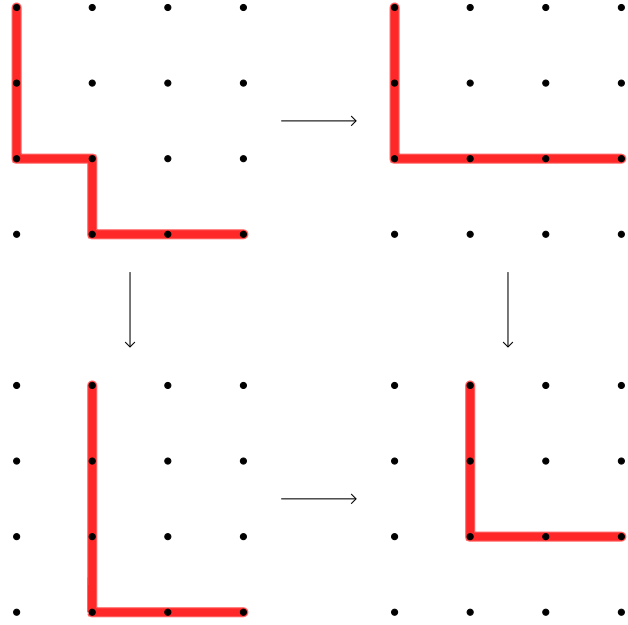
  \end{enumerate}
\end{proof}

\begin{ntt}
  Let $\vec{\mathbf{u}}$ be a non-degenerate simplex in $\Delta^{n}\times \Delta^{m}$ represented by a path in a grid. We call a subsimplex $\vec{\mathbf{u}}_\mathbf{p} \subset \vec{\mathbf{u}}$ a \emph{subpath}, if it arises from a spine inclusion. In terms of paths, the subpath $\vec{\mathbf{u}}_\mathbf{p}$ of $\vec{\mathbf{u}}$ is determined by the start and end points, as it contains everything in between.
\end{ntt}

\begin{ntt}
  Let $\vec{\mathbf{u}}$ be a non-degenerate simplex in $\Delta^{n}\times \Delta^{m}$ represented by a path in a grid. We call a subpath $\vec{\mathbf{u}}_\mathbf{v}=(\mathbf{u}_{i})_{i\in I_{v}}\subset \vec{\mathbf{u}}= (\mathbf{u}_{i})_{i\in I}$ a \emph{vertical component} if
  \begin{itemize}
    \item $u_{i}^{1}=u_{j}^{1}$ for all $i,j \in I_{v}$
    \item $u_{i}^{1} \neq u_{j}^{1}$ for all $i\in I_{v}$ and $j \in I \setminus I_{v}$.
  \end{itemize}
  Analogously, we call a subpath a \emph{horizontal component} if it satisfies the dual conditions.
\end{ntt}

\begin{defn} \label{defn:main_part}
  Let $\vec{\mathbf{u}}$ non-degenerate simplex in $\Delta^{n}_\flat\otimes \Delta^{m}_\flat$ that is not thin. Then, by \zcref{lem:marked_in_flat_gray}, vertical components of $\vec{\mathbf{u}}$, if present, precede the horizontal ones, if present. If both are present, we define the \emph{main part} of $\vec{\mathbf{u}}$ to be the minimal subpath $\vec{\mathbf{u}}_{\main}\subset \vec{\mathbf{u}}$ that contains the last (i.e., maximal) vertical component and first (i.e., minimal) horizontal component of $\vec{\mathbf{u}}$, where the order on the vertical (resp. horizontal) components is induced by that on the vertices a simplex. If no vertical (resp. horizontal) component is present, we define the main part to be the minimal subpath containing the first (resp. last) vertex of $\vec{\mathbf{u}}$ and the first (resp. last) horizontal (resp. vertical) component of $\vec{\mathbf{u}}$. If neither is present, the main part is the whole of $\vec{\mathbf{u}}$.
\end{defn}

\begin{defn}
  Let $\vec{\mathbf{u}}$ be as in \zcref{defn:main_part} with main part $\vec{\mathbf{u}}_{\main}$ and rectilinear paths $M$ and $M_{\main}$, which are images under $L_{n,m}$ of $\vec{\mathbf{u}}$ and $\vec{\mathbf{u}}_{\main}$, respectively. We denote by $\{N_{s}\}_{s\in S}$ the set of minimal L-shaped paths under $M$, i.e., the minimals of $\SD_{M/}$.

  We view any collection of minimal paths $T\subset S$ as an L-shaped path $N_T$, the join of $N_t$ for $t\in T$. We call a collection of minimal paths $T\subset S$ \emph{irrelevant} if the poset $(U_{T})_{M_{\main}/}$ is empty, where $U_T \coloneqq \SD_{N_T/} = \bigcap_{s\in T} \SD_{N_s/}$. We denote by $\Irr(M)\subset \Power_*(S)$ the subset of irrelevant paths and by $\Irr_{1}(M)\subset S$ the subset of irrelevant minimal paths. We denote by $\Rel(M) \subset \Power(S)$  the subset of relevant paths, where a collection of minimal paths is \emph{relevant} if it is not irrelevant.
\end{defn}

\begin{rmk}
  Since $T'\subset T\subset S$ implies $U_{T} \subset U_{T'}$, the set $\Irr(M)$ is closed under taking unions with arbitrary elements of $\Power(S)$.
\end{rmk}

\begin{defn}\label{defn:decoration_L-saturated}
  Let $(\Delta^{n}\otimes \Delta^{m})_{\mathrm{d}}$ be a decorated simplicial set, where the decoration extends that on the Gray tensor product $\Delta^n_\flat \otimes \Delta^m_\flat$. The decoration is said to be \emph{L-saturated} if for every rectilinear path $M$ which is the image under $L_{n,m}$ of a (non-degenerate) thin simplex $\vec{\mathbf{u}}$, the simplex associated to any relevant minimal $N\in \Rel(M)$ is also thin.
\end{defn}

We are now ready to state the main result of this subsection, whose proof, after some preparation, will appear at the end of the subsection.

\begin{prop}~\label{prop:transfer_L-saturated_decorations}
  Let $(\Delta^{n}\otimes \Delta^{m})_{\mathrm{d}}$ be equipped with an L-saturated decoration. Then, the equivalence of \zcref{prop:unmarkedcase} restricts to an equivalence of spaces
  \[
    \Map^{\cart}(\sd(\Delta^{n}\otimes \Delta^{m})_{\mathrm{d}}, \cC) \simeq \Map^{\mix}((\sd(\Delta^{n})\times \sd(\Delta^{m}))_{\mathrm{d}}, \cC), \teq \label{eq:transfer_L-saturated_decoration}
  \]
  where the right-hand side denotes the space of functors satisfying the mixed Segal conditions for $(\sigma,\tau) \in \sd(\Delta^{n})\times \sd(\Delta^{m})$, when $(\sigma,\tau)$ is the image under $Z_{n,m}$ of an L-shaped path in $\sd(\Delta^n\otimes \Delta^m)$ corresponding to a thin simplex in $(\Delta^n \otimes \Delta^m)_\mathrm{d}$ (see also \zcref{prop:mixed_segal_suffices_for_L_shaped}).
\end{prop}

Note that both spaces appearing in \zcref{eq:transfer_L-saturated_decoration} are subspaces of the ones in \zcref{prop:unmarkedcase}. Hence, it suffices to trace, for every decorated simplex $\vec{\mathbf{u}} \in \sd(\Delta^{n}\otimes\Delta^{m})_{\mathrm{d}}$ that is not decorated in $\sd(\Delta^{n}_{\flat} \otimes\Delta_{\flat}^{m})$, the corresponding limit conditions through the equivalence of \zcref{prop:unmarkedcase}.

\begin{conv} \label{conv:L_saturatedness_and_chosen_Segal}
  For the remainder of this subsection, we fix an L-saturated decoration $(\Delta^n \otimes \Delta^m)_\mathrm{d}$, a rectilinear path $M$ corresponding to a non-degenerate thin simplex $\vec{\mathbf{u}}$ of $(\Delta^n\otimes \Delta^m)_\mathrm{d}$ (that is not thin in $\Delta^n\otimes \Delta^m$), and one of the Segal cubes for $\vec{\mathbf{u}}$.

  We denote by $M$ the rectilinear path corresponding to $L_{n,m}(\vec{\mathbf{u}})$ and by $\{M_{i}\}_{i\in K}$ a localized Segal cover. In the following, we adopt the notation from \zcref{prop:SegaliffKE}. In particular, we denote by $\{N_{s}\}_{s\in S}$ the collection of minimal L-shaped paths under $M$.
\end{conv}

The proof of \zcref{prop:transfer_L-saturated_decorations} relies heavily on the following construction.

\begin{const}
  Let $f:R^{n,m}\to \scr{C}$ be a functor which satisfies the localized Segal conditions with respect to the minimal markings on $\Delta^n$ and $\Delta^m$. By \zcref{prop:SegaliffKE}, $f$ is right Kan extended from $\sd(\Delta^n) \times \sd(\Delta^m)$ along the fully faithful functor $\ell_{n,m}$. As in the proof of \zcref{prop:SegaliffKE}, we have the following equivalences
  \[
    \lim_{I\in \Power_*(K)}f(M_I)\simeq \lim_{I\in \Power_*(K)} \lim_{V\in \SD_{M_I/}} f(V)\simeq \lim_{I\in \Power_*(K)}\lim_{T\in \Power_*(S)} \lim_{V\in (U_T)_{M_I/}} f(V)\simeq \lim_{T\in \Power_*(S)}  \lim_{I\in \Power_*(K)}\lim_{V\in (U_T)_{M_I/}} f(V).
  \]
  Using the functoriality of limits, we can extend this to a cubical diagram
  \[
    C_{\vec{\mathbf{u}}}:\Power(S \sqcup K)\simeq \Power(S)\times \Power(K) \to \scr{C}, \quad (T,I) \mapsto \lim_{V\in (U_{T})_{M_{I}/}} f(V)
  \]
  by defining $C_{\vec{\mathbf{u}}}(T,\varnothing)= \lim_{V\in (U_{T})_{M/}} f(V)\simeq f(N_{T})$ and $ C_{\vec{\mathbf{u}}}(\varnothing, I)\simeq \lim_{V\in (\SD)_{M_{I}/}} f(V)\simeq f(M_{I})$. Note that in particular, the restriction $C_{\vec{\mathbf{u}}}(\varnothing,-)$
  agrees with the localized Segal cube of $\vec{\mathbf{u}}$. We call the cube $C_{\vec{\mathbf{u}}}$ the \emph{hyper Segal cube}.

  Under the bijection $\Power(S) \times \Power(K) \simeq \Power(S\sqcup K)$, $(T, I)$ corresponds to $T \cup I$. Thus, when treating the two directions on equal footing, we also use the notation $C_{\vec{\mathbf{u}}}(-)$ instead of $C_{\vec{\mathbf{u}}}(-,-)$.
\end{const}

The idea for the proof of \zcref{prop:transfer_L-saturated_decorations} is to separate the indexing set of the hyper Segal cube into the components corresponding to irrelevant and relevant paths and repeatedly apply \zcref{lem:cubefacelemma} to this situation. To this end, we first need to study the properties of these paths separately.

\begin{rmk}
  As most proofs below can be best understood visually in terms of paths on a grid, we include many figures to aid the readers. Since the localized Segal conditions on $M$ originate from Segal conditions $\vec{\mathbf{u}}$, we draw our figures in terms of $\vec{\mathbf{u}}$ and $\vec{\mathbf{u}}_\main$ rather than $M$ and $M_\main$.
\end{rmk}

\begin{lem}\label{lem:irrelevants constant}
  Let $T \in \Irr(M)$ be an irrelevant path. Then, there exists an index $i_{0} \in K$ such that for all $I \in \Power(K)$, the induced morphism
  \[
    (U_{T})_{M_{I \cup \{i_0\}}/} \rightarrow (U_{T})_{M_{I/}}
  \]
  is an isomorphism of posets.
\end{lem}
\begin{proof}
  The argument is exactly the same as in the proof of \zcref{lem:constantdirectionSegalcube}. As in there, we can assume that $i_0 \notin I$ and moreover, it suffices to consider the case where $T$ is given by a single minimal $N_t$. By definition, if $N_t$ is irrelevant, the minimal $N_t$ either lies completely above or to the right of $M_{\main}$. Schematically, we have the following figure,
  \[
    \begin{tikzpicture}
      \drawGrid{5}{5}
      \begin{scope}[on background layer]
        \begin{scope}[red]
          \drawLine{6pt}{red}  (00.center) -- (11.center) -- (21.center) -- (32.center) -- (33.center) -- (44.center);
          \path (-0.5,0.5) node {\huge$\ddots$};
          \path (4.5,-4.5) node {\huge$\ddots$};
          \path (1.5,-1) node {$A$};
          \path (1.5,-2) node {$B$};
          \path (2,-2.5) node {$C$};
          \path (3,-2.5) node {$D$};
        \end{scope}

        \drawLine{3pt}{cyan} (11.center) -- (21.center) -- (32.center) -- (33.center);

        \drawLine{3pt}{blue}  (00.center) -- (01.center) -- (02.center) -- (03.center) -- (04.center);

        \drawLine{3pt}{green} (04.center) -- (14.center) --(24.center) -- (34.center) -- (44.center);
      \end{scope}
    \end{tikzpicture}
  \]
  where the red (resp. cyan) path represents $\vec{\mathbf{u}}$ (resp. $\vec{\mathbf{u}}_\main$), while the blue and green paths represent irrelevant paths. Note that the Segal cube we are considering must involve either removing $A$ or $B$, and must involve removing either $C$ or $D$. We then have the following cases:
  \begin{itemize}
    \item If $N_{t}$ lies above $M_{\main}$, as illustrated by the blue path in the schematics, removing $A$ or $B$ does not change anything for things under $N_{t}$.
    \item If $N_{t}$ lies to the right of $M_{\main}$, as illustrated by the green path in the schematics, removing $C$ or $D$ does not change anything under $N_{t}$.
  \end{itemize}
\end{proof}

In contrast to the irrelevant paths, the relevant paths behave differently.

\begin{ntt}
  Let $N_t$ be a minimal L-shaped path corresponding to some $t\in S$, which corresponds to a pair of non-degenerate simplices $(\sigma, \tau) \in \sd(\Delta^n) \times \sd(\Delta^m)$. We call the vertex of $N_t$ arising as the image of $(\sigma(\dim \sigma),\tau(0))$ the corner point of $N_t$ and denote it by $\mathbf{c}_t$.
\end{ntt}

Note that the corner point of a minimal $N_t$ agrees with one of the vertices of $\vec{\mathbf{u}}$.

\begin{ntt}
  We call a relevant minimal path $N_{t}$ \emph{Segal} if the chosen Segal cube for $M$ (see \zcref{conv:L_saturatedness_and_chosen_Segal}) does not remove the corner point of $N_{t}$. We denote the collection of Segal minimals by $\Segal(M)$. We call a relevant minimal $N_{t}$ \emph{special} if it is not Segal. We denote the collection of special paths by $\Special(M)$.
\end{ntt}

\begin{lem}\label{lem:Segal paths are limit}
  Let $N_{t}$ be a Segal path for $M$. Then, the $K$-shaped cube $C_{\vec{\mathbf{u}}}(\{t\},-):\Power(K)\rightarrow \cC$ is the outer mixed Segal cube for $N_t$ and hence, is limit, by the L-saturatedness assumption (see \zcref{conv:L_saturatedness_and_chosen_Segal}).
\end{lem}
\begin{proof}
  We consider the following schematic situation:
  \begin{figure}[htb]
    \begin{tikzpicture}
      \drawGrid{5}{5}

      \begin{scope}[on background layer]
        \begin{scope}[red]
          \drawLine{9pt}{red}  (00.center) -- (11.center) -- (21.center) --(32.center) -- (33.center) -- (44.center);
          \path (-0.5,0.5) node {\huge$\ddots$};
          \path (4.5,-4.5) node {\huge$\ddots$};
          \path (1.5,-1) node {$A$};
          \path (2,-2.5) node {$B$};
          \path (4.5,-4) node {$C$};
        \end{scope}

        \drawLine{6pt}{cyan} (11.center) -- (21.center) --(32.center) -- (33.center);

        \drawLine{3pt}{blue} (01.center) -- (11.center) -- (21.center) -- (22.center) -- (23.center) --(24.center);
      \end{scope}
    \end{tikzpicture}
    \caption{Schematic description of the Segal path. Note that removing either of the vertices $A,B$ or $C$ in the Segal condition for $\vec{\mathbf{u}}$ corresponds to a direction in the Segal cube for $N_{t}$.}\label{fig:schema Segal}
  \end{figure}
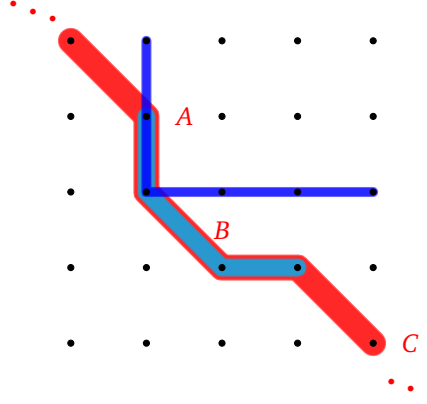
  First note that by assumption on the decoration, the simplex associated to the path $N_{t}$ is itself decorated. Note that in this case for every $I\in \Power(K)$ the slice $(U_t)_{M_I/}$ has an initial object, which we denote by $M_I\cap N_t$, and that the objects $M_I\cap N_t$ form a Segal cover for the L-shaped path $N_t$ (consider \zcref{fig:schema Segal} for a schematic description). In this case we then have
  \[
    \lim_{I\in \Power(K)_\ast}\lim_{V\in (U_t)_{M_I/}} f(V)\simeq \lim_{I\in \Power(K)_\ast} f(M_I\cap N_t)\simeq f(N_t),
  \]
  as desired.
\end{proof}

\begin{defn} \label{defn:linera_order_on_minimals_and_neighbor_vs_stranger}
  The linear order on the vertices of the simplex $\vec{\mathbf{u}}$ induces a linear order on the set of minimals $\{N_{s}\}_{s\in S}$ by declaring $N_{s}\leq N_{t}$ if and only if $\mathbf{c}_{s}\leq \mathbf{c}_{t}$ in the linear order of $\vec{\mathbf{u}}$.

  Let $N_{s},N_{t}$ be two relevant minimal paths with $N_{s}<N_{t}$. We call $N_{s}$ and $N_{t}$ \emph{neighbors} if there exists no vertex $\mathbf{c}_{s}<\mathbf{v} <\mathbf{c}_{t}$ that gets removed by the Segal cube of $\vec{\mathbf{u}}$. Otherwise, we call $N_{s}$ and $N_{t}$ \emph{strangers}.
\end{defn}

\begin{lem}\label{lem:neighbors}
  Let $T$ be a collection of relevant minimal paths such that for some $t_{0},t_{1}\in T$, the corresponding minimals $N_{t_{0}}$ and $N_{t_{1}}$ are stranger. Denote by $\mathbf{c}_{t_{0}}< \mathbf{v}<\mathbf{c_{t_{1}}}$ a vertex that gets removed by the Segal cube for $\vec{\mathbf{u}}$. Then the $K$-shaped cube $C_{\vec{\mathbf{u}}}(T,-):\Power(K)\rightarrow \cC$ is constant in the direction that removes the vertex $\mathbf{v}$.
\end{lem}
\begin{proof}
  Let $t_{\min},t_{\max}\in T$ be the minimals that are furthest apart in the linear order on the vertices of $\vec{\mathbf{u}}$ defined in \zcref{defn:linera_order_on_minimals_and_neighbor_vs_stranger}. Note, that we have an equivalence of posets $(U_{T})_{M_{I}/}\simeq (U_{\{t_{\min},t_{\max}\}})_{M_{I}/}$ (see \zcref{fig:poset_depends_on_minimals}).
  \begin{figure}[htb]
    \begin{tikzpicture}
      \drawGrid{4}{4}

      \begin{scope}[on background layer]
        \drawLine{6pt}{red}  (00.center) -- (10.center) -- (21.center) -- (32.center) -- (33.center);

        \drawLine{3pt}{blue} (00.center) -- (10.center) -- (11.center) -- (12.center) --(13.center);

        \drawLine{3pt}{teal} (01.center) -- (11.center) -- (21.center) -- (22.center) -- (23.center);

        \drawLine{3pt}{green} (02.center) -- (12.center) -- (22.center) -- (32.center) -- (33.center);
      \end{scope}
    \end{tikzpicture}
    \caption{Schematic depiction of the equivalence $(U_{T})_{M_{I}/}\simeq (U_{\{t_{\min},t_{\max}\}})_{M_{I}/}$. Note that the area cut out by blue and green path does not depend on the intermediate red path.}\label{fig:poset_depends_on_minimals}
  \end{figure}
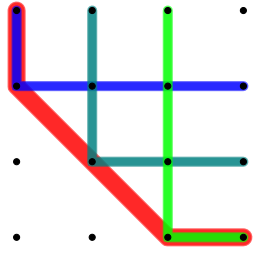
  Hence, it suffices to treat the case that $T=\{t_{0},t_{1}\}$ consists of two strangers. Let $i_{0}\in K$ be the direction that removes the vertex $\mathbf{v}$ in between $\mathbf{c}_{t_{0}},\mathbf{c}_{t_{1}}$. Then it is easy to see that for every $I$, we have an isomorphism of posets
  \[
    (U_{\{t_{0},t_{1}\}})_{M_{I\cup\{i_{0}\}}/} \rightarrow (U_{\{t_{0},t_{1}\}})_{M_{I}/}. \teq \label{eq:strangers}
  \]
  Hence, the cube $C_{\vec{\mathbf{u}}}(T,-)$ is constant in direction $i_{0}$ as desired.
  \begin{figure}[htb]
    \begin{tikzpicture}
      \begin{scope}
        \drawGrid{4}{4}

        \begin{scope}[on background layer]
          \drawLine{6pt}{red} (00.center) -- (10.center) -- (21.center) -- (32.center) -- (33.center);
          \drawLine{3pt}{blue} (00.center) -- (10.center) -- (11.center) -- (12.center) --(13.center);
          \drawLine{3pt}{green} (02.center) -- (12.center) -- (22.center)  --(32.center) -- (33.center);
        \end{scope}

        \node (A) at (-0.5,-1) {$\mathbf{c}_{t_{0}}$};
        \node (B) at (0.5,-2) {$\mathbf{v}$};
        \node (C) at (1.5,-3) {$\mathbf{c}_{t_{1}}$};
      \end{scope}

      \begin{scope}[xshift=6cm]
        \drawGrid{4}{4}

        \begin{scope}[on background layer]
          \drawLine{6pt}{red}  (00.center) -- (10.center) -- (32.center) -- (33.center);
          \drawLine{3pt}{blue} (00.center) -- (10.center) -- (12.center) -- (13.center);
          \drawLine{3pt}{green} (02.center) -- (12.center) -- (32.center) -- (33.center);
        \end{scope}
      \end{scope}

      \draw[->](3.5,-1.5) to (5.5,-1.5);
    \end{tikzpicture}
    \caption{Schematic depiction of the equivalence in \zcref{eq:strangers}. Note that the area to the top right of the blue, green and red path does not change under removing the vertex $\mathbf{v}$}
  \end{figure}
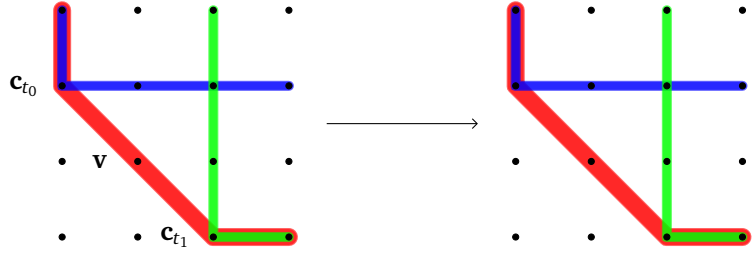
\end{proof}

\begin{lem}\label{lem:special_path}
  Let $N_{t}$ be a special path and denote by $i\in K$ the direction in the cube $C_{\vec{\mathbf{u}}}$ that removes the vertex $\mathbf{c}_{t}$.
  \begin{enumerate}
    \item \label{item:lem:special_path:two_neighbors} If $N_{t}$ has two neighbors $N_{t_{0}},N_{t_{1}}$ then the cube
      \[
        C_{\vec{\mathbf{u}}}(\{t\}\cup-):\Power(\{t_{0},t_{1}\}\cup (K\setminus\{i\})) \rightarrow \cC
      \]
      is the inner mixed Segal cube of $N_{t}$.
    \item \label{item:lem:special_path:one_neighbor} If $N_{t}$ has only one neighbor $N_{t_{0}}$ then the cube
      \[
        C_{\vec{\mathbf{u}}}(\{t\}\cup-):\Power(\{t_{0}\} \cup K) \rightarrow \cC
      \]
      is the inner mixed Segal cube of $t$.
  \end{enumerate}
\end{lem}
\begin{proof}
  We first show \zcref{item:lem:special_path:two_neighbors}. It follows as in \zcref{lem:Segal paths are limit} that the cube
  \[
    C_{\vec{\mathbf{u}}}(\{t\}\cup -):\Power(K\setminus\{i\}) \rightarrow \cC
  \]
  is the sub-cube of the mixed Segal cube that does not contain the corner directions. Hence, it suffices, to show that the directions $t_{0}$ and $t_{1}$ coincide with the corner directions. But it is easy to see that for every $I\in K\setminus \{i\}$ and $j \in \{0,1\}$ the poset $(U_{\{t_{j},t\}})_{M_{I}/}$ has a unique minimal element that recovers the corner direction of the mixed Segal cube (see \zcref{fig:mixed1} for a schematic description).
  \begin{figure}[htb]
    \begin{tikzpicture}
      \begin{scope}
        \drawGrid{4}{4}

        \begin{scope}[on background layer]
          \drawLine{6pt}{red} (00.center) -- (10.center) -- (21.center) -- (32.center) -- (33.center);
          \drawLine{3pt}{green} (02.center) -- (12.center) -- (22.center) -- (32.center) -- (33.center);
          \drawLine{3pt}{teal} (01.center) -- (11.center) -- (21.center) -- (22.center) -- (23.center);
        \end{scope}
      \end{scope}

      \begin{scope}[xshift=6cm]
        \drawGrid{4}{4}

        \begin{scope}[on background layer]
          \drawLine{6pt}{red} (00.center) -- (10.center) -- (32.center) -- (33.center);
          \drawLine{3pt}{teal} (01.center) -- (11.center) -- (21.center) -- (22.center) -- (23.center);
          \drawLine{3pt}{blue} (00.center) -- (10.center) -- (12.center) -- (13.center);
        \end{scope}
      \end{scope}
    \end{tikzpicture}
    \caption{Schematic description of the poset $(U_{\{t_{j},t\}})_{M_{I}/}$ in \zcref{item:lem:special_path:two_neighbors}. Note the intersection of the path always contains a unique minimal, that corresponds to a vertex of the mixed Segal cube.}\label{fig:mixed1}
  \end{figure}
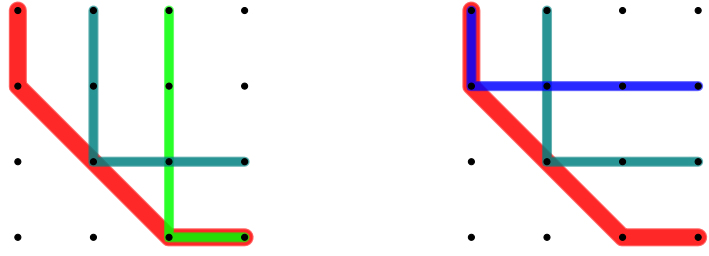

  For \zcref{item:lem:special_path:one_neighbor} we assume that $N_{t}$ has only one neighbor. In this case it suffices to show that the additional directions in the cube $i$ and $t_{0}$ recover the corner condition of the mixed Segal cube. But this can be seen analogously to \zcref{item:lem:special_path:two_neighbors} (see \zcref{fig:mixed2} for a schematic description.)
  \begin{figure}[htb]
    \begin{tikzpicture}
      \begin{scope}
        \drawGrid{4}{4}
        \begin{scope}[on background layer]
          \drawLine{6pt}{red} (00.center) -- (21.center) -- (32.center) -- (33.center);
          \drawLine{3pt}{teal} (00.center) -- (10.center) -- (12.center) -- (13.center);
          \drawLine{3pt}{blue} (00.center) -- (01.center) -- (02.center) -- (03.center);
        \end{scope}
      \end{scope}

      \begin{scope}[xshift=6cm]
        \drawGrid{4}{4}
        \begin{scope}[on background layer]
          \drawLine{6pt}{red} (00.center) -- (10.center) -- (32.center) -- (33.center);
          \drawLine{3pt}{teal} (00.center) -- (10.center) -- (12.center) -- (13.center);
          \drawLine{3pt}{green} (01.center) -- (11.center) -- (21.center) -- (22.center) -- (23.center);
        \end{scope}
      \end{scope}
    \end{tikzpicture}
    \caption{Schematic description of the posets $(U_{\{t\}})_{M_{\{i\}\cup I}/}$ left and $(U_{\{t_{0},t\}})_{M_{I}/}$ right. Note that the poset always contains a unique minimal (drawn in green on the left) that corresponds to a vertex of the mixed Segal cube in the corner direction.}
    \label{fig:mixed2}
  \end{figure}
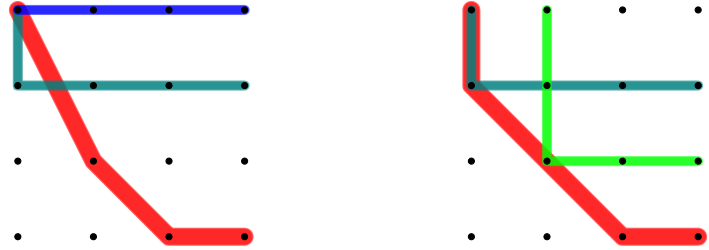
\end{proof}

\begin{proof}[Proof of \zcref{prop:transfer_L-saturated_decorations}] \label{proof:prop:transfer_L-saturated_decorations}
  By \zcref{prop:mixed_segal_suffices_for_L_shaped}, we see that the LHS of \zcref{eq:transfer_L-saturated_decoration} is a subspace of the right. It remains to prove the other inclusion. Note that also by \zcref{prop:mixed_segal_suffices_for_L_shaped} and the L-saturatedness assumption, we already know that the Segal conditions are satisfied for relevant minimal L-shaped paths.

  Fix one of the Segal conditions and consider the hypercube $C_{\vec{\mathbf{u}}}$. As the $K$-shaped cube $C_{\vec{\mathbf{u}}}(\varnothing,-)$ is the localized Segal cube for $M$, our goal is to show that it is a limit cube. Observe that by \zcref{lem:colim_of_USs}, for every $I\in \Power(K)$ the $S$-shaped cube $C_{\vec{\mathbf{u}}}(-,I)$ is a limit cube, and hence, by \zcref{lem:cubefacelemma}, so is the whole cube $C_{\vec{\mathbf{u}}}$. It thus suffices to show that $C_{\vec{\mathbf{u}}}(\varnothing,-)$ is a limit cube if and only if the whole cube $C_{\vec{\mathbf{u}}}$ is. We do so by decomposing $S= \Special(M) \cup \Segal(M) \cup \Irr_{1}(M)$ and iteratively applying \zcref{lem:cubefacelemma}.

  By \zcref{lem:irrelevants constant}, the $K$-shaped cube $C_{\vec{\mathbf{u}}}(T,-)$ is degenerate along a certain direction for every irrelevant collection $T\in \Irr(M)$. Hence, we can conclude from \zcref{lem:cubefacelemma} that the cube $C_{\vec{\mathbf{u}}}(\varnothing, -)$ is limit if and only if
  \[
    C_{\vec{\mathbf{u}}}(-,-): \Power(\Irr_{1}(M)) \times \Power(K) \simeq \Power(\Irr_{1}(M)\cup K) \rightarrow \cC
  \]
  is also limit.

  Our next goal is to show that this cube is limit if and only if the cube
  \[
    C_{\vec{\mathbf{u}}}(-,-):\Power(\Segal(M) \cup \Irr_{1}(M) \cup K) \rightarrow \cC
  \]
  is. By \zcref{lem:cubefacelemma}, it suffices to show that for every non-empty collection of Segal simplices $T\subset \Segal(M)$ the cube
  \[
    C_{\vec{\mathbf{u}}}(T\cup-,-):\Power(\Irr_{1}(M)\cup K) \rightarrow \cC
  \]
  is limit. We need to distinguish the following cases
  \begin{itemize}
    \item if $\#(T)>1$, then by \zcref{lem:neighbors} the cube has a constant direction and hence is limit.
    \item if $\#(T)=1$ we apply \zcref{lem:cubefacelemma} once again. As the intersection with an irrelevant minimal is irrelevant, it suffices to show that the $K$-shaped cube
      $C_{\vec{\mathbf{u}}}(T,-)$ is limit. But this is the statement of \zcref{lem:Segal paths are limit}.
  \end{itemize}

  To summarize, we have shown that the cube $C_{\vec{\mathbf{u}}}(\varnothing, -)$ is limit if and only if the cube that in addition contains all Segal minimals and irrelevant paths is limit. To finish the proof, we once again apply \zcref{lem:cubefacelemma} to the cube
  \[
    C_{\vec{\mathbf{u}}}(-,-):\Power(\Special(M) \cup \Segal(M) \cup \Irr_{1}(M) \cup K) \rightarrow \cC
  \]
  Hence, we need to show that for every non-empty collection of special paths $T \subset \Special(M)$, the cube
  \[
    C_{\vec{\mathbf{u}}}(T\cup-,-): \Power(\Segal(M)\cup \Irr_{1}(M) \cup K) \rightarrow \cC
  \]
  is a limit cube. We consider the following cases
  \begin{itemize}
    \item if $\#(T)>1$ and contains strangers, then the cube has a constant direction, by \zcref{lem:neighbors}, and hence, is limit.
    \item if $\#(T)=2$ consists of two neighbors, then there exists a Segal path $N_{t}$ in between. It follows as in \zcref{lem:neighbors} that we have an isomorphism of posets
      \[
        (U_{T\cup\{t\}})_{M_{I}/} \rightarrow (U_{T})_{M_{I}/}
      \]
      and hence the cube is constant along direction $t$.
    \item if $T=\{t\}$ where $t$ is special for the Segal direction $i\in K$ and $N_{t}$ admits two neighbors $N_{t_{0}},N_{t_{1}} \in \Seg(M)$, we apply \zcref{lem:cubefacelemma} to the decomposition
      \[
        \Segal(M)\cup \Irr_{1}(M)\cup K = ((\Segal(M) \setminus \{t_{0},t_{1}\}) \cup \Irr_{1}(M) \cup\{i\})\sqcup (\{t_{0},t_{1}\} \cup K\setminus \{i\})
      \]
      and the cube
      \[
        C_{\vec{\mathbf{u}}}(\{t\}\cup-,-): \Power(\Segal(M)\cup \Irr_{1}(M)\cup K) \rightarrow \cC.
      \]
      As the underlying cube
      \[
        C_{\vec{\mathbf{u}}}(\{t\}\cup -,-):\Power(\{t_{0},t_{1}\}\cup K\setminus\{i\}) \rightarrow \cC
      \]
      is a limit cube by \zcref{lem:special_path}.\zcref{item:lem:special_path:two_neighbors}. We hence need to show that for every $I\in \Power_*((\Segal(M) \setminus \{t_{0},t_{1}\}) \cup \Irr_{1}(M) \cup\{i\})$ the cube
      \[
        C_{\vec{\mathbf{u}}}(\{t\}\cup I\cup -):\Power(\{t_{0},t_{1}\}\cup K\setminus\{i\}) \rightarrow \cC
      \]
      is a limit cube. We have the following cases
      \begin{itemize}
        \item if $I$ contains an irrelevant path, then the cube admits a constant direction by \zcref{lem:irrelevants constant}. Note that since $N_{t}$ admits two neighbors, the Segal direction corresponding to $i$ we are excluding is not the one used in the proof of the Lemma.
        \item if $I$ contains no irrelevant path but the direction of a Segal path $t_{s}$, then the path $N_{t_{s}}$ is not a neighbor of $N_{t}$ by assumption. Hence, the cube is constant in the direction of the Segal condition between $N_{t}$ and $N_{t_{s}}$ by \zcref{lem:neighbors}.
        \item if $I$ contains no irrelevant path and no Segal condition, it must be of the form $I=\{i\}$ In this case the cube is constant in the direction of any of the neighbors $N_{t_{0}}$ and $N_{t_{1}}$ of $N_{t}$. Indeed, it is easy to see that for every $J \in \Power(K)$ with $i\in J$ and any $T\in \Power((\Segal(M)\setminus\{t_{0},t_1\}) \cup \Irr_{1}(M))$, the map
          \[
            (U_{{T} \cup \{t\}})_{M_{J}/} \rightarrow (U_{T\cup \{t, t_{0}\}})_{M_{J}/}
          \]
          is an isomorphism of posets (consider \zcref{fig:finale} for a schematic description).
      \end{itemize}

      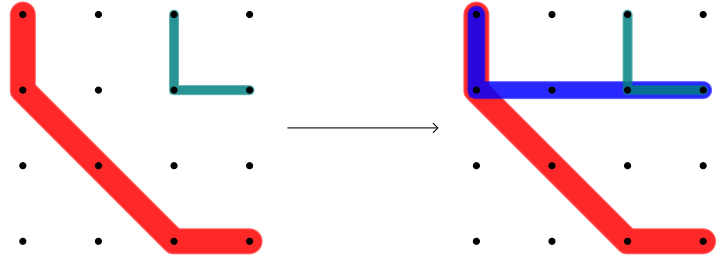
\begin{figure} [htb]
        \begin{tikzpicture}
          \begin{scope}
            \drawGrid{4}{4}
            \begin{scope}[on background layer]
              \drawLine{9pt}{red} (00.center) -- (10.center) --(21.center)  -- (32.center) -- (33.center);
              \drawLine{3pt}{teal} (02.center) -- (12.center) -- (13.center);
            \end{scope}
          \end{scope}

          \begin{scope}[xshift=6cm]
            \drawGrid{4}{4}

            \begin{scope}[on background layer]
              \drawLine{9pt}{red} (00.center) -- (10.center) -- (32.center) -- (33.center);
              \drawLine{6pt}{blue} (00.center) -- (10.center) -- (12.center) -- (13.center);
              \drawLine{3pt}{teal} (02.center) -- (12.center) -- (13.center);
            \end{scope}
          \end{scope};
          \draw[->] (3.5,-1.5) to (5.5,-1.5);
        \end{tikzpicture}
        \caption{Schematic description of the equivalence $(U_{T\cup \{t\}})_{M_{J}/} \rightarrow (U_{T\cup \{t,t_{0}\}})_{M_{J}/}$. Note that intersecting the teal path with the blue path does not change the area to the upper right of both paths. }
        \label{fig:finale}
      \end{figure}
    \item if $T=\{t\}$ and $N_{t}$ admits only one neighbor $N_{t_{0}}$, we apply \zcref{lem:cubefacelemma} to the decomposition
      \[
        \Segal(M)\cup \Irr_{1}(M)\cup K = ((\Segal(M)\setminus\{t_{0}\}) \cup \Irr_{1}(M)) \sqcup (\{t_{0}\}\cup K)
      \]
      and the cube
      \[
        C_{\vec{\mathbf{u}}}(\{t\}\cup -,-):\Power(\Irr_{1}(M)\cup \Segal(M)\cup K) \rightarrow \cC.
      \]
      It follows from \zcref{lem:special_path}.\zcref{item:lem:special_path:one_neighbor} that the cube
      \[
        C_{\vec{\mathbf{u}}}(\{t\}\cup-):\Power(\{t_{0}\}\cup K) \rightarrow \cC
      \]
      is limit. Hence, we need to show that for every $I\in \Power_*((\Segal(M)\setminus\{t_{0}\}) \cup \Irr_{1}(M))$ the cube
      \[
        C_{\vec{\mathbf{u}}}(\{t\}\cup I\cup-):\Power(\{t_{0}\}\cup K) \rightarrow \cC
      \]
      is a limit cube. We have the following cases
      \begin{itemize}
        \item if $I$ contains an irrelevant path, then the cube admits a constant direction by \zcref{lem:irrelevants constant}, and hence, is limit.
        \item So we can assume that $I$ contains no irrelevant path but a Segal path. By assumption there is a Segal path $t_{s}\in I$ that is not a neighbor of $t$. Hence, the cube is constant in the intermediate Segal condition by \zcref{lem:neighbors}.
      \end{itemize}
  \end{itemize}
  This finally finishes the proof.
\end{proof}

\subsection{Decorated simplices}

In this section, we apply \zcref{prop:transfer_L-saturated_decorations} to the cases of $\Delta^{n}_{\flat}\otimes \Delta^{m}_{+}$, $\Delta^{n}_{+}\otimes \Delta^{m}_{\flat}$, and $\Delta^{n}_{+}\otimes \Delta^{m}_{+}$. In what follows, the cases where $n = 1$ or $m=1$ are trivially true. We will thus only explicitly argue the cases where $n,m>1$.

\begin{prop}\label{prop:decorated_one_sided}
  The additional marked simplices in $\Delta^{n}_{+}\otimes \Delta^{m}_{\flat}$ are precisely those non-degenerate $k$-simplices $\vec{\mathbf{u}}$, with $\sqcup_{n,k-n}^{1}\vec{\mathbf{u}}=(0,\dots,n)$ and $u^2_{n-1} = u^2_{n}$.

  Dually, the additional marked simplices in $\Delta^{n}_{\flat}\otimes \Delta^{m}_{+}$ are precisely those non-degenerate $k$-simplices $\vec{\mathbf{u}}$, with $\sqcup_{k-m,m}^{2}\vec{\mathbf{u}}=(0,\dots,m)$ and $u^{1}_{k-m}=u^{1}_{k-m+1}$.
\end{prop}

\begin{proof}
  We consider the first case, as the other one is dual. By definition, these are the only possible additional marked simplices. It therefore suffices to show every simplex $\vec{\mathbf{u}}$ of the above form is marked.

  We consider their images under the complicial parity operators $\sqcup_{p,q}^{1}, \sqcup_{p,q}^{2}$:

  \begin{itemize}
    \item for $p\leq n-1$ the simplex $\sqcup_{p,q}^{2}\vec{\mathbf{u}}=(u_{p}^{2},\dots,u_{n-1}^{2}=u_{n}^{2},\dots, u_{k}^{2})$ is degenerate and hence decorated,
    \item for $p=n$ the simplex $\sqcup^{1}_{p,q}\vec{\mathbf{u}}=(u^{1}_{0},\dots, u^{1}_{n})=(0,\dots,n)$ is decorated,
    \item for $p>n$ the simplex $\sqcup^{1}_{p,q}\vec{\mathbf{u}}$ is degenerate and hence decorated.\qedhere
  \end{itemize}
\end{proof}

Schematically, the additional marked path in $\Delta^{n}_{+}\otimes \Delta^{m}_{\flat}$ start at the top of the grid and strictly monotonically increase their $y$-coordinate by $1$ in each step, until they reach the bottom of the grid by a vertical step, followed by possible horizontal piece. Dually, the additional marked path in $\Delta^{n}_{\flat}\otimes \Delta^{m}_{+}$ start at the left side of the grid, by a possible vertical piece, and then starting by a horizontal step, monotonically increase their $x$-coordinate until they reach the right side of the grid.

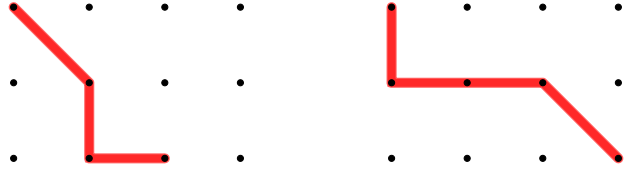
\begin{figure}[htb]
  \label{fig:additional_path}
  \[
    \begin{tikzpicture}
      \begin{scope}
        \drawGrid{4}{3}
        \begin{scope}[on background layer]
          \drawLine{3pt}{red} (00.center) -- (11.center) -- (21.center) -- (22.center);
        \end{scope};
      \end{scope};

      \begin{scope}[xshift=5cm]
        \drawGrid{4}{3}
        \begin{scope}[on background layer]
          \drawLine{3pt}{red} (00.center) -- (10.center) -- (11.center) -- (12.center) -- (23.center);
        \end{scope};
      \end{scope};
    \end{tikzpicture}
  \]
  \caption{Examples of additional marked path in $\Delta^{n}_{+}\otimes \Delta^{m}_{\flat}$ on the left and on $\Delta^{n}_{\flat}\otimes \Delta^{m}_{+}$ the right.}
\end{figure}

\begin{prop} \label{prop:one_sided_decorated_case}
  The equivalence from \zcref{prop:unmarkedcase} induces equivalences of spaces
  \[
    \Map(\Delta^{n}_{\flat}\otimes \Delta^{m}_{+},\Span_{\infty}(\cC)) \simeq \Map^{\cart,\cart}(\sd(\Delta^n_\flat) \times \sd(\Delta^m_+), \mathcal{C}) \simeq \Map(\sd(\Delta^n_\flat), \Fun^\cart(\sd(\Delta^m_+), \mathcal{C})),
  \]
  and
  \[
    \Map(\Delta^{n}_{+}\otimes \Delta^{m}_{\flat},\Span_{\infty}(\cC)) \simeq \Map^{\cart,\cart}(\sd(\Delta^n_+) \times \sd(\Delta^m_\flat), \mathcal{C}) \simeq \Map^\cart(\sd(\Delta^n_+), \Fun(\sd(\Delta^m), \mathcal{C})).
  \]
\end{prop}
\begin{proof}
  Note that the second equivalences on both lines follow from the definition. We will treat the first equivalence on the first line, as the one on the second is completely analogous. Since the main part $\vec{\mathbf{u}}_{\main}$ of every decorated simplex $\vec{\mathbf{u}}$ contains no diagonals, $\vec{\mathbf{u}}_{\main}$ has a unique relevant minimal (see \zcref{fig:onedecorated}). As it is of maximal length in the horizontal direction, it is marked as well. Hence, the decoration is L-saturated and by \zcref{prop:transfer_L-saturated_decorations}, we have an equivalence of spaces
  \[
    \Map(\Delta^n_\flat \otimes \Delta^m_+, \Span_\infty(\mathcal{C})) \simeq \Map^{\mix}(\sd(\Delta^n_\flat) \times \sd(\Delta^m_+), \mathcal{C}).
  \]
  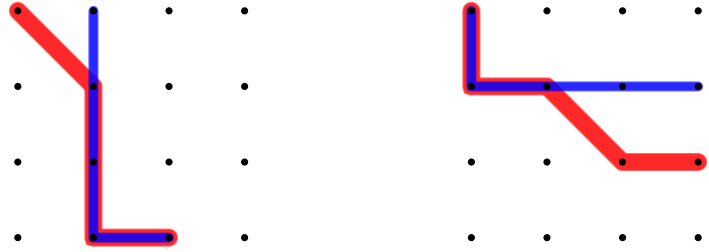
\begin{figure}[htb]
    \begin{tikzpicture}
      \begin{scope}
        \drawGrid{4}{4}

        \begin{scope}[on background layer]
          \drawLine{6pt}{red}  (00.center) -- (11.center) -- (21.center) -- (31.center) -- (32.center);
          \drawLine{3pt}{blue} (01.center) -- (11.center) -- (21.center) -- (31.center) -- (32.center);
        \end{scope}
      \end{scope}

      \begin{scope}[xshift=6cm]
        \drawGrid{4}{4}
        \begin{scope}[on background layer]
          \drawLine{6pt}{red}  (00.center) -- (10.center) -- (11.center) -- (22.center) -- (23.center);
          \drawLine{3pt}{blue} (00.center) -- (10.center) -- (11.center) -- (12.center) -- (13.center);
        \end{scope}
      \end{scope}
    \end{tikzpicture}
    \caption{Decorated path (red) with its unique minimal (blue). Note that the minimal path in the left (resp. right) figure is also of maximal vertical (resp. horizontal) length.}
    \label{fig:onedecorated}
  \end{figure}

  It remains to show that we have an equivalence of spaces
  \[
    \Map^{\mix}(\sd(\Delta^n_\flat) \times \sd(\Delta^m_+), \mathcal{C}) \simeq \Map^{\cart,\cart}(\sd(\Delta^n_\flat) \times \sd(\Delta^m_+), \mathcal{C}).
  \]
  Dropping the $\mix$ and $\cart$ conditions, we have an equivalence of spaces by definition. It remains to show that a functor $F$ satisfies the mixed Segal conditions if and only if it satisfies the Cartesian conditions on the second variable (as the one on the first variable, which is minimally marked, is empty). We will prove the following statement by induction on $1 \leq k \leq n$ (with the $k=n$ case being the desirable conclusion)
  \begin{itemize}
    \item[$(*_k)$] $F(K, -): \sd(\Delta^m_+) \to \mathcal{C}$ is Cartesian for all $K \in \sd(\Delta^n)$ of length at most $k$ if and only if $F$ satisfies the mixed Segal condition for all L-shaped paths of maximal horizontal length and with vertical component of length at most $k$.
  \end{itemize}

  The $k=1$ case is clear since for any $\{v\} \in \sd(\Delta^n)$, $\{v\}\times [m]$ is thin, and the mixed Segal conditions coincide with the Segal conditions on $[m]$. Assume that $(*_{k-1})$ is true and let $(K, [m]) \in \sd(\Delta^n) \times \sd(\Delta^m)$ where $K$ has length $k$. Observe that for each Segal cube on $[m]$, of some $S_{[m]}$-shape, there exists a unique Segal cube on $K$, of some $S_K$-shape, such that
  \[
    C(-,-): \Power(S_K) \times \Power(S_{[m]}) \simeq \Power(S_K \cup S_{[m]}) \to \sd(\Delta^n) \times \sd(\Delta^m) \xrightarrow{f} \mathcal{C}
  \]
  is one of the mixed Segal cubes associated to the pair $(K, [m])$ (and vice versa). It thus suffices to show that $C(\varnothing, -)$ is limit if and only if $C(-,-)$ is limit. By induction, for each $I \subset \Power_*(S_K)$, $C(I, -)$ is limit. The desired conclusion then follows from \zcref{lem:cubefacelemma}.
\end{proof}

It remains to consider the case $\Delta^{n}_{+}\otimes \Delta^{m}_{+}$.
\begin{prop} \label{prop:thin_simplices_in_fully_decorated_case}
  A non-degenerate simplex in $\Delta^{n}_{+}\otimes \Delta^{m}_{+}$ is thin if and only if it is thin in either
  $\Delta^{n}_{+}\otimes \Delta^{m}_{\flat}$ or $\Delta^{n}_{\flat}\otimes \Delta^{m}_{+}$ or is the following $(n+m-1)$-simplex
  \[
    \vec{\mathbf{u}} = ((0,0), (1, 0), (2, 0), \dots, (n-1, 0), (n, 1), (n, 2), \dots, (n,m)).
  \]
\end{prop}
\begin{proof}
  First, we check that the extra non-degenerate simplex is thin. Indeed, this follows from the following observations
  \begin{itemize}
    \item for $p\leq n-2$, $\sqcup^2_{p,q} \vec{\mathbf{u}}$ is degenerate,
    \item $\sqcup^2_{n-1,m} \vec{\mathbf{u}}$ is thin by definition,
    \item $\sqcup^1_{n,m-1} \vec{\mathbf{u}}$ is thin by definition,
    \item for $p\geq n+1$, $\sqcup^1_{p, q} \vec{\mathbf{u}}$ is degenerate.
  \end{itemize}

  To complete the proof, it remains to show that the only possible thin simplex in addition to those in $\Delta^n_+ \otimes \Delta^m_\flat$ and $\Delta^n_\flat \otimes \Delta^m_+$ is given above. Indeed, let $\vec{\mathbf{u}}$ be such a simplex and $k$ the dimension of $\vec{\mathbf{u}}$. Then, $\sqcup^1_{n,k-n} \vec{\mathbf{u}} = (0, 1, \dots, n)$ and $\sqcup^2_{k-m,m} \vec{\mathbf{u}} = (0, 1, \dots, m)$. It suffices to show that $k = n+m-1$. Since $\sqcup^1_{n-1, k-(n-1)} \vec{\mathbf{u}}$ is not thin, $\sqcup^2_{n-1, k-(n-1)} \vec{\mathbf{u}}$ is. But since $\sqcup^2_{p,q} \vec{\mathbf{u}}$ is not thin when $p>k-m$, $n-1 \leq k-m$ and hence, $n+m-1 \leq k$.

  Since $k \leq n+m$, the dimension of the largest non-degenerate simplex  on $\Delta^n \otimes \Delta^m$, it remains to show that $k \neq n+m$. Indeed, when $k=n+m$, $\vec{\mathbf{u}}$ is the maximal L-shaped path, which is already thin in both $\Delta^n_\flat \otimes \Delta^m_+$ and $\Delta^n_+ \otimes \Delta^m_\flat$, and the proof concludes.
\end{proof}

\begin{prop} \label{prop:fully_decorated_case}
  The equivalence from \zcref{prop:unmarkedcase} induces equivalences of spaces
  \[
    \Map(\Delta^{n}_{+}\otimes \Delta^{m}_{+},\Span_{\infty}(\scr{C})) \simeq \Map^{\cart,\cart}(\sd(\Delta^n_+) \times \sd(\Delta^m_+), \mathcal{C}) \simeq \Map^{\cart}(\sd(\Delta^{n}_{+}),\Fun^{\cart}(\sd(\Delta^{m}_{+}),\scr{C})).
  \]
\end{prop}
\begin{proof}
  Observe that the additional thin simplex does not add any new relevant L-shaped minimals (see schematic below). Thus, the decoration on $\Delta^n_+ \otimes \Delta^m_+$ is L-saturated, thanks to \zcref{prop:thin_simplices_in_fully_decorated_case} and the fact that the decorations on $\Delta^n_+ \otimes \Delta^m_\flat$ and $\Delta^n_\flat \otimes \Delta^m_+$ are L-saturated. But now, the result follows from \zcref{prop:one_sided_decorated_case}.

  \begin{figure}[htb]
    \begin{tikzpicture}
      \drawGrid{4}{4}

      \begin{scope}[on background layer]
        \drawLine{6pt}{red} (00.center) -- (10.center) -- (20.center) -- (31.center) -- (32.center) -- (33.center);
        \drawLine{3pt}{blue} (00.center) -- (10.center) -- (20.center) -- (21.center) -- (22.center) -- (23.center);
        \drawLine{3pt}{green} (01.center) -- (11.center) -- (21.center) -- (31.center) -- (32.center) -- (33.center);
      \end{scope}
    \end{tikzpicture}
  \end{figure}

\end{proof}

\section{Span $n$-categories, hom-categories, and other models}
\label{sec:comparison_of_models_and_misc}

As a coda, we establish a number of more basic properties of $\Span_\infty(\scr{C})$. In \zcref{subsec:subcats_Span_infty}, we identify two kinds of span $n$-categories. Namely, for a fixed $n$, $\Span_{n}(\scr{C})$ is the $n$-category whose objects are the objects $\scr{C}$, and morphisms in all dimensions up to and including $n$ are spans in $\scr{C}$. On the other hand $\Span_{n\frac{1}{2}}(\scr{C})$ is the $(n+1)$-subcategory of $\Span_{n+1}(\scr{C})$ where the top-dimensional morphisms are genuine morphisms in $\scr{C}$, which we can identify with left-degenerate spans. Then, in~\zcref{subsec:universal_props_Span_n}, we formulate and prove universal properties for mapping into these categories. In~\zcref{subsec:Hom_in_Span}, we show that, as expected, $\Hom$-categories in these span categories have a recursive description involving span categories in lower dimension and the slice categories. As a final sanity check, in~\zcref{subsec:comparison_models}, we use the universal property of $\Span_\infty(\scr{C})$ to prove that $\Span_n(\scr{C})$ is equivalent to other models of $n$-categories of iterated spans in the literature. In particular, we show that the categories $\Span_{n\frac{1}{2}}(\scr{C})$ agree with those defined in \cite{stefanich_higher_2020}. The base case requires a reformulation of our results in terms of twisted arrow categories, carried out in \zcref{subsec:Span_small}.

\subsection{Subcategories of \texorpdfstring{$\Span_{\infty}(\scr{C})}{Span_∞(𝒞)}$} \label{subsec:subcats_Span_infty}

Having proven that $\Span_{\infty}(\scr{C})$ is an $\infty$-category (i.e., an $(\infty, \infty)$-category, using our implicit-$\infty$ convention), we immediately obtain models for the span $n$-categories.

\begin{defn}
  Let $\scr{C}$ be a category with finite limits. The $n$-category of \emph{$n$-fold spans} in $\scr{C}$ is the maximal $n$-trivial decorated simplicial subset $\Span_n(\scr{C})\subset \Span_{\infty}(\scr{C})$. That is, $\Span_n(\scr{C})$ is the underlying $n$-category of $\Span_{\infty}(\scr{C})$.
\end{defn}

There is, however, another variant of span constructions which will be important to our comparison. In $\Span_n(\scr{C})$ as we have constructed it, morphisms at all levels (from $1$ to $n$) are spans. However, there is also an $(n+1)$-category of $n$-fold spans in $\scr{C}$, in which the top dimensional non-invertible morphisms are simply morphisms in $\scr{C}$ which commute with the defining morphisms of the spans. We will generally denote this $n+1$-category by $\Span_{n\frac{1}{2}}(\scr{C})$.

\begin{defn}
  We call an $n$-simplex $\sd(\Delta^n)\to \scr{C}$ of $\Span_{\infty}(\scr{C})$ \emph{one-sided} if it satisfies the Segal condition which does not contain the $0$-lacuna. Note that the one-sided Segal cover is precisely the one with the same parity as $n+1$.

  For each $n$, we define the decorated simplicial subset
  \[
    \Span_{n\frac{1}{2}}(\scr{C})\subset \Span_{\infty}(\scr{C})
  \]
  to be the maximal one consisting of only one-sided simplices in dimension $n+1$, and only thin simplices in all higher dimensions. The thin simplices are those inherited from $\Span_{\infty}(\scr{C})$.
\end{defn}

\begin{prop} \label{prop:fibrancy_truncated_span_1-half}
  The decorated simplicial subset $\Span_{n\frac{1}{2}}(\scr{C})$ is fibrant.
\end{prop}

\begin{proof}
  Since the complicial thinness extensions and saturation extensions only change the decoration, the lifting property against those follows from the lifting properties for $\Span_\infty(\scr{C})$.

  Consider the complicial inner horn extensions $\Lambda^m_k\to \Delta^m_k$. It is easy to see that only the case where $m=n+1$ or $m=n+2$ needs to be justified. Since every thin $(n+1)$-simplex is, in particular, one-sided, the lift for $\Span_{\infty}(\scr{C})$ suffices here. We are thus left to consider the case $m=n+2$.

  In light of the fibrancy of $\Span_{\infty}(\scr{C})$ it suffices to show that for a map $f:\Delta^{n+2}_k\to \Span_{n\frac{1}{2}}(\scr{C})$, if all $(n+1)$-simplices except $d_k$ are one-sided, then so is $d_k$.  Since we have already established the lifting property against $\Delta^1_\sharp\to E_\sharp$, we can restrict to the case $0<k<n+2$.

  Consider the parity $(n+2)\tmod 2$ Segal sieve $\mathcal{F}_k$ for the $k$-face of $\Delta^{n+2}$. By \zcref{lem:k-vanishing}, the corresponding cube is sent to a limit if and only if the cube corresponding to $\scr{F}_k\ast \{k\}$ is sent to a limit. Now, there are two possibilities, $\mathcal{F}_k$ contains either the $(k-1)$-lacuna or the $(k+1)$-lacuna. Moreover, in the $(k-1)$-lacuna case, by assumption, $k-1\neq 0$.

  We will treat the first case below, as the other case is similar. We have the following pasting move
  \lacunaMove{
    \lacuna{6}{
      $k-3$ & $k-2$ & $\cancel{k-1}$ & $k$ & $k+1$ & $k+2$ \\ \hline
      $\varnothing$ \\
      & & & $\varnothing$ \\
      & & & & & $\varnothing$
    }
  }{
    \lacuna{6}{
      $k-3$ & $k-2$ & $k-1$ & $k$ & $k+1$ & $k+2$ \\ \hline
      $\varnothing$ \\
      & & $\varnothing$ & $\varnothing$ \\
      & & & & & $\varnothing$
    }
  }{
    \lacuna{6}{
      $k-3$ & $k-2$ & $k-1$ & $k$ & $k+1$ & $k+2$ \\ \hline
      $\varnothing$ \\
      & & $\varnothing$ \\
      & & & & & $\varnothing$
    }
  }
  where the sieve on the left is $\mathcal{F}_k \ast \{k\}$ and the cube associated to the mediating sieve is sent to a limit because the $(k-1)$-face is one-sided.

  The limit condition for the right sieve follows from the cube move below
  \lacunaMove{
    \lacuna{6}{
      $k-3$ & $k-2$ & $k-1$ & $k$ & $\cancel{k+1}$ & $k+2$ \\ \hline
      $\varnothing$ \\
      & & $\varnothing$ \\
      & & & & & $\varnothing$
    }
  }{
    \lacuna{6}{
      $k-3$ & $k-2$ & $k-1$ & $k$ & $k+1$ & $k+2$ \\ \hline
      $\varnothing$ \\
      & & $\varnothing$ \\
      & & & & & $\varnothing$
    }
  }{
    \lacuna{6}{
      $k-3$ & $k-2$ & $k-1$ & $k$ & $k+1$ & $k+2$ \\ \hline
      $\varnothing$ \\
      & & $\varnothing$ \\
      & & & & $\varnothing$ \\
      & & & & & $\varnothing$
    }
  }
  where the limit condition for the intermediate sieve (resp. the sieve on the right) follows from the one-sided assumption on the $(k+1)$-face (resp. from \zcref{rmk:sheaf_broad_sieve}). The proof thus concludes.
\end{proof}

We thus obtain a filtration of $\Span_\infty (\scr{C})$
\[
  \Span_0(\scr{C}) \subset \Span_{\frac{1}{2}}(\scr{C}) \subset \Span_1(\scr{C}) \subset \Span_{1\frac{1}{2}}(\scr{C})\subset \cdots
\]
Note that there are canonical equivalences $\Span_{\frac{1}{2}}(\scr{C})\simeq \scr{C}$ induced by the first-vertex maps $\sd{(\Delta^n)}\to \Delta^n$ which restrict to equivalences $\Span_0(\scr{C})\simeq \scr{C}^\simeq$.

\begin{lem} \label{lem:span_infty_colimit_span_lower}
  We have the following equivalences in $\Cat_\infty$
  \begin{align*}
    \Span_\infty(\mathcal{C})
    &\simeq \colim(\Span_0(\mathcal{C}) \to \Span_1(\mathcal{C}) \to \Span_2(\mathcal{C}) \to \cdots) \\
    &\simeq \colim(\Span_0(\mathcal{C}) \to \Span_{\frac{1}{2}}(\mathcal{C}) \to \Span_1(\mathcal{C}) \to \cdots) \\
    &\simeq \colim(\Span_{\frac{1}{2}}(\mathcal{C}) \to \Span_{1\frac{1}{2}}(\mathcal{C}) \to \Span_{2\frac{2}{3}}(\mathcal{C}) \to \cdots).
  \end{align*}
\end{lem}
\begin{proof}
  The first is a direct consequence of \cite[Proposition 2.2.1.53]{loubaton_categorical_2024}. The remaining two equivalences are due to cofinality.
\end{proof}

\subsection{Universal properties for span $n$-categories} \label{subsec:universal_props_Span_n}

We now briefly explore the implications of our two main results in the context of the subcategories $\Span_k(\scr{C})$ and $\Span_{k\frac{1}{2}}(\scr{C})$. Both of these turn out to be obtained as a special instance of \zcref{prop:transfer_L-saturated_decorations}.

\begin{defn} \label{defn:truncation_and_friends}
  We write
  \[
    \begin{tikzcd}
      (-)^{\leq k} : &[-3em] \Cat_\infty \arrow[r] & \Cat_k
    \end{tikzcd}
  \]
  for the right adjoint to the inclusion $\iota_k:\Cat_k\to \Cat_\infty$ and
  \[
    \begin{tikzcd}
      \tau_k: &[-3em] \Cat_\infty \arrow[r] & \Cat_k
    \end{tikzcd}
  \]
  for the left adjoint to the inclusion. Note that $\Span_k(\scr{C})=\Span_\infty(\scr{C})^{\leq k}$.
\end{defn}

Our aim is now to identify the subcategory
\[
  \Fun^{\oplax}(\scr{D},\Span_\infty(\scr{C})^{\leq k}) \subset \Fun^{\oplax}(\scr{D},\Span_\infty(\scr{C}))
\]
with a subcategory of $\Span_\infty(\Fun^{\cart}(\SD(\scr{D}),\scr{C})).$

To do this, we again work with the objects of $\tDelta$. As we will see, it will suffice to consider the minimally decorated case.

\begin{defn}
  For a decorated simplicial set $(X,tX)$ and an integer $k\geq 0$, we denote by $k^\sharp(X,tX)$ the decorated simplicial set obtained by adding all simplices of dimension greater than $k$ to the decoration $tX$.
\end{defn}

\begin{rmk}
  If we abusively denote by $\scr{D}$ a fibrant decorated simplicial set and the corresponding $\infty$-category, then, by \cite[2.2.1.15]{loubaton_complicial_2024}, $\tau_k(\scr{D})$ is modelled complicially by $k^\sharp(\scr{D})$. It follows that for a $1$-category $\scr{C}$ with finite limits, there is a natural equivalence
  \[
    \Map(\Delta^n_\flat,  \Fun^{\oplax}(\Delta^m_\flat,\Span_\infty(\scr{C})^{\leq k}))\simeq \Map(k^\sharp(\Delta^n_\flat\otimes \Delta^m_\flat),\Span_\infty(\scr{C})).
  \]
\end{rmk}

\begin{lem}\label{lem:dec_simps_trunc_gray}
  The decorated simplices of $\tau_k(\Delta^n_\flat\otimes \Delta^m_\flat)$ which are not decorated in $\Delta^n_\flat\otimes \Delta^m_\flat$ are precisely those nondegenerate $\vec{\mathbf{u}}$ of dimension greater than $k$ such that, if there is an $i$ such that $u^1_i=u^1_{i+1}$, then for all $j\geq i$, $u^{2}_{j}\neq u^{2}_{j+1}$.
\end{lem}

\begin{proof}
  Immediate from the definitions and \zcref{lem:marked_in_flat_gray}.
\end{proof}

\begin{defn} \label{defn:k-truncated_mixed}
  We call the decoration of \zcref{lem:dec_simps_trunc_gray} the \emph{$k$-truncated decoration} of $\Delta^n\otimes \Delta^m$.

  We call the set of mixed Segal cubes in $\sd(\Delta^{n}) \times \sd(\Delta^{m})$ associated to all pair of simplices $\sigma:\Delta^{l}\rightarrow \Delta^{n}$ and $\tau:\Delta^{j}\rightarrow \Delta^{m}$ with $l+j> k$, the $k$-truncated mixed Segal cubes.

  We say a functor $f:\sd(\Delta^{n})\times \sd(\Delta^{m})\rightarrow \cC$ satisfies the $k$-truncated mixed Segal conditions if it sends all $k$-truncated mixed Segal cubes to limit cubes.
\end{defn}

The following lemma is immediate.

\begin{lem}\label{lem:ktrunc_dec_L-saturated}
  The $k$-truncated decoration of $\Delta^n\otimes \Delta^m$ is $L$-saturated in the sense of \zcref{defn:decoration_L-saturated}.
\end{lem}
\begin{proof}
  We need to show that for every simplex $\vec{\mathbf{u}}$ of length greater than $k$, all relevant minimals also have length greater than $k$. But this is clear, as every relevant minimal $N_{t}$ has the same length as $\vec{\mathbf{u}}$ (see \zcref{fig:k-truncated} for a schematic description)

  \begin{figure}[htb]
    \begin{tikzpicture}
      \drawGrid{5}{5}

      \begin{scope}[on background layer]
        \begin{scope}[red]
          \drawLine{9pt}{red}  (00.center) -- (11.center) -- (21.center) --(32.center) -- (33.center) -- (44.center);
          \path (-0.5,0.5) node {\huge$\ddots$};
          \path (4.5,-4.5) node {\huge$\ddots$};
          \path (1.5,-1) node {$A$};
          \path (2,-2.5) node {$B$};
          \path (4.5,-4) node {$C$};
        \end{scope}

        \drawLine{6pt}{cyan} (11.center) -- (21.center) --(32.center) -- (33.center);

        \drawLine{3pt}{blue} (01.center) -- (11.center) -- (21.center) -- (22.center) -- (23.center) --(24.center);
      \end{scope}
    \end{tikzpicture}
    \caption{Note that the relevant minimal $N_{t}$ (blue) has the same length as the original simplex $\vec{\mathbf{u}}$ (red). Note that the Segal condition for the original simplex that includes forgetting vertices $A$,$B$ and $C$ induces a Segal condition of the same parity for the relevant minimal.}
    \label{fig:k-truncated}
  \end{figure}
\end{proof}

\begin{cor}\label{cor:trunc_map_equiv}
  Let $k\geq 1$. The equivalence of \zcref{prop:unmarkedcase} restricts to an equivalence of spaces
  \[
    \Map((\Delta^{n}_{\flat}\otimes \Delta_{\flat}^{m}),\Span_{k}(\cC)) \simeq \Map^{\mix}((\sd(\Delta^{n}_{\flat}) \times \sd(\Delta^{m}_{\flat}))_{\leq k},\cC)
  \]
  where the right-hand side denotes the space of those functors that satisfy the $k$-truncated mixed Segal conditions.
\end{cor}
\begin{proof}
  Direct from \zcref{prop:transfer_L-saturated_decorations}.
\end{proof}

\begin{defn} \label{defn:Span_k_mix}
  Let $\scr{C}$ be a $1$-category with finite limits, and $\scr{D}$ an $\infty$-category. We define a subobject $\Span_k^{\mix}(\Fun^{\cart}(\SD(\tau_k(\scr{D})),\scr{C}))$ of $\Span_\infty(\Fun^{\cart}(\SD(\scr{D}),\scr{C}))$ in $\Spc_{\tDelta}$ as follows.

  We declare
  \[
    \Map(\Delta^n_t,\Span_k^{\mix}(\Fun^{\cart}(\SD(\tau_k(\scr{D})),\scr{C})))\subset \Map^{\cart,\cart}(\sd(\Delta^n_t)\times\SD(\scr{D}),\scr{C})
  \]
  to be the full subspace spanned by functors such that for every $\Delta^m_s\to \scr{D}$, the corresponding composite
  \[
    \begin{tikzcd}
      \sd(\Delta^n_t)\times \sd(\Delta^m_s)\arrow[r] &  \sd(\Delta^n_t)\times\SD(\scr{D})\arrow[r] & \scr{C}
    \end{tikzcd}
  \]
  satisfies the $k$-truncated mixed Segal conditions. Note that $\Span_n^{\mix}(\Fun^{\cart}(\SD(\tau_k(\scr{D})),\scr{C}))$ is, in fact, a subobject of $\Span_n(\Fun^{\cart}(\SD(\tau_k(\scr{D})),\scr{C}))$.
\end{defn}

\begin{thm} \label{thm:universal_property_Span_k}
  Let $\scr{C}$ be a $1$-category with finite limits and $\scr{D}$ an $\infty$-category. Then, there is a natural equivalence
  \[
    \Span_k^{\mix}(\Fun^{\cart}(\SD(\tau_k(\scr{D})),\scr{C}))\simeq \Fun^{\oplax}(\scr{D},\Span_k(\scr{C})).
  \]
  In particular $\Span_k^{\mix}(\Fun^{\cart}(\SD(\tau_k(\scr{D})),\scr{C}))$ is a $k$-category.
\end{thm}
\begin{proof}
  This is a direct consequence of \zcref{cor:trunc_map_equiv} and the density of $\tDelta$ in $\Spc_{\tDelta}$ from \zcref{lem:ktrunc_dec_L-saturated}, similar to how \zcref{thm:functor_UP_inftycat} follows from \zcref{prop:functor_UP_prop}.
\end{proof}

It now only remains for us to provide  universal properties for the $(k+1)$-categories $\Span_{k\frac{1}{2}}(\mathcal{C})$. To this end, we introduce one-sided versions of our $k$-truncated mixed Segal conditions.

\begin{defn} \label{defn:one_sided_k-truncated_mixed}
  We call the set of cubes in $\sd(\Delta^{n})\times \sd(\Delta^{m})$ consisting of
  \begin{itemize}
    \item the mixed Segal cubes associated to all pairs of simplices $\sigma:\Delta^{l}\to \Delta^{n}$ and $\tau:\Delta^{j} \to \Delta^{m}$ with $l+j>k+1$,
    \item the inner mixed Segal cube associated to all pairs of simplices $\sigma:\Delta^{l}\to \Delta^{n}$ and $\tau:\Delta^{j} \to \Delta^{m}$ with $l+j=k+1$ and $l$ odd,
    \item the outer mixed Segal condition associated to all pairs of simplices $\sigma:\Delta^{l}\to \Delta^{n}$ and $\tau:\Delta^{j} \to \Delta^{m}$ with $l+j=k+1$ and $l$ even
  \end{itemize}
  the \emph{one-sided $k$-truncated mixed Segal cubes}. We say that a map $f:\sd(\Delta^{n})\times \sd(\Delta^{m})\to \cC$ satisfies the \emph{one-sided $k$-truncated mixed Segal conditions} if it maps every one-sided $k$-truncated mixed Segal cube to a limit cube.
\end{defn}

\begin{cor}
  Let $k\geq 1$. The equivalence of \zcref{prop:unmarkedcase} restricts to an equivalence of spaces
  \[
    \Map((\Delta^{n}_{\flat}\otimes \Delta_{\flat}^{m}),\Span_{k\frac{1}{2}}(\cC)) \simeq \Map^{\mix}((\sd(\Delta^{n}_{\flat})\times \sd(\Delta^{m}_{\flat}))_{\leq k\frac{1}{2}},\cC,)
  \]
  where the right-hand side denotes the space of those functors that satisfy the one-sided $k$-truncated mixed Segal conditions.
\end{cor}

\begin{proof}
  Let $\vec{\mathbf{u}}$ be a decorated simplex and choose one of its Segal cubes. Note that for every relevant minimal $N_{t}$ of $\vec{\mathbf{u}}$, the mixed Segal condition associated to $N_{t}$ used in the proof of \zcref{prop:transfer_L-saturated_decorations} is of the same parity as the chosen Segal condition for $\vec{\mathbf{u}}$ (compare \zcref{fig:k-truncated}). Hence, the result follows from the same argument as \zcref{prop:transfer_L-saturated_decorations}.
\end{proof}

Defining $\Span_{k\frac{1}{2}}^{\mix}(\Fun^{\cart}(\SD(\tau_{k+1}(\mathcal{D})), \mathcal{C}))$ as in \zcref{defn:Span_k_mix} but using \zcref{defn:one_sided_k-truncated_mixed} instead of \zcref{defn:k-truncated_mixed}, we obtain the following theorem.

\begin{thm} \label{thm:universal_property_Span_k_1/2}
  Let $\scr{C}$ be a $1$-category with finite limits and $\scr{D}$ an $\infty$-category. Then, there is a natural equivalence
  \[
    \Span_{k\frac{1}{2}}^{\mix}(\Fun^{\cart}(\SD(\tau_{k+1}(\scr{D})),\scr{C}))\simeq \Fun^{\oplax}(\scr{D},\Span_{k\frac{1}{2}}(\scr{C})).
  \]
  In particular $\Span_{k\frac{1}{2}}^{\mix}(\Fun^{\cart}(\SD(\tau_k(\scr{D})),\scr{C}))$ is a $(k+1)$-category.
\end{thm}

\subsection{Hom-\texorpdfstring{$(\infty,n)$}{(∞,n)}-categories} \label{subsec:Hom_in_Span}

To identify the hom categories in spans, we make use of the well known suspension-hom adjunction (see, e.g., \cite[4.3.21]{gepner_enriched_2015}), as well as the model for the suspension provided by \cite[Construction 2.2.2.15]{loubaton_complicial_2024}, which we now recall.

\begin{defn}
  Let $(K,tK)$ be a decorated simplicial set. The \emph{suspension} of $(K,tK)$ is the pushout
  \[
    \Sigma (K,tK)\coloneqq \Delta^0\coprod_{(K,tK)} ((K,tK)\star \Delta^0).
  \]
  For an $(\infty,n)$-category $\scr{D}$ containing objects $x$ and $y$, the \emph{mapping $(\infty,n)$-category} $\scr{D}(x,y)$  is defined by the adjunction
  \[
    \Map_{\Cat_\infty^{\ast,\ast}}(\Sigma (K,tK),\scr{D}_{x,y})\simeq \Map_{\Cat_\infty}((K,tK),\scr{D}(x,y))
  \]
  where $\Cat_\infty^{\ast,\ast}$ denotes the category of bipointed $\infty$-categories, and $\scr{D}_{x,y}$ denotes $\scr{D}$ considered as a bipointed $\infty$-category via the points $x$ and $y$.
  In the context where $\scr{D}$ is presented by a fibrant complicial set, this can also be viewed as a strict Quillen adjunction.
\end{defn}

\begin{thm}
  For an $(\infty,1)$-category with enough limits, and $x,y\in \scr{C}$, there is an equivalence
  \[
    \Span_{\infty}(\scr{C})(x,y)\simeq \Span_{\infty}(\scr{C}_{/x\times y}).
  \]
  On $\Span_n(\mathcal{C})$, this induces an equivalence
  \[
    \Span_n(\scr{C})(x,y)\simeq \Span_{n-1}(\scr{C}_{/x\times y}).
  \]
  On $\Span_{n\frac{1}{2}} (\mathcal{C})$, this also induces an equivalence
  \[
    \Span_{n\frac{1}{2}}(\scr{C})(x,y)\simeq \Span_{(n-1)\frac{1}{2}}(\scr{C}_{/x\times y}).
  \]
\end{thm}

\begin{proof}
  Considering the defining adjunction
  \[
    \Map_{\Cat_\infty^{\ast,\ast}}(\Sigma (K,tK), \Span_{\infty}(\scr{C})_{x,y}) \simeq \Map_{\Cat_\infty}((K,tK), \Span_{\infty}(\scr{C})(x,y)).
  \]
  By the density of $\tDelta$ in $\Cat_\infty$ and the fact that $\Sigma$ preserves (homotopy) colimits, it suffices to provide the comparison for $(K,tK)\in \tDelta$. Passing through the defining adjunctions of $\Span_\infty$, the product, and the slice quasi-category, it will thus suffice for us to show
  \[
    \Map^{\cart}(\Delta^0\coprod_{\sd(\Delta^n_t)}\sd(\Delta^n_t\star\Delta^0),\scr{C})\simeq \Map^{\cart}(\sd{(\Delta^n_t)}\star \partial \Delta^1,\scr{C}).
  \]
  Here, the Cartesian condition on the left is defined via the equivalence
  \[
    \Delta^0 \coprod_{\sd(\Delta^n_t)} \sd(\Delta^n_t \star \Delta^0) \simeq \sd(\Delta^0 \coprod_{\Delta^n_t} \Delta^n_t \star \Delta^0).
  \]
  The Cartesian condition on the right is simply defined to be the Cartesian condition on the composite
  \[
    \begin{tikzcd}
      \sd(\Delta^n_t)\arrow[r] & \sd(\Delta^n_t)\star\partial \Delta^1 \arrow[r] & \scr{C}
    \end{tikzcd}
  \]
  since limits of non-empty connected diagrams in the slice $\mathcal{C}_{/x\times y}$ can be computed in $\mathcal{C}$.

  The natural maps of \zcref{const:mapcatloc-nat_trans} provide natural maps
  \[
    \begin{tikzcd}
      \Map(\sd{(\Delta^n_t)}\star \partial \Delta^1,\scr{C})\arrow[r] & \Map(\Delta^0\coprod_{\sd(\Delta^n_t)}\sd(\Delta^n_t\star\Delta^0),\scr{C})
    \end{tikzcd} \teq \label{eq:mapcatloc-natmap}
  \]
  which are equivalences by \zcref{lem:mapcatloc-Joyal_equivalence}. It thus only remains for us to check that these restrict to equivalences on spaces of Cartesian functors.

  \zcref{lem:mapcatloc-Segal_conditions} then shows that the natural equivalence of (\ref{eq:mapcatloc-natmap}) restricts to a natural equivalence
  \[
    \Map^{\cart}(\Delta^0\coprod_{\sd(\Delta^n_t)}\sd(\Delta^n_t\star\Delta^0),\scr{C})\simeq \Map^{\cart}(\sd{(\Delta^n_t)}\star \partial \Delta^1,\scr{C}),
  \]
  completing the proof.
\end{proof}

\begin{const}\label{const:mapcatloc-nat_trans}
  We construct natural maps of posets
  \[
    \begin{tikzcd}
      \xi_{n}:&[-3em] \sd(\Delta^n\star \Delta^0) \arrow[r] & \sd(\Delta^n)\star \partial \Delta^1
    \end{tikzcd}
  \]
  as follows. Denote the terminal vertex of $\Delta^n\star \Delta^0$ by $v$, and denote the two vertices appended by join in $\sd(\Delta^n)\star \partial \Delta^1$ as $w_0$ and $w_1$. Given a subset $U\subset \Delta^n\star \Delta^0$ representing a non-degenerate simplex, we define
  \[
    \xi_n(U)=
    \begin{cases}
      w_0 & v\notin U \\
      w_1 & U=\{v\}\\
      U\setminus \{v\} & \text{else}.
    \end{cases}
  \]
  It is immediate that this is a natural map of bipointed simplicial sets.

  Since $\xi_n$ collapses the image of the inclusion $\sd(\Delta^n)\to \sd(\Delta^n\star \Delta^0)$, it descends to a natural map of simplicial sets
  \[
    \begin{tikzcd}
      \psi_n: &[-3em] \Delta^0\coprod_{\sd(\Delta^n)} \sd(\Delta^n\star \Delta^0) \arrow[r] & \sd(\Delta^n)\star \partial \Delta^1
    \end{tikzcd}
  \]
  which we denote by $\psi_n$.
\end{const}

\begin{lem}\label{lem:mapcatloc-Joyal_equivalence}
  For every $n$, the map
  \[
    \begin{tikzcd}
      \psi_n: &[-3em] \Delta^0\coprod_{\sd(\Delta^n)} \sd(\Delta^n\star \Delta^0) \arrow[r] & \sd(\Delta^n)\star \partial \Delta^1
    \end{tikzcd}
  \]
  is a Joyal equivalence.
\end{lem}

\begin{proof}
  We equivalently show that the $\psi_n$ are equivalences of marked simplicial sets with minimal marking.

  We first show that $\xi_n$ is a localization map, showing that the criterion of \cite[Proposition 5.0.6]{abellan_garcia_theorem_2022} applies to $\sd(\Delta^n\star \Delta^0)_{a/}$. We have three cases to consider, though in all, the weak fibers are contractible.
  \begin{enumerate}
    \item When $a=U\in \sd(\Delta^n)$, the slice $\sd(\Delta^n\star \Delta^0)_{U/}$ can be identified with the full subposet of $\sd(\Delta^n\star \Delta^0)$ on the objects which are subsets of $U\cup \{v\}$. The weak fiber is the one-point poset on $U\cup \{v\}$, which is an initial object of the slice, and so the criterion holds.
    \item When $a=w_0$, the slice is isomorphic to $\sd(\Delta^n)$, as is the weak fiber. Since identity maps are coinitial, the criterion holds.
    \item When $a=w_1$, both the slice and the weak fiber are one-point posets, and so again the criterion holds.
  \end{enumerate}

  Since $\xi_n$ is a localization at the morphisms of $\sd(\Delta^n)$, it suffices to show that the morphism
  \[
    \begin{tikzcd}
      \sd(\Delta^n)^\sharp\coprod_{\sd(\Delta^n)^\flat}\sd(\Delta^n\star\Delta^0)^\flat \arrow[r] & \left(\Delta^0\coprod_{\sd(\Delta^n)} \sd(\Delta^n\star \Delta^0)\right)^\flat
    \end{tikzcd}
  \]
  is a weak equivalence of marked simplicial sets. Since both pushouts are pushouts of diagrams consisting of cofibrant objects in which one leg is a cofibration and since the marked model structure is left proper, both are homotopy pushouts. But as the map is induced by a natural equivalence, it is itself an equivalence, completing the proof.
\end{proof}

\begin{lem}\label{lem:mapcatloc-Segal_conditions}
  Let $\Delta^n_t$ be a decorated simplex. Then a functor
  \[
    \begin{tikzcd}
      f: &[-3em] \sd(\Delta^n_t)\star \partial \Delta^1 \arrow[r] & \scr{C}
    \end{tikzcd}
  \]
  is Cartesian if and only if $f\circ \psi_n$ is Cartesian.
\end{lem}

\begin{proof}
  By the definition of the join, there is a bijection between thin simplices $U$ of $\sd(\Delta^n_t)\subset \sd(\Delta^n_t)\star\partial \Delta^1$, and thin simplices $U\star \{v\}$ of $\sd(\Delta^n_t \star\Delta^0)$. We can thus compare the Segal conditions simplex-by-simplex.

  Given a thin simplex $U$ of $\sd(\Delta^n_t)$, there are two Segal conditions associated to $U\star\{v\}$ the one containing the $v$-lacuna, and the one not containing the $v$-lacuna.

  The Segal condition containing the $v$-lacuna gives rise to a cube in $\Delta^0\coprod_{\sd(\Delta^n_t)}\sd(\Delta^n_t\star \Delta^0)$ whose face corresponding to the $v$-lacuna is constant, and whose face containing $v$ is mapped isomorphically under $\psi_n$ to the corresponding Segal cube in $\sd(\Delta^n_t)$. Thus the $v$-lacuna Segal condition is satisfied by $f\circ \psi_n$ if and only if the corresponding Segal condition is satisfied by $f$.

  The Segal condition not containing the $v$-lacuna gives rise to a Segal cube in $\sd(\Delta^n_t\star\Delta^0)$ all of whose vertices contain $v$. As such, it is mapped isomorphically to the corresponding Segal cube in $\sd(\Delta^n_t)\star\partial\Delta^1$, completing the proof.
\end{proof}

\subsection{$\Span_1$, $\Span_{1\frac{1}{2}}$, and twisted arrows} \label{subsec:Span_small}

We briefly digress to establish the relation of $\Span_1(\scr{C})$ to existing approaches to span $(\infty,1)$-categories. Along the way, we also will display a new approach to the $(\infty,2)$-category $\Span_{1\frac{1}{2}}(\scr{C})$.

\begin{defn}
  We define an \emph{interval} in $[n]$ to be a subset of the form
  \[
    \langle i,j\rangle =\{i,i+1,\ldots, j\}
  \]
  for $0\leq i\leq j\leq n$.

  For every $n\geq 0$, we define the \emph{twisted arrow category of $[n]$}, denoted by $\Tw([n])$, to be the poset of intervals in $[n]$, ordered by reverse inclusion.
\end{defn}

For every $0\leq i \leq k\leq j \leq n$, the inclusion $\{i,j,k\} \hookrightarrow [n]$ induces a functor $\Tw([2])\hookrightarrow\Tw([n])$. We call a functor $f:\Tw([n])\rightarrow \cC$ \emph{adequate} if for every $0\leq i \leq j \leq k\leq n$ the composite
\[
  f\vert_{\{i,j,k\}}:\Tw([2])\hookrightarrow \Tw([n]) \rightarrow \cC
\]
is a limit diagram.

It was shown in \cite{haugseng_two-variable_2023} that there exists an category $\Span_{1}^{B}(\cC)$ with the universal property
\[
  \Map([n],\Span_{1}^{B}(\cC)) \simeq \Map^{\adq}(\Tw([n]),\cC),
\]
where the right hand side denotes the subspace of adequate functor. Our first aim is to compare this version of the category of spans to $\Span_1(\cC)$.

\begin{const}
  We construct natural functors $\kappa_n:\sd(\Delta^n)\to \Tw([n])$ as follows. Given $U\in \sd(\Delta^n)$, $\kappa_n(U)$ is the smallest interval containing $U$.
\end{const}

\begin{rmk}\label{rmk:kappa_adequate}
  It is easy to see that a functor $f:\Tw([2])\to \scr{C}$ is adequate if and only if $f\circ \kappa_n$ satisfies the even $1$-Segal conditions.
\end{rmk}

\begin{lem}\label{lem:twisted}
  For every $n\geq 0$, $\kappa_n$ is a localization at the set of endpoint-preserving morphisms.
\end{lem}

\begin{proof}
  The endpoint-preserving morphisms are precisely those sent to identities by $\kappa_n$. It thus suffices to show that $\kappa_n$ is a localization.

  To see this, consider the slice $\sd(\Delta^n)_{\langle i,j\rangle/}$ and the weak fiber $\sd(\Delta^{n})_{\langle i,j\rangle}$. The latter has an initial element and hence is contractible. By \cite[Proposition 5.0.7]{abellan_garcia_theorem_2022}, it suffices to show that for every $U\in \sd(\Delta^{n})_{\langle i,j\rangle/}$ the category $(\sd(\Delta^{n})_{\langle i,j\rangle})_{/U}$ is contractible. But this has a initial element given by $\langle i,j \rangle$, completing the proof.
\end{proof}

\begin{lem}\label{lem:characterization_Cartesian_in_Span1}
  Let $\cC$ be an $(\infty, 1)$-category with finite limits. A functor $f:\sd(1^\sharp(\Delta^{n}))\rightarrow \cC$ is Cartesian if and only if it sends endpoint preserving maps to equivalences and it satisfies all the even $1$-Segal conditions.
\end{lem}

\begin{proof}
  Assume first that $f$ sends endpoint-preserving maps to equivalences and satisfies the even $1$-Segal conditions. For any other Segal condition, the corresponding Segal cube is degenerate, since one of the prongs of the cube must be the image of an endpoint-preserving map.

  On the other hand, if $f$ is Cartesian, it satisfies the even $1$-Segal conditions by assumption. It follows from \zcref{lem:k-vanishing} that every endpoint preserving map is sent to an equivalence.
\end{proof}

\begin{prop}
  Let $\cC$ be an $(\infty,1)$-category with finite limits. There exists an equivalence of $(\infty,1)$-categories
  \[
    \Span_{1}(\cC)\simeq \Span_{1}^{B}(\cC).
  \]
\end{prop}

\begin{proof}
  The claim follows from the natural equivalence of mapping spaces
  \[
    \Map([n],\Span_{1}(\cC)) \simeq \Map^{\cart}(\sd(1^\sharp(\Delta^{n})),\cC) \simeq \Map^{\adq}(\Tw([n]),\cC)\simeq \Map([n],\Span^{B}_{1}(\cC)).
  \]
  Here, the first equivalence is definitional (but is also a special case of \zcref{cor:trunc_map_equiv}), the second is \zcref{lem:twisted}, \zcref{lem:characterization_Cartesian_in_Span1}, and \zcref{rmk:kappa_adequate}, and the third is definitional.
\end{proof}

The relation between $\sd(\Delta^n)$ and $\Tw([n])$ also allows us to provide a universal property of the $2$-category $\Span_{1\frac{1}{2}}(\scr{C})$.

\begin{cor}\label{cor:TwSpan112}
  Let $\cC$ be a category with finite limits. There is a natural equivalence
  \[
    \Map(\Delta^n_t,\Span_{1\frac{1}{2}}(\cC)) \simeq \Map^{\adq}(\Tw(\Delta^n_t),\cC),
  \]
  where the right-hand side denotes the space of functors such that the map $\Tw([2])\to \cC$ corresponding to a map $\Delta^2_+\to \Delta^n$ is adequate.
\end{cor}

\begin{proof}
  The proof of  \zcref{lem:characterization_Cartesian_in_Span1} shows that a functor $f:\sd(\Delta^n)\to \scr{C}$ satisfies the odd $1$-Segal conditions if and only if it is a localization at the endpoint-preserving maps. The corollary follows from \zcref{rmk:kappa_adequate}.
\end{proof}

\subsection{Comparison with other models}\label{subsec:comparison_models}

We now compare our construction of $\Span_{n\frac{1}{2}}(\scr{C})$ to that of \cite{stefanich_higher_2020}. There, Stefanich constructs an $n$-category $n\Corr(\scr{C})$, which will be shown to correspond to our category $\Span_{(n-1)\frac{1}{2}}(\scr{C})$, and an inclusion
\[
  \begin{tikzcd}
    \iota_{\scr{C}}^n: &[-3em] \scr{C}\arrow[r] & n\Corr(\scr{C}).
  \end{tikzcd}
\]
He then proves a universal property of $n\Corr(\scr{C})$.

\begin{thm}[{\cite[Theorem 1.2.1]{stefanich_higher_2020}}] \label{thm:stefanich_UP}
  Let $\scr{C}$ be a 1-category admitting pullbacks and $\scr{D}$ an $n$-category. Restricting along $\iota_{\scr{C}}^n$ induces an equivalence between $\Map_{\Cat_n}(n\Corr(\scr{C}),\scr{D})$ and the subspace of $\Map_{\Cat_n}(\scr{C},\scr{D})$ consisting of functors which satisfy the $(n-1)$-fold left Beck--Chevalley condition.
\end{thm}

We will use this universal property to prove the following.

\begin{thm}\label{thm:compare_models}
  Let $\scr{C}$ be a $1$-category with finite limits. Then, there is an essentially unique equivalence $n\Corr(\scr{C})\simeq \Span_{(n-1)\frac{1}{2}}(\scr{C})$ such that the diagram
  \[
    \begin{tikzcd}
      n\Corr(\scr{C}) \arrow[r,"\simeq"]
      & \Span_{(n-1)\frac{1}{2}}(\scr{C}) \\
      \scr{C}\arrow[u,"\iota_{\scr{C}}^n"]\arrow[r,"\simeq"'] & \Span_{\frac{1}{2}}(\scr{C})\arrow[u,hookrightarrow]
    \end{tikzcd}
  \]
  commutes.
\end{thm}

Our proof will consist of three key steps. Firstly, in \zcref{prop:base_comparison} we will construct an explicit equivalence $2\Corr(\scr{C})\simeq \Span_{1\frac{1}{2}}(\scr{C})$, and secondly, we will verify in \zcref{prop:sat_BC} that
\[
  \begin{tikzcd}
    \scr{C}\arrow[r,"\simeq"] & \Span_{\frac{1}{2}}(\scr{C})\arrow[r,hookrightarrow] & \Span_{(n-1)\frac{1}{2}}(\scr{C})
  \end{tikzcd}
\]
satisfies the $(n-1)$-fold left Beck--Chevalley condition. Finally, in \ref{prop:functor_is_equiv}, we verify inductively that each of these functors is an equivalence.

For the first statement, we briefly recall the definition of $2\Corr(\scr{C})$.

\begin{defn}[{\cite[Notation 3.1.9]{stefanich_higher_2020}}]
  For $\scr{C}\in \Cat_1^{\lex}$, the $2$-category $2\Corr(\scr{C})$ is defined to be the Segal category
  \[
    \begin{tikzcd}[row sep=0em]
      2\Corr(\scr{C}): &[-3em] \Delta^\op \arrow[r] & \Cat_1 \\
      & {[n]}\arrow[r,mapsto] & \Map(\{0,\ldots,n\},\scr{C})\times_{\Fun(\{0,\ldots,n\},\scr{C})} \Fun^{\adq}(\Tw([n]),\scr{C}).
    \end{tikzcd}
  \]
  Which is shown to lie in the essential image of the inclusion
  \[
    \begin{tikzcd}
      \Cat_2\arrow[r] & \sf{Seg}(\Cat_1)
    \end{tikzcd}
  \]
  of $2$-categories into Segal categories in \cite[Cor 3.1.12]{stefanich_higher_2020}.
\end{defn}

To make the comparison, we provide the following explicit description of the inclusion $\Cat_2\to \Seg(\Cat_{1})$.

\begin{defn}
  Let $\scr{D}$ be an $n$-category. The \emph{Segal nerve} of $\scr{D}$ is the Segal object in $\Cat_{n-1}$ which assigns to $[n]$ the $(n-1)$-category
  \[
    \Map(\{0,\ldots,n\},\scr{C})\times_{\Fun^{\oplax}(\{0,\ldots,n\},\scr{C})^{\leq n-1}}\Fun^{\oplax}(1^\sharp(\Delta^n), \scr{C})^{\leq{n-1}}.
  \]
  We write
  \[
    \begin{tikzcd}
      N^{\seg}_n: &[-3em] \Cat_n \arrow[r] & \sf{Seg}(\Cat_{n-1})
    \end{tikzcd}
  \]
  for the functor which sends each $n$-category to its Segal nerve.
\end{defn}

\begin{thm}[{\cite[Theorem 2.65]{abellan_comparing_2023}}]
  The Segal nerve
  \[
    \begin{tikzcd}
      N^{\seg}_2: &[-3em] \Cat_2 \arrow[r] & \sf{Seg}(\Cat_{1})
    \end{tikzcd}
  \]
  is fully faithful and its essential image consists of the complete Segal objects in $\Cat_1$.
\end{thm}

\begin{rmk}
  While we expect that the Segal nerve provides the equivalence between $\Cat_n$ and complete Segal objects in $\Cat_{n-1}$, to our knowledge this statement as not yet appeared in the literature in full generality. We do not need this more general statement in the paper.
\end{rmk}

\begin{ntt}
  We note that by passing through the defining adjunctions, for any $k$-category $\scr{D}$ we have
  \[
    \Map(\Delta^m_t,N^{\seg}_k(\scr{D})_n)\simeq \Map\left(\Delta^m_{\sharp}\otimes \Ob(\Delta^n)\coprod_{\Delta^m_{t}\otimes \Ob(\Delta^n)}(k-1)^\sharp(\Delta^m)\otimes 1^\sharp(\Delta^n),\scr{D}\right).
  \]
  We therefore set the notational convention
  \[
    \Delta^m_t\odot^k 1^\sharp(\Delta^n)\coloneqq \Delta^m_{\sharp}\otimes \Ob(\Delta^n)\coprod_{\Delta^m_{t}\otimes \Ob(\Delta^n)}(k-1)^\sharp(\Delta^m)\otimes 1^\sharp(\Delta^n)
  \]
\end{ntt}

\begin{prop}\label{prop:base_comparison}
  For a $1$-category $\scr{C}$ with finite limits, there is an equivalence \[
    N^{\seg}_2(\Span_{1\frac{1}{2}}(\scr{C}))\simeq 2\Corr(\scr{C}).
  \]
\end{prop}

\begin{proof}
  Tracing through the definitions, it suffices to show that for any decorated simplex $\Delta^m_t$ the natural maps
  \[
    \begin{tikzcd}
      (\fv,\kappa_{n}):\sd(\Delta^m)\times \sd(\Delta^n) \arrow[r] &\Delta^m\times \Tw([n])
    \end{tikzcd}
  \]
  induce equivalences between
  \begin{enumerate}
    \item Functors out of the left hand side to $\scr{C}$ which send every morphism in $\sd(\Delta^m)\times\Ob(\Delta^n)$ to an equivalence, are Cartesian with respect to the 1-truncated marking on $\Delta^{n}$, and respect the one-sided $1$-truncated mixed Segal conditions.
    \item Functors out of the right hand side to $\scr{C}$ which send every morphism in $\Delta^m\times \Ob(\Delta^n)$ to equivalences and which send $\{i\}\times \Tw([n])$ to a functor of adequate triples.
  \end{enumerate}
  To prove this statement, it suffices by \zcref{lem:characterization_Cartesian_in_Span1} and \zcref{lem:sdEx} to show that a functor satisfying the first two properties of condition (i) also satisfies the one-sided $1$-truncated mixed Segal conditions if and only if, for every $U\in \sd(\Delta^n)$, it maps minimum-preserving morphisms in $\sd(\Delta^{m})\times \{U\}$ to equivalences. In particular, it suffices to show this for all minimum-preserving maps in $\sd(\Delta^{m})$ between subsets that differ by one element.

  Observe that every minimum-preserving map $V^\prime\times U \subset V\times U$ with $\#(V^\prime)=\#(V)-1$ arises as the prong of a one-sided $1$-truncated mixed Segal cube for $U\times V$ and vice versa. We argue by induction on the cardinality of $V$ and $U$ with the base case displayed in \zcref{fig:1-Segal}. For the general case, it follows inductively that for every minimum preserving map as above, the cube is degenerate along the direction of this map. Hence, the minimum preserving map is an equivalence if and only if the cube is a limit cube.

  \begin{figure}[htb]
    \begin{tikzpicture}
      \begin{scope}
        \drawGrid{2}{2}
        \begin{scope}[on background layer]
          \drawLine{3pt}{red}  (00.center) --  (10.center) -- (11.center);
        \end{scope}
      \end{scope}

      \begin{scope}[xshift=3cm]
        \drawGrid{2}{2}
        \begin{scope}[on background layer]
          \drawLine{3pt}{red}  (00.center) --  (01.center) ;
        \end{scope}
      \end{scope}

      \begin{scope}[yshift=-3cm]
        \drawGrid{2}{2}
        \begin{scope}[on background layer]
          \drawLine{3pt}{red} (01.center) --  (11.center);
        \end{scope}
      \end{scope}

      \begin{scope}[yshift=-3cm,xshift=3cm]
        \drawGrid{2}{2}
        \begin{scope}[on background layer]
          \drawLine{3.5pt}{red} (01.center) -- (01.center);
        \end{scope}
      \end{scope}
      \draw[->] (1.5,-0.5) to (2.5,-0.5);
      \draw[->] (1.5,-3.5) to (2.5,-3.5);
      \draw[->] (0.5,-1.5) to (0.5, -2.5);
      \draw[->] (3.5,-1.5) to (3.5,-2.5);
    \end{tikzpicture}
    \caption{The Segal cube diagram associated to the object $(\{0,1\}, \{0,1\})\in \sd(\Delta^1)\times \sd(\Delta^1)$. Since the lower horizontal map is sent to an equivalence, the square is pullback if and only if the top horizontal map is sent to an equivalence.}\label{fig:1-Segal}
  \end{figure}
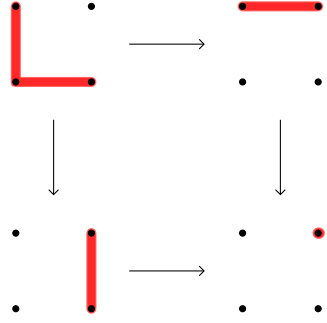
\end{proof}

It is immediate from the construction that the diagram
\[
  \begin{tikzcd}
    2\Corr(\scr{C}) \arrow[r,"\simeq"]
    & \Span_{1\frac{1}{2}}(\scr{C}) \\
    \scr{C}\arrow[u,"\iota_{\scr{C}}^n"]\arrow[r,"\simeq"'] & \Span_{\frac{1}{2}}(\scr{C})\arrow[u,hookrightarrow]
  \end{tikzcd}
\]
commutes, and so by \zcref{thm:stefanich_UP}, the equivalence constructed in \zcref{prop:base_comparison} is essentially unique with this property, establishing the dimension $2$ case of \zcref{thm:compare_models}.

\begin{prop}\label{prop:sat_BC}
  The inclusion
  \[
    \begin{tikzcd}
      j_{\scr{C}}^{n+1}:&[-3em]\scr{C}\arrow[r,"\simeq"] & \Span_{\frac{1}{2}}(\scr{C}) \arrow[r] & \Span_{n\frac{1}{2}}(\scr{C})
    \end{tikzcd}
  \]
  satisfies the $n$-fold left Beck--Chevalley conditions.
\end{prop}

\begin{proof}
  It follows immediately from \zcref{prop:base_comparison} that $j_{\scr{C}}^{n+1}$ satisfies the $1$-fold Beck--Chevalley conditions. By \zcref{thm:stefanich_UP}, we thus obtain a commutative diagram below.
  \[
    \begin{tikzcd}
      2\Corr(\scr{C})\arrow[rr,"F"] & & \Span_{n\frac{1}{2}}(\scr{C})\\
      & \scr{C}\arrow[ul,"\iota_{\scr{C}}^{n+1}"] \arrow[ur,"j_{\scr{C}}^{n+1}"'] &
    \end{tikzcd}
  \]
  Taking hom-objects at objects $z$ and $w$, we obtain a functor
  \[
    \begin{tikzcd}
      \Map_{2\Corr(\scr{C})}(z,w)\arrow[r] & \Span_{(n-1)\frac{1}{2}}(\scr{C}_{/z\times w})
    \end{tikzcd}
  \]
  which by construction commutes with the identification $  \Map_{2\Corr(\scr{C})}(z,w)\simeq \scr{C}_{/z\times w}$ and the functor $j_{\scr{C}_{/z\times w}}^n$. Indeed, this can be seen by chasing through an analog of \zcref{const:mapcatloc-nat_trans} for $\Tw$ and show that it is compatible with the one given in \zcref{const:mapcatloc-nat_trans} itself under the natural map $\sd \to \Tw$. But now, by \cite[Lemma 4.2.9]{stefanich_higher_2020}, $j_{\scr{C}}^{n+1}$ satisfies the $n$-fold left Beck--Chevalley conditions if and only if $j_{\scr{C}_{/z\times w}}$ satisfies the $(n-1)$-fold left Beck--Chevalley conditions. Inductively, we reduce to the case $n=1$, which holds by \zcref{prop:base_comparison}.
\end{proof}

\begin{cor}
  There is an essentially unique functor $F_n: n\Corr(\scr{C})\to \Span_{(n-1)\frac{1}{2}}(\scr{C})$ such that the diagram
  \[
    \begin{tikzcd}
      n\Corr(\scr{C}) \arrow[rr,"F_n"] & & \Span_{(n-1)\frac{1}{2}}(\scr{C}) \\
      & \scr{C}\arrow[ul,"\iota_{\scr{C}}^n"] \arrow[ur,"j_{\scr{C}}^n"'] &
    \end{tikzcd}
  \]
  commutes.
\end{cor}
\begin{proof}
  This follows from \zcref{thm:stefanich_UP,prop:sat_BC}.
\end{proof}

\begin{prop}\label{prop:functor_is_equiv}
  For every $n\geq 2$, the functor $F_n$ is an equivalence.
\end{prop}

\begin{proof}
  For fixed $n$, the commutative diagram
  \[
    \begin{tikzcd}
      n\Corr(\scr{C}) \arrow[rr,"F_n"] & & \Span_{(n-1)\frac{1}{2}}(\scr{C}) \\
      & \scr{C}\arrow[ul,"\iota_{\scr{C}}^n"] \arrow[ur,"j_{\scr{C}}^n"'] &
    \end{tikzcd}
  \]
  tells us that the functor $F_n$ is essentially surjective. Moreover, this commutative diagram can be promoted to a commutative diagram
  \[
    \begin{tikzcd}
      n\Corr(\scr{C}) \arrow[rr,"F_n"] & & \Span_{(n-1)\frac{1}{2}}(\scr{C}) \\
      & 2\Corr(\mathcal{C}) \arrow[ul] \arrow[ur] &
    \end{tikzcd}
  \]

  Passing to hom-objects gives us a commutative diagram
  \[
    \begin{tikzcd}
      (n-1)\Corr(\scr{C}_{/x\times y}) \arrow[rr,"{\Map_{F_n}(x,y)}"] & & \Span_{(n-2)\frac{1}{2}}(\scr{C}_{/x\times y}) \\
      & \scr{C}_{/x\times y} \arrow[ul,"\iota_{\scr{C}}^{n-1}"] \arrow[ur,"j_{\scr{C}}^{n-1}"'] &
    \end{tikzcd}
  \]
  By the essential uniqueness of $F_{n-1}$, it follows that ${\Map_{F_n}(x,y)}\simeq F_{n-1}$, and that $F_n$ is fully faithful (and thus an equivalence) if and only if $F_{n-1}$ is an equivalence.

  Inductively, we can thus reduce to showing that $F_2$ is an equivalence, which follows from \zcref{prop:base_comparison} together with the essential uniqueness of $F_2$.
\end{proof}

\begin{cor} \label{cor:functor_is_equiv_sym_mon}
  For every $n\geq 2$, the functor $F_n$ is an equivalence of symmetric monoidal $n$-categories, compatible with the functor from $\mathcal{C}$.
\end{cor}
\begin{proof}
  This follows directly from the fact that both $n\Corr$ and $\Span_{(n-1)\frac{1}{2}}$ are limit preserving functors $\Cat^\lex \to \Cat_\infty$.

  Alternatively, one can also run the argument above, but instead, using~\cite[Corollary 1.2.2]{stefanich_higher_2020}, which is a symmetric monoidal variant of \zcref{thm:stefanich_UP}.
\end{proof}

\begin{rmk} \label{rmk:span_infty_colim}
  Combining \zcref{prop:functor_is_equiv,lem:span_infty_colimit_span_lower}, we see that
  \[
    \Span_\infty(\mathcal{C}) \simeq \colim_{n\in \mathbb{N}} n\Corr(\mathcal{C}).
  \]

  The comparison results above automatically imply the agreement between our model for $\Span_n(\mathcal{C})$ and Haugseng's~\cite{haugseng_iterated_2018} since both are the underlying $(\infty,n)$-category of $\Span_{n\frac{1}{2}}(\mathcal{C})$, denoted as $\Span_n^+(\mathcal{C})$ over there. In particular, \zcref{lem:span_infty_colimit_span_lower} also implies that
  \[
    \Span_\infty(\mathcal{C}) \simeq \colim_{n\in \mathbb{N}} \Span_n(\mathcal{C}),
  \]
  where $\Span_n(\mathcal{C})$ refers to Haugseng's model here.

  Thus, the main results of this paper, \zcref{intro:thm:SpanUP_I_internal,intro:thm:functor_UP_inftycat}, characterize \emph{mapping into} these colimits.
\end{rmk}

\appendix
\renewcommand{\thesubsubsection}{\Alph{section}.\arabic{subsubsection}}

\section{Orientals, Gray tensors, and truncations}\label{app:loc}

This appendix is independent of the rest of the paper, and is used solely to establish the localization result \zcref{thm:the_localization}. The goal is to relate strict $2$-categorical Gray tensor products of strictly $2$-truncated orientals to $(\infty,2)$-categorical Gray tensor products of $(\infty,2)$-truncated orientals. In this appendix, we will work mostly model-independently, specializing to a chosen model only at the very last step.

\begin{defn}
  The \emph{$n$\textsuperscript{th} oriental} $\mathbb{O}^n$ is the strict $\omega$-category freely generated by the simplices of $\Delta^n$. See \cite{street_algebra_1987} for an explicit characterization.
\end{defn}

\begin{defn}
  We call a strict $\omega$-category $\scr{C}$ \emph{gaunt} if any invertible $k$-morphism is an identity.
\end{defn}

\begin{rmk}
  The strict $\omega$-category $\mathbb{O}^n$ is gaunt. It is, in fact a so-called \emph{strong Steiner complex}.
\end{rmk}

\begin{ntt}
  For $n\in \NN \cup \{\infty\}$, we write $\Cat_n^g$ for the category of strict, gaunt, $n$-categories. We denote by
  \[
    \begin{tikzcd}
      \tau_n: &[-3em]\Cat_\infty\arrow[r] & \Cat_n
    \end{tikzcd}
  \]
  and
  \[
    \begin{tikzcd}
      \tau_n^g: &[-3em]\Cat_\infty \arrow[r] & \Cat_n^g
    \end{tikzcd}
  \]
  for the left adjoints to the inclusions.
\end{ntt}

\begin{rmk}
  Note that if $\tau_n(\scr{C})$ is gaunt, then there is a canonical equivalence $\tau_n(\scr{C})\simeq \tau_n^g(\scr{C})$ provided by the universal property.
\end{rmk}

\begin{ntt}
  We denote by $\otimes_2$ the Gray tensor product of $(\infty,2)$-categories and by $\otimes$ the Gray tensor product of $(\infty,\infty)$-categories. We will use the comparisons described in the introduction of \cite{loubaton_squares_2025} to make use of multiple models for $\otimes_2$. We will write $\otimes_2^s$ for the Gray tensor product of strict $2$-categories.
\end{ntt}

\begin{lem}\label{lem:nat_equiv_weak_2_trunc}
  There are natural equivalences of $(\infty,2)$-categories
  \[
    \tau_2(\mathbb{O}^n)\otimes_2\tau_2(\mathbb{O}^m)\simeq \tau_2(\mathbb{O}^n\otimes \mathbb{O}^m).
  \]
\end{lem}

\begin{proof}
  By Theorem 3.14 of \cite{campion_gray_2023}, the localization functor
  \[
    \begin{tikzcd}
      \Cat_\omega^f\arrow[r] & \Cat_2
    \end{tikzcd}
  \]
  is strong monoidal, and by the construction of the gray tensor in  \emph{op. cit.}, the inclusion of strong Steiner complexes into $\Cat_\omega^f$ is strong monoidal. The composite of these two functors agrees with the composite of $\tau_2$ with the inclusion of strong Steiner complexes into $\Cat_\infty$, and so the latter composite is also strong monoidal. Since $\mathbb{O}^n$ is a strong Steiner complex, the lemma follows.
\end{proof}

\begin{lem}\label{lem:graygaunt}
  The $(\infty,2)$-categories $\tau_2((\mathbb{O}^1)^{\otimes n})$, $\tau_2(\mathbb{O}^n)$, and $\tau_2(\mathbb{O}^n\otimes \mathbb{O}^m)$ are all gaunt.
\end{lem}

\begin{proof}
  By \cite[Corollary A.3.6 (1)]{masuda_algebra_2024}, $\mathbb{O}^n$ is a retract of $(\mathbb{O}^1)^{\otimes n}$, and it follows that $\mathbb{O}^n\otimes \mathbb{O}^m$ is a retract of  $(\mathbb{O}^1)^{\otimes (n+m)}$. By the construction of $\otimes_2$ in \cite{maehara_gray_2021}, the Gray tensor $\otimes_2$ agrees with the strict gray tensor on $\Theta_2$, and in particular on copies of $\tau_2(\mathbb{O}^1)$, so it follows that $\tau_2((\mathbb{O}^1)^{\otimes n})$ are gaunt. Since gaunt $n$-categories are stable under retract, the lemma follows.
\end{proof}

\begin{prop}\label{prop:2trunc_gray_nat_equiv}
  There are natural equivalences of $(\infty,2)$-categories
  \[
    \tau_2(\mathbb{O}^n)\otimes_2\tau_2(\mathbb{O}^m)\simeq \tau_2^g(\mathbb{O}^n)\otimes_2^s\tau_2^g(\mathbb{O}^m).
  \]
\end{prop}

\begin{proof}
  By \zcref{lem:nat_equiv_weak_2_trunc}, we have a natural equivalences
  \[
    \tau_2(\mathbb{O}^n)\otimes_2\tau_2(\mathbb{O}^m)\simeq \tau_2(\mathbb{O}^n\otimes \mathbb{O}^m).
  \]
  Since the $(\infty,2)$-category on the right is gaunt and the localization $\tau_2^g$ is monoidal, we have further natural equivalences
  \[
    \tau_2(\mathbb{O}^n\otimes \mathbb{O}^m)\simeq \tau_2^g(\mathbb{O}^n\otimes \mathbb{O}^m)\simeq \tau_2^g(\mathbb{O}^n)\otimes_2^s\tau_2^g(\mathbb{O}^m)
  \]
  completing the proof.
\end{proof}

\begin{cor}\label{cor:gray_prod_nat_equiv_scaled}
  Considering $\mathfrak{C}^{\sc}[\Delta^n]$ as a strict $2$-category, there are equivalences
  \[
    \begin{tikzcd}
      \mathfrak{C}^{\sc}[\Delta^n_\flat\otimes_2\Delta^m_\flat]\arrow[r] & \mathfrak{C}^{\sc}[\Delta^n]\otimes_2^s \mathfrak{C}^{\sc}[\Delta^m_\flat]
    \end{tikzcd}
  \]
  natural in $\Delta\times \Delta$.
\end{cor}

\begin{proof}
  We work in the model of scaled simplicial sets. First note that $\mathfrak{C}^{\sc}[\Delta^n]$ can be identified with the gaunt $2$-truncation of the $n$\textsuperscript{th} oriental and that $\mathfrak{C}^{\sc}[\Delta^n]\otimes_2^s \mathfrak{C}^{\sc}[\Delta^m_\flat]$ can be viewed via the hom-wise nerve as a fibrant simplicially enriched category.

  Since the target is gaunt we can take the naturality of the equivalence of \zcref{prop:2trunc_gray_nat_equiv} to be strict naturality. By \cite[Theorem 2.1]{maehara_orientals_2023}, we see that the minimally scaled simplicial set $\Delta^n_\flat$ is equivalent to $\tau_2(\mathbb{O}^n)$, and by the functoriality of fibrant replacement, this equivalence is natural in $\Delta$. We thus obtain a natural weak equivalence of scaled simplicial sets
  \[
    \begin{tikzcd}
      \Delta^n_\flat\otimes_2\Delta^m_\flat \arrow[r]  & N^{\sc}(\mathfrak{C}^{\sc}[\Delta^n]\otimes_2^s \mathfrak{C}^{\sc}[\Delta^m_\flat]).
    \end{tikzcd}
  \]
  Since $\mathfrak{C}^{\sc}\dashv N^{\sc}$ is a Quillen equivalence, the adjoint maps give the desired natural equivalence.
\end{proof}

\section{Homotopy colimits over complicial categories of elements}

In this appendix, we prove two technical model-categorical results we rely on for the model-categorical versions of our main theorems. Both of these results follow directly from standard techniques in the theory of Reedy categories. For the sake of brevity, we will not recapitulate the theory of (elegant) Reedy categories here, and we direct the reader to \cite[Chapter 15]{hirschhorn_model_2009} for background. Though we use results from a variety of papers on Reedy category, we will follow the terminology and notation of \cite{hirschhorn_model_2009} for consistence.

\begin{defn}
  The category $\tDelta$ admits the structure of an elegant Reedy category by appendix C of \cite{ozornova_model_2020}. This Reedy structure on $\tDelta$ is defined as follows.
  \begin{enumerate}
    \item The direct subcategory is generated by the injective maps $\Delta^n_\flat\to \Delta^m_\flat$, together with the map $\Delta^n_\flat\to \Delta^n_+$.
    \item The inverse subcategory is generated by the surjective maps $\Delta^n_\flat\to \Delta^m_\flat$ and the face maps $\Delta^n_+\to \Delta^{n-1}_\flat$.
  \end{enumerate}
  The degree of the object $\Delta^n_\flat$ is $2n-1$, and the degree of the object $\Delta^n_+$ is $2n$.
\end{defn}

Since $\tDelta$ is an elegant Reedy category, for any $(K,tK)\in \Set_\Delta^\dec$, the category $\tDelta_{/(K,tK)}$ inherits the structure of a Reedy category by \cite[Proposition 1.1.2.2]{loubaton_complicial_2024}.

\begin{lem}
  Let $(K,tK)$ be a decorated simplicial set. Then the Reedy category $\tDelta_{/(K,tK)}$ has
  fibrant constants.
\end{lem}

\begin{proof}
  Per \cite[Proposition 15.10.2]{hirschhorn_model_2009}, it suffices for us to see that the matching category for some $\sigma:\Delta^n_t\to (K,tK)$ is either empty or connected. If $\sigma$ represents a non-degenerate simplex of $K$, then the matching category is clearly empty.

  If, on the other hand, $\sigma$ is degenerate, there is a unique non-degenerate $k$-simplex $\gamma$ of $K$ such that $\sigma$ is degenerate on $\gamma$. Since the non-identity morphisms of the inverse subcategory can only have targets of the form $\Delta^k_\flat$, it follows that $\gamma:\Delta^k_\flat\to (K,tK)$ is a terminal object in the matching category, completing the proof.
\end{proof}

\begin{lem}\label{lem:Reedy_cof_diag1}
  The diagram
  \[
    \begin{tikzcd}[row sep=0em]
      (\Delta_m)_{/(K,tK)} \arrow[r] & \msSet\\
      (\sigma:\Delta^n_t\to (K,tK)) \arrow[r,mapsto] & \Delta^n_t
    \end{tikzcd}
  \]
  is Reedy cofibrant.
\end{lem}

\begin{proof}
  This is completely parallel to the case of usual simplex category. In every case, the latching object associated to $\sigma:\Delta^n_{t}\to \scr{C}$ is the unmarked boundary of the simplex $\sigma$, and so the latching map is clearly a cofibration.
\end{proof}

\begin{cor}\label{cor:K_t_hocolim}
  For any decorated simplicial set $K_t$,
  \[
    K_t\simeq \hocolim_{\tDelta_{/K_t}} \Delta^n_{t}.
  \]
\end{cor}

\begin{proof}
  Since $K_t$ is the strict colimit of the diagram from \zcref{lem:Reedy_cof_diag1}, and this diagram is Reedy cofibrant, the result follows from \cite[Theorem 15.10.8]{hirschhorn_model_2009}.
\end{proof}

We conclude with a similar result for the functor $\sd$ of \zcref{defn:sd_SD_LKE}.

\begin{lem}
  The composite
  \[
    \begin{tikzcd}
      \tDelta_{/K_t} \arrow[r] & \Set_\Delta^{\dec} \arrow[r,"\sd"] & \Set_\Delta
    \end{tikzcd}
  \]
  is Reedy cofibrant with respect to the Joyal model structure.
\end{lem}

\begin{proof}
  Because $\sd$ preserves colimits, it preserves latching objects. Since $\sd$ additionally preserves cofibrations, it preserves Reedy cofibrancy.
\end{proof}

\begin{cor}\label{cor:sd_hocolim}
  For any decorated simplicial set $K_t$, there is an equivalence in the Joyal model structure
  \[
    \sd(K_t)\simeq \hocolim_{\tDelta_{/K_t}} \sd(\Delta^n_{t}).
  \]
\end{cor}

\printbibliography

\end{document}